\documentclass[a4paper,11pt]{article} 

\input{preamble}
\input{macros}

\begin{document}

\title{Dimension theory of group actions by circle diffeomorphisms II: Minimal sets}

\author{Weikun He, Yuxiang Jiao and Disheng Xu,\\ with an appendix joint with Wenyuan Yang}

\date{}

\maketitle

\begin{abstract}

We establish a dimension theory for smooth group actions on the circle.
For finitely generated groups of real-analytic circle diffeomorphisms preserving a Cantor minimal set, we introduce a dynamically defined critical exponent and prove the following properties of that set:
\begin{enumerate}[\indent (1)]
	\item The box dimension exists and coincides with its Hausdorff dimension;
	\item The dimension is given by a formula involving the dynamical critical exponent;
	\item The dimension lies strictly between $0$ and $1$;
	\item A parabolic fixed point in this Cantor minimal set yields a stronger lower bound for its dimension.
\end{enumerate}
These results substantially generalize the classical dimension theory of limit sets of Fuchsian groups to a broader smooth setting.
Moreover, item (3) strengthens a recent breakthrough of Deroin--Kleptsyn--Navas \cite{DKN18}, who showed that such sets have zero Lebesgue measure.

A key novelty of this work is the use of the dynamical critical exponent to link the group action with stationary and conformal measures.
This provides a framework combining methods from nonuniform hyperbolic dynamics and Patterson--Sullivan theory.

\end{abstract}

\tableofcontents
\section{Introduction}
\subsection{Overview and the main theorem}

The dimensions of fractal sets and associated invariant measures arising from dynamical systems have been extensively studied, particularly in relation to dynamical quantities such as entropy, Lyapunov exponents, critical exponents, and the uniformity of hyperbolicity.
This theme has been explored across several areas of mathematics, including the study of the dimensions of Julia sets in holomorphic dynamics \cite{AL1, AL2, Shi, Buff, Mc87, EL84, ADU, GS09}; the dimensions of attractors and invariant measures of smooth dynamical systems, and their relations with entropy and Lyapunov exponents \cite{Led81, Led84, LY1, LY2, Y82, BaPeSc}; the analysis of self-affine sets and measures arising from iterated function systems (IFS) in fractal geometry \cite{Hoch14, FengHu, Var19, Hut, Fal88, Fal92, BHR, MS, Rap24}; the study of stationary measures for random walks \cite{HS17, LedLe, Rap, Led83}; and the dimensions of limit sets of hyperbolic group actions \cite{Thu78, Sul, Bow79, BJ97, Pat76, Beardon}.

A fundamental example of a fractal is the limit set of a Fuchsian group.
Recall that a Fuchsian group $\Gamma$ is a discrete subgroup of the orientation-preserving isometry group $\Isom_+(\DD)\cong\PSL_2(\RR)$ of the Poincar\'e disk $\DD.$
The \textit{limit set} of $\Gamma$ is the collection of accumulation points of an orbit $\Gamma o$ in $\ol\DD=\DD\cup\partial\DD $.
This set is independent of the choice of the base point $o \in \DD.$
We denote it by $\Lambda_\Gamma.$
Because of the discreteness of $\Gamma,$ $\Lambda_\Gamma$ is contained in the boundary $\partial \DD\cong\SS^1.$
An elementary topological classification shows that exactly one of the following holds:
\begin{enumerate}[\indent (1)]
	\item $\Lambda_\Gamma$ consists of $0$, $1$ or $2$ points. 
    This happens if and only if $\Gamma$ is elementary, that is, $\Gamma$ is finite or virtually-$\ZZ$.
	\item $\Lambda_\Gamma$ is the entire circle. 
    In this case, $\Gamma$ is called a \textit{Fuchsian group of the first kind}. 
	\item $\Lambda_\Gamma$ is a nowhere-dense perfect set; equivalently, $\Lambda_\Gamma$ is homeomorphic to a Cantor set. 
    In this case, $\Gamma$ is called a \textit{Fuchsian group of the second kind}.
\end{enumerate}
Given a Fuchsian group $\Gamma$ of the second kind, it is of fundamental interest to determine the \textit{Hausdorff dimension} $\dimH\Lambda_\Gamma$ and the \textit{box dimension} $\dim_{\mr B}\Lambda_\Gamma.$

For a finitely generated, nonelementary Fuchsian group of the second kind, the dimension theory of its limit set $\Lambda_\Gamma$ has been developed in considerable depth through the work of Beardon \cite{Bear71, Beardon}, Patterson \cite{Pat, Pat76}, Sullivan \cite{Sul} and Stratmann-Urba\'nski \cite{SU96}. 
Related questions in higher dimensions have also been studied; see, for instance, Bishop-Jones \cite{BJ97} for a thorough treatment of limit sets of Kleinian groups.

One of the key ingredients in this theory is the relation between the Hausdorff dimension of $\Lambda_\Gamma$ and another invariant of $\Gamma$, the \textit{critical exponent} $\delta(\Gamma).$
Here $\delta(\Gamma)$ is defined as the abscissa of convergence of the classical Poincar\'e series associated with $\Gamma,$ and it admits the following equivalent descriptions:
\[\delta(\Gamma)=\limsup_{n\to+\infty}\frac{1}{n}\log\# \lb{g\in\Gamma: d_{\mr{hyp}}(o,go)\leqslant n}=\limsup_{n\to+\infty}\frac{1}{n}\log\# \lb{g\in\Gamma: \|g \|_{\mr{op}}^2\leqslant e^n}, \]
where $d_{\mr{hyp}}$ denotes the hyperbolic distance on $\DD,$ $o\in\DD$ is a fixed base point, and $\|g\|_{\mr{op}}$ is the operator norm of $g$ as a matrix.
The following summarizes classical results on the dimension of limit sets of such groups. 
\begin{theorem}(Beardon, Patterson, Sullivan, Stratmann--Urba\'nski)\label{thm: Fuchsian}
	Let $\Gamma$ be a finitely generated, nonelementary Fuchsian group of the second kind, and let $\Lambda_\Gamma$ be its limit set.
	Then the following hold:
	\begin{enumerate}[\indent (1)]
		\item The box dimension of $\Lambda_\Gamma$ exists and coincides with its Hausdorff dimension.
		\item Both dimensions coincide with the critical exponent of $\Gamma$: $\dim_{\mr B}\Lambda_\Gamma= \dim_{\mr H}\Lambda_\Gamma=\delta(\Gamma).$
		\item The critical exponent satisfies $0<\delta(\Gamma)<1$; hence $0<\dim_{\mathrm H}\Lambda_\Gamma=\dim_{\mathrm B}\Lambda_\Gamma<1.$
		\item If $\Gamma$ contains a parabolic element, then $\delta(\Gamma)>1/2$ and hence $\dimH\Lambda_\Gamma>1/2.$
	\end{enumerate}
\end{theorem}

We now adopt a different perspective on the limit sets of Fuchsian groups.
Instead of considering the action of $\Gamma$ on the Poincar\'e disk $\DD$, we focus on its action on the boundary $\partial\DD = \SS^1$. 
For a nonelementary Fuchsian group $\Gamma,$ the limit set $\Lambda_\Gamma$ is the unique $\Gamma$-invariant minimal subset of $\SS^1.$
This point of view naturally leads to the study of fractal sets arising from group actions on the circle by homeomorphisms.

Let $G$ be a subgroup of $\Homeo(\SS^1)$.
By the classical classification of $G$-minimal sets on the circle (see \cite[Theorem 2.1.1]{Na11}), exactly one of the following holds:
\begin{enumerate}[\indent (1)]
	\item $G$ has a finite orbit.
	\item $G$ acts minimally on the circle.
	\item $G$ has a unique $G$-invariant minimal set on $\SS^1$. It is homeomorphic to a Cantor set.
\end{enumerate}
In the latter two cases, the action of $G$ has a unique minimal set on the circle, denoted by $\Lambda$.
When $\Lambda$ is a Cantor set, we call it an \emph{exceptional minimal set}.
Such actions naturally generalize the boundary actions of Fuchsian groups of the second kind to nonhomogeneous settings.
Our goal is to investigate the dimensional properties of exceptional minimal sets, in particular their \emph{Hausdorff dimension} $\dim_{\mathrm H}\Lambda$ and \emph{box dimension} $\dim_{\mathrm B}\Lambda$.

Before stating the main theorem, we introduce the \textit{dynamical critical exponent}.
It generalizes the classical critical exponent of Fuchsian groups and will play a central role in the statements and proofs throughout this paper.
\begin{definition}\label{def: C1 critical exponent}
Let $G$ be a subgroup of $\Diff_+(\SS^1)$ without finite orbits, and let $\Lambda$ denote its unique minimal set.
The \textit{dynamical critical exponent} of $G$ is defined as 
\begin{equation}\label{eqn: DCE def}
\delta(G) = \lim_{\ve \to 0^+} \limsup_{n\to +\infty} \frac{1}{n}\log\#\{\,g \in  G : \exists x \in \Lambda,\, g'|_{B(x,\ve)} \geq e^{-n} \,\}.
\end{equation}
\end{definition}

It is straightforward to verify that this quantity agrees with the classical definition of the critical exponent for nonelementary Fuchsian groups.
Roughly speaking, the induced action of a hyperbolic element $g \in \Isom_+(\DD) \cong \PSL_2(\RR)$ has derivative comparable to $\|g\|_{\mathrm{op}}^{-2}$ at most points of the circle.
We refer to $\delta(\cdot)$ as the \textit{dynamical} critical exponent because, unlike the classical critical exponent, which is determined by geometric information on hyperbolic spaces, its definition depends solely on the local contraction rates of the group elements acting on the circle.
This viewpoint makes the notion more flexible and easier to generalize to other group actions on manifolds.

We are now ready to state our main theorem.
Let $\Diff_+^\omega(\SS^1)$ denote the group of orientation-preserving real-analytic diffeomorphisms of the circle.
Recall that a fixed point $x$ of an element $f \in \Diff_+^\omega(\SS^1)$ is said to be of \textit{multiplicity} $k+1$ if $x$ is a zero of order $k+1$ of the function $\vp(z)=f(z)-z$, where we identify $\SS^1$ with $\RR/\ZZ$. 
We refer to a fixed point as \textit{parabolic} if it has multiplicity at least 2.
\begin{maintheorem}\label{mainthm}
	Let $G$ be a finitely generated subgroup of $\Diff_+^\omega(\SS^1)$ with an exceptional minimal set $\Lambda$.
	Then the following statements hold:
	\begin{enumerate}[\indent (1)]
		\item\label{mainthm: 1} The box dimension of $\Lambda$ exists and coincides with its Hausdorff dimension.
		\item\label{mainthm: 2} Both dimensions of $\Lambda$ are equal to the dynamical critical exponent:
		\[\dim_{\mr B}\Lambda= \dim_{\mr H}\Lambda=\delta(G).\]
		\item\label{mainthm: 3} These dimensions satisfy $0<\dim_{\mr H}\Lambda= \dim_{\mr B}\Lambda<1.$
		\item\label{mainthm: 4} If $G$ contains an element $f$ with a parabolic fixed point $x\in\Lambda$ of multiplicity $(k+1)$, then $\delta(G)>k/(k+1)$, and hence $ \dim_{\mr H}\Lambda >k/(k+1).$ 
		Consequently, the multiplicities of fixed points of nontrivial elements of $G$ are uniformly bounded.
	\end{enumerate}
\end{maintheorem}
Note that $\Isom_+(\DD)$ embeds naturally into $\Diff^\omega_+(\SS^1).$
Hence our main theorem extends Theorem \ref{thm: Fuchsian} to a broader smooth setting. 
Indeed, every parabolic element of $\Isom_+(\DD)$ has a fixed point of multiplicity $2$ on the circle, which explains the appearance of the bound $1/2$ in item \eqref{mainthm: 4} of Theorem~\ref{thm: Fuchsian}.

Part \eqref{mainthm: 3} is also closely related to an important problem in the theory of group actions on the circle and foliation theory.
Motivated by the study of codimension-one foliations, Sullivan and Ghys asked whether an exceptional minimal set of a group of $C^2$-diffeomorphisms must have zero Lebesgue measure (see \cite{NavICM, Hector, Hur02}).
A strong piece of evidence for this conjecture was provided by Deroin, Kleptsyn, and Navas, who showed in \cite{DKN18} that any exceptional minimal set preserved by a finitely generated subgroup of $\Diff^\omega(\SS^1)$ has zero Lebesgue measure ---  a major advance toward resolving the conjecture. 
Part \eqref{mainthm: 3} strengthens their result by showing that, under the same assumptions, the box dimension satisfies  $\dim_{\mathrm{B}} \Lambda < 1$.

\subsection{Behind the main theorem}
We discuss the roles of the assumptions in the main theorem and several related questions.

\paragraph{Finite generation assumption.}
The finite generation assumption is essential for concluding that the dimension of the exceptional minimal set is strictly less than $1$, as the following construction shows.
If the group is allowed to be generated by countably many diffeomorphisms, one can choose them to play ping-pong on the circle, ensuring that the group preserves an exceptional minimal set.
Moreover, for arbitrarily large $k$, one can arrange that one of the generators has a parabolic, topologically attracting fixed point of multiplicity $k+1$.
Since the exceptional minimal set of the whole group contains that of any finitely generated subgroup (to which part \eqref{mainthm: 4} of the Main Theorem applies), it follows that the former has Hausdorff dimension $1.$
This phenomenon is already known, in a stronger form, for Fuchsian groups:
the limit set of an infinitely generated Fuchsian group of the second kind may even have positive Lebesgue measure (see \cite{Mat00}).

\paragraph{Regularity.}
Another natural question concerns the regularity assumption in our main theorem. 
Different parts of the theorem exhibit different sensitivities to this assumption.

Part \eqref{mainthm: 3} of our main theorem, which establishes the strict estimate $\dimH\Lambda<1$, depends crucially on the real-analytic regularity.
This differs from the situation of the Sullivan-Ghys conjecture mentioned above.
Recall that their conjecture, which is still open in the $C^2$ setting, predicts that every exceptional minimal set of a finitely generated group of $C^2$ diffeomorphisms has zero Lebesgue measure, a statement confirmed in the real-analytic case by Deroin--Kleptsyn--Navas \cite{DKN18}.
In contrast, the real analyticity assumption in the inequality $\dim_{\mathrm{H}} \Lambda < 1$ cannot be replaced by $C^\infty$ regularity in general. 
In the $C^\infty$ setting, a parabolic fixed point may have infinite multiplicity, and the heuristics of part \eqref{mainthm: 4} of our main theorem then allow one to construct an exceptional minimal set of dimension equal to $1$, leading to the following example.
\begin{theorem}\label{thm: C infty example}
	There exists a finitely generated subgroup $G\sbs\Diff_+^\infty(\SS^1)$ with an exceptional minimal set $\Lambda$ satisfying $\dimH\Lambda=1.$
\end{theorem}

On the other hand, the equality between the Hausdorff dimension and the dynamical critical exponent also holds in certain  $C^2$ cases. 
See Propositions \ref{prop: Hausdorff upper bound} and \ref{prop: Hausdorff lower bound}.

Regarding the existence of the box dimension of an exceptional minimal set and its coincidence with the Hausdorff dimension, we expect the same result to hold in the general $C^2$ setting. 
In the broader $C^{1+}$ setting, however, it is known that certain Denjoy counterexamples preserve exceptional minimal sets whose box and Hausdorff dimensions differ.
See, for instance, \cite{KS02} and the references therein.  
Since the arguments in this paper rely heavily on real-analyticity, we propose the following conjecture.
\begin{conjecture}
	Let $ G$ be a subgroup of $\Diff^2(\SS^1)$ with an exceptional minimal set $\Lambda$.
	Then $\dim_{\mr B}\Lambda$ exists and coincides with $\dimH\Lambda$.
\end{conjecture}

\paragraph{Discreteness of the group.}
Another subtle point in our main theorem is the finiteness of the dynamical critical exponent, which is directly implied by part \eqref{mainthm: 3} of the main theorem.
In general, such finiteness requires certain discreteness assumptions on the group action.
For instance, the finiteness of the classical critical exponent for Fuchsian groups stems from their discreteness as subgroups of $\Isom_+(\DD)$.
For general groups of circle diffeomorphisms, a natural and commonly used assumption in this context is that the group is $C^1$-locally discrete (referred to simply as locally discrete in later discussions):
\begin{definition}\label{def: locally discrete}
Let $G\subset\Diff_+^1(\SS^1)$ be a subgroup without finite orbits, and let $\Lambda$ be its unique minimal set.
We say that $G$ is \emph{$C^1$-locally discrete} if, for every open interval $I\subset\SS^1$ intersecting $\Lambda$, no sequence of pairwise distinct elements of $G$ converges to the identity in the $C^1$ topology on $I$. 
\end{definition}
\begin{remark}
	The restriction to intervals meeting the minimal set follows the convention in \cite[Section 1]{pingpong1}. 
	In the smooth setting, we use a strong form of local discreteness, requiring distinct group elements rather than merely distinct restrictions; compare \cite[Section 1]{Deroin}.
	For real analytic actions, this distinction disappears.
\end{remark}
Thanks to this discreteness condition, we are able to derive an upper bound for the dynamical critical exponent (see Proposition \ref{prop: Hausdorff lower bound}).

Although not stated explicitly, this assumption is implicitly ensured by the hypotheses of our main theorem.
Indeed, if $G$ is a group of real-analytic circle diffeomorphisms, a classical argument due to Loray--Rebelo--Nakai--Shcherbakov shows that $C^1$-local non-discreteness forces the $C^1$-closure of $G$ to contain a nontrivial flow, which in turn precludes the existence of an exceptional minimal set preserved by $G.$
We refer the reader to \cite{DKN18} and the references therein for further discussion on this topic.

\subsection{An application: classification of orbit closures}
We now turn to an application of our main theorem, which concerns the classification of orbit closures. 
\begin{corollary}\label{cor: orbit closure classifications}
Let $G\sbs \Diff^\omega_+(\SS^1)$ be a finitely generated subgroup. 
Then every $G$-orbit closure is one of the following:
\begin{enumerate}
\item\label{orbit closure case 1} an infinite countable set. 
In this case, $G$ must be solvable and, in fact, is isomorphic either to $\ZZ\times \ZZ/k\ZZ$ or to a semidirect product $\ZZ\rtimes\ZZ/(2k\ZZ) $ with presentation $\pair{a,b\mid bab^{-1}=a^{-1},~b^{2k}=1},$ for some positive integer $k$;
\item\label{orbit closure case 2} a submanifold, that is, a finite set, a finite union of closed intervals, or the whole circle;
\item\label{orbit closure case 3} the union of a countable set with the unique $G$-invariant exceptional minimal set $\Lambda$. 
In this case, the orbit closure has Hausdorff dimension $\delta(G)\in (0,1)$. 
\end{enumerate}
\end{corollary}

\begin{remark}
For every $k\in \ZZ_{>0}$, the solvable group $G\cong\pair{a,b~|~bab^{-1}=a^{-1},~b^{2k}=1}$ has an abelian subgroup $\{g^{2k}: g\in G\}$ and falls into the first case of the classification of solvable group actions on the circle in \cite[Theorem 1.9]{BW04}.
See also the classification result of Navas \cite{Na04} and the recent work of Brum, Matte-Bon, Rivas and Triestino \cite{BMRT} on solvable groups of $C^1$ diffeomorphisms. 
For further results on actions on the circle by solvable and nilpotent groups, see also \cite{PT76, FF03}, and Chapter 4 of the classical textbook by Navas \cite{Na11} for a systematic treatment and more references.
\end{remark}
Although the exact value of the dynamical critical exponent $\delta(G)$ is typically difficult to compute, it is often possible to obtain a lower bound using the $C^1$-norms of the generators and the growth rate of the group.
Such a bound can be used to rule out the existence of exceptional minimal sets whenever one can establish that $\delta(G) \geqslant 1$.
In view of Corollary~\ref{cor: orbit closure classifications}, this implies that every orbit closure is a submanifold in such cases.
As an illustration, we obtain:
\begin{corollary}\label{cor: free orbit closure}
	Let $G$ be a nonabelian free group freely generated by $f_1,\cdots,f_n\in\Diff_+^\omega(\SS^1).$
	If $\max_{1\leqslant i\leqslant n}\lb{\|f_i'\|, \|(f_i^{-1})'\|}\leqslant 2n-1,$
	then every $G$-orbit closure is a submanifold.
	Moreover, the bound $(2n-1)$ is sharp. 
\end{corollary}

This corollary can be viewed as a rigidity statement concerning the classification of orbit closures.
It is generally expected that orbit closures of group actions exhibit strong structural properties under mild assumptions.
For instance, in homogeneous dynamics, orbit closures are often homogeneous themselves (see \cite{Rat91, BQ, BFLM, EMM}), while in the smooth setting, recent progress \cite{BRH, BEFRH} has shown that orbit closures may contain submanifolds for certain smooth group actions.

The rigidity phenomenon illustrated by our corollary has so far been understood only in the analytic setting. 
In the $C^\infty$ case, the conclusion of the corollary need not hold: one can choose the generators to be the identity on a fixed subinterval $I\subset\SS^1$, except for one of them, in which case the orbit closure may be countable.
Nevertheless, it remains unclear whether a fractal orbit closure can occur when the generators are sufficiently close to the identity.
These considerations naturally lead to the following question, extending the rigidity problem to general manifolds:
\begin{question}
	Let $M$ be a closed manifold, and let $G$ be a nonabelian free group freely generated by $f_1,\cdots,f_n\in\Diff^\omega(M).$
	Assume that the generators $f_1,\cdots,f_n$ are close to the identity in some $C^k$-topology. 
	Must every $G$-orbit closure be a submanifold?
    
    Furthermore, if the generators $f_1,\cdots,f_n$ are only assumed to lie in $\Diff^\infty(M),$ must every $G$-orbit closure be nonfractal? 
\end{question}

\subsection{Key ingredients}

\paragraph{Dynamical critical exponents.}
The notion of the \emph{dynamical critical exponent} was introduced in Definition \ref{def: C1 critical exponent}.
To make use of the full strength of this notion, we extend it to a broader framework by defining the \emph{dynamical critical exponent on sets}. 
This generalized version will serve as a key ingredient in the proof of our main theorem.

\begin{definition}\label{def: DCE on sets}
Let $G \subset \Diff_+^1(\SS^1)$ be a subgroup, and let $\Delta \subseteq \SS^1$ be a nonempty subset. 
The \textit{dynamical critical exponent of $G$ on $\Delta$} is defined by
\[
\delta(G, \Delta) \defeq \lim_{\varepsilon \to 0^+} \limsup_{n \to \infty} \frac{1}{n} \log \, \#\{ g \in G : \exists x \in \Delta \text{ such that } g'|_{B(x,\ve)} \geq e^{-n}\}.
\]
\end{definition}

Clearly, if $G$ has no finite orbits and preserves a minimal set $\Lambda$, then $\delta(G, \Lambda)$ coincides with the dynamical critical exponent $\delta(G)$ given in Definition \ref{def: C1 critical exponent}.

Another important instance of $\delta(G, \Delta)$ is the \textit{pointwise dynamical critical exponent}, defined for $x\in\SS^1$ by $\delta(G,x) \defeq \delta(G,\{x\})$.
When $G$ is finitely generated and $x$ is wandering, $\delta(G,x)$ coincides with the abscissa of convergence of the Poincar\'{e} series (Proposition \ref{prop: wandering})
\[
P(x,s) = \sum_{g \in G} g'(x)^s.
\]
This series serves as a dynamical analogue of the Poincar\'{e} series $P(s) = \sum_{\gamma \in \Gamma} e^{-s d_{\text{hyp}}(o, \gamma o)}$ for Fuchsian groups $\Gamma$, as well as $P_z(s)=\sum_{n=0}^\infty\sum_{ \xi \in f^{-n}(z)} |D f^{n}(\xi)|^{-s} $ for rational maps $f:\CC\to\CC$ (see \cite{GS09,AL1}).
The pointwise dynamical critical exponent plays a central role in constructing and analyzing certain fractal measures supported on the minimal set, in close analogy with the Fuchsian group case, and is essential for our dimension-theoretic results.

Unlike the classical critical exponents associated with Fuchsian groups and rational maps, the pointwise dynamical critical exponent $\delta(G, x)$ is genuinely \textbf{point-dependent}: it may vary with $x$ and may even exceed $\dimH\Lambda$.
This feature causes additional technical difficulties in the proof and leads to certain subtleties in the statement of our main theorem:
\begin{enumerate}
    \item The quantity $\delta(G)$ appearing on the right-hand side of the identity in part \eqref{mainthm: 2} (namely, $\delta(G,\Lambda)$) cannot, in general, be replaced by $\delta(G, \SS^1)$.
    Although $\delta(G,\SS^1)$ may seem more natural since it is defined 
    independently of the minimal set, it does not necessarily coincide with $\dimH \Lambda$.
    See Remark~\ref{rem: delta(G,S1)}.
    \item The assumption in part \eqref{mainthm: 4} that the parabolic fixed point lies in $\Lambda$ is essential. 
    An explicit example illustrating this necessity is given in Proposition~\ref{prop: parabolic outside Lambda}.
\end{enumerate}

\paragraph{Fractal measures supported on the minimal set.}
A general approach to the dimension theory of fractal sets is through the study of measures supported on them.
In our context, we focus on measures that additionally possess certain dynamical properties.
Two such notions will be used throughout this paper: the \textit{stationary measure} and the \textit{conformal measure}. 
\begin{definition}\label{def: stationary and conformal}
Let $G$ be a subgroup of $\Diff_+^1(\SS^1).$
	\begin{enumerate}[(1)]
		\item\label{it: stationary} Let $\mu$ be a probability measure on $G.$ 
        A probability measure $\nu $ on the circle is said to be \textit{$\mu$-stationary} if it satisfies the identity
				\[\nu=\mu*\nu,\]
			where $\mu*\nu\defeq\int g_*\nu\dd\mu(g).$
		\item\label{it: conformal} Let $\delta\geqslant 0$ be a constant. 
        A probability measure $\nu$ on the circle is said to be a \textit{$\delta$-conformal measure (for $G$)} if, for every Borel set $A\sbs\SS^1$ and for every map $g\in G$,
				\[\nu(gA)=\int_A |g'(x)|^\delta \dd\nu(x).\]
        Here, $\delta$ is called the \textit{conformal exponent} of $\nu.$
	\end{enumerate}
\end{definition}

Stationary and conformal measures are both classical tools in the study of group actions and dimension theory, each with its own advantages and limitations. 
To this end, we employ them in tandem: stationary measures are used to establish lower bounds for the dimension of the minimal set, while conformal measures are used mainly to derive the corresponding upper bounds.

The main advantage of stationary measures lies in their flexibility and ease of construction: any probability measure on the group $G$ naturally induces a random walk on the circle, which in turn yields a stationary measure.
However, an arbitrary stationary measure may have a dimension quite different from that of the minimal set.
For our purposes, we proceed in two steps: 
\begin{enumerate}[\indent(S1)]
\item\label{it:S1} Estimate the dimension of a given stationary measure.
\item\label{it:S2} Construct random walks whose stationary measures have large dimension.
Specifically, we aim to construct stationary measures whose dimensions approach $\delta(G)$. 
We refer to this type of approximation as a \textit{variational principle}.
\end{enumerate}

In contrast, conformal measures are much less abundant.
Among these, atomless conformal measures supported on the minimal set $\Lambda$ are the key objects encoding the dimensional properties of $\Lambda$, and Deroin--Kleptsyn--Navas \cite[Theorem F]{DKN09} showed that such a measure is unique when the group satisfies property $(\Lambda\star)$. 
However, the existence of such a measure is often unclear.
In practice, we proceed in two steps:
\begin{enumerate}[\indent(C1)]
\item\label{it:C1} Construct an atomless conformal measure supported on $\Lambda$.
\item\label{it:C2} Relate the conformal exponent $\delta$ of the constructed measure to the dimension of $\Lambda$ and of the measure itself.
\end{enumerate}

Let us highlight that, to carry out the scheme outlined above, especially the steps (\hyperref[it:S2]{S2}) and (\hyperref[it:C1]{C1}), we rely crucially on the notion of dynamical critical exponent. 
We now turn to a more detailed discussion of this approach.

\paragraph{Stationary measures.}
We establish the exact dimensionality and a dimension formula for stationary measures on the circle.
The proofs of these dimensional properties, as well as some other results needed in the present paper, rely on a thorough study of smooth random walks on the circle, drawing on techniques from nonuniformly hyperbolic and partially hyperbolic dynamics.
Since these methods differ substantially from those developed here and are of independent interest, they are presented in a companion paper \cite{HJX1}.
Here, we briefly describe the result that will be used in the present work.

The dimension of a stationary measure is expressed as a Ledrappier--Young type formula, namely, as the ratio of an entropy to the absolute value of a Lyapunov exponent.
However, the entropy that naturally arises in this context is the \textit{Furstenberg entropy}, which is in general difficult to compute.
If the random walk satisfies certain discreteness conditions, however, it coincides with \textit{the random walk entropy}, which is easier to estimate and will be used later in the variational principle. 
In homogeneous settings, this identification holds quite generally.
For instance, this holds for random walks on discrete subgroups of $\SL_n(\RR)$ that act on full flag varieties, due to properties of Poisson boundaries.
It has also been extended to exponentially separating random walks in low-dimensional cases  \cite{Hoch14, HS17, BHR, Rap24, LPX}.
In our smooth setting, we establish an analogous identity under the assumption of $C^1$-local discreteness (see Theorem \ref{thm: stationary measure dimension}).
A related result was previously obtained by Deroin~\cite{Deroin}, although our setting requires a stronger version for subsequent arguments.

\paragraph{Variational principle.}
As we explained above, we aim to show a variational principle for the dynamical critical exponent $\delta(G)$, see Proposition \ref{prop: variational principle of critical exponents}.
Similar variational principles have been studied in several homogeneous contexts \cite{MS, MS23, JLPX}.
Let us briefly outline the strategy of the proof.

Thanks to the dimension formula, it suffices to find a finite subset in $G$ generating a random walk with large entropy and a Lyapunov exponent of small absolute value.
To achieve this, we look for large subsets of $G$ that freely generate semigroups and whose elements have controlled derivatives on the circle.
The main difficulty of the construction comes from the fact that we must treat $G$ simultaneously as an abstract group and as a group of diffeomorphisms.
While the construction of large free semigroups in abstract groups is usually studied in the broader framework of groups with hyperbolic features, our setting requires adapting these ideas to the smooth dynamical context.
Accordingly, we develop new geometric-group–theoretic tools specifically tailored to this purpose.

\paragraph{Conformal measures.}
The concept of conformal measures was introduced by Sullivan \cite{Sul83} and has become a powerful tool for studying conformal group actions.
Patterson-Sullivan measures provide central examples and are widely used to estimate the dimensions of limit sets.

As we mentioned above, the existence of an atomless conformal measure supported on the minimal set $\Lambda$ is not guaranteed in general.
In fact, the example constructed in Theorem \ref{thm: C infty example} admits no such measure, see Remark \ref{rem: no conformal measures} for details.
Even in the real-analytic setting, constructing such a measure remains a delicate task. 
The main obstruction arises from the presence of parabolic fixed points outside $\Lambda,$ a phenomenon absent in the case of Fuchsian group actions or rational maps. 
A key observation is that a smooth modification of the action outside $\Lambda$, while leaving it unchanged on $\Lambda$, does not affect the existence of conformal measures supported on $\Lambda.$
This flexibility enables us to overcome the obstruction and establish the existence of the desired conformal measure.

The next step is to analyze the dimensional properties of the resulting atomless conformal measure $\nu$ supported on $\Lambda.$
A useful tool here is the expansion strategy developed by Deroin--Kleptsyn--Navas, originally inspired by Sullivan and further applied in \cite{DKN09, DKN18} to study ergodic properties of smooth group actions on the circle.
This method allows one to expand a small interval by elements of $G$ with uniformly controlled distortion.
By combining it with the conformality of $\nu$, we derive, among other results, an anti-Frostman-type estimate (see Proposition \ref{prop: conformal lower local dim}), analogous to the shadow lemma in hyperbolic geometry.
These estimates yield both Hausdorff and box dimension bounds for $\Lambda$.

\subsection*{Organization.}
Section \ref{sec: 2} recalls preliminaries on dimension theory and group actions on the circle, together with results from the companion paper \cite{HJX1} needed in the present work.
In Section \ref{sec: 3}, we study basic properties of dynamical critical exponents.
Section \ref{sec: 4} establishes the relationship between the Hausdorff dimension of minimal sets and the dynamical critical exponent. 
A key geometric-group-theoretic lemma used in this section is proved in Appendix \ref{app}.
In Section \ref{sec: 5}, we construct conformal measures and investigate their dimensional properties.
Finally, Section \ref{sec: 6} completes the proofs of the main results and presents examples to illustrate subtleties of our theorems. 

\subsection*{Notation and conventions.}

We identify the circle $\SS^1$ with the quotient $\RR/\ZZ$, which is endowed with the cyclic order: for pairwise distinct $x,y,z,$ define $x\prec y\prec z$ by the existence of lifts $\wt x<\wt y<\wt z<\wt x+1.$ 
Define the relation $ x\preceq y\preceq z$ accordingly.
Then any closed subinterval of $\SS^1$ can be written in the form $[x,y]\defeq \lb{z\in\SS^1: x\preceq z\preceq y}$.
For a Borel subset $E \subset \SS^1$, we denote its Lebesgue measure by $|E|$.
In particular, if $I \subset \SS^1$ is a subinterval, then $|I|$ coincides with its length.
To distinguish notation, we write $\# F$ for the cardinality of a finite set $F$.

Given a subinterval $I \subset \SS^1$ and a constant $k < |I|^{-1}$, we denote by $kI$ the subinterval of $\SS^1$ having the same center as $I$ and length $k|I|$.
The interval $kI$ is required to inherit the same openness or closedness at its endpoints as $I$.

For a subset $\cG$ of some group, we denote $\cG^{-1}\defeq\lb{g^{-1}:g\in\cG}$ and, for every positive integer $n$, $\cG^{*n}\defeq\lb{g_1\cdots g_n: g_i\in\cG}$; both are subsets of the ambient group.
For a probability measure $\mu$ on a group $G$ and an integer $n\in\NN,$ $\mu^{*n}=\mu*\dotsm *\mu$ denotes the $n$-fold convolution of $\mu.$

We use Vinogradov notation: $A\ll B$ means that $A$ is bounded above by a constant multiple of $B$ and $A\ll_\square B$ if the implicit constant depends only on $\square.$
We write $A\asymp B$ (resp. $A\asymp_\square B$) if both $A\ll B$ and $B\ll A$ (resp. $A\ll_\square B$ and $B\ll_\square A$) hold.

\subsection*{Acknowledgements.}
W. H. is supported by the National Natural Science Foundation of China (No. 12288201). 
Y. J. is supported by the National Natural Science Foundation of China (No. 124B2008).
D. X. is supported by National Key R\&D Program of China No. 2024YFA1015100 and NSFC grant 12090010, 12090015. 
Part of this work was completed during D. X.'s visits to AMSS and the University of Chicago.
He thanks both institutions for their hospitality. 
Y. J. thanks Gang Tian for his constant support and encouragement. 

We thank Amie Wilkinson for useful suggestions and for pointing out how to strengthen an earlier version of Corollary~\ref{cor: orbit closure classifications}. 
We thank Artur Avila, De-Jun Feng, Shaobo Gan, Fran\c cois Ledrappier, Pablo Lessa, Jialun Li, Minghui Ouyang, Wenyu Pan, Yi Shi, and Wenyuan Yang for useful discussions.

\section{Preliminaries}
\label{sec: 2}

\subsection{Distortion estimates}\label{sec: distortion}
In this subsection we collect some tools for distortion control.
They allow one to control compositions of $C^2$ maps on small intervals in terms of their derivatives at a single point.
These techniques date back to Denjoy~\cite{Den32}, Schwartz~\cite{Sch63} and Sacksteder~\cite{Sa65} and were later recast from an influential viewpoint by Sullivan~\cite{Sul83}.
Further developments and discussions on this topic can be found in \cite{DKN07, DKN09, DKN18}.

Let $f\in\Diff^2_+(\SS^1)$ and let $I\sbs\SS^1$ be an interval.
Define the \textit{distortion coefficient} of $f$ on $I$ by
\[\vk(f,I)\defeq\sup_{x,y\in I}\abs{\log f'(x)-\log f'(y)}.\]
Then the distortion coefficients satisfy the following inequality for every $g \in \Diff^2_+(\SS^1)$,
\begin{equation}\label{eqn: distortion sum}
	\vk(fg,I)\leq\vk(g,I)+\vk(f,g(I)).
\end{equation}
Fix a finite subset $\cS\sbs\Diff_+^2(\SS^1),$ and set
\[M=M(\cS)\defeq\max\lb{\max_{g\in\cS}\|g'\|_{C^0},\max_{g\in\cS}\|\log g'\|_{\mr{Lip}}}.\]
Given a sequence $g_1, \dotsc, g_n \in \cS$, consider
\[f_0=\Id \quad \text{and} \quad f_k=g_k\cdots g_2g_1,\ \forall k \in \{1, \dotsc, n\}.\]
Then for every interval $I\sbs\SS^1,$ we have
\begin{equation}
\label{eqn:vksumI}
\vk(f_n,I)\leq\sum_{k=0}^{n-1}\vk(g_{k+1},f_k(I))\leq M\sum_{k=0}^{n-1}\abs{f_k(I)}.
\end{equation}
Let $x_0\in I$.
Set $I_0=I$ and then $I_k=f_k(I_0)$ for $k = 1,\dotsc,n$.
We have the following inequality \cite[Corollary 2.3]{DKN18}.
\begin{equation}
\label{eqn:vkfkI0}
\forall k=1,\dotsc,n,\quad \abs{\log\frac{f_k'(x_0)\abs{I_0}}{\abs{I_k}}}\leq\vk(f_k,I_0)\leq M\sum_{i=0}^{k-1}\abs{I_i}.
\end{equation}
Recall also the following estimate from \cite[Proposition 2.4]{DKN18}.

\begin{proposition}\label{prop: distortion estimate}
Let $M= M(\cS)$ and $(f_k)_{0 \leq k \leq n}$ be as above.
Given a point $x_0 \in \SS^1$, write $S=\sum_{k=0}^{n-1}f_k'(x_0)$.
Then for any $r \in \bigl]0,\frac{\log 2}{4MS}\bigr]$, we have
\begin{equation}
\label{eqn:2MSdelta}
\vk(f_n,B(x_0,r)) \leq 4MS r.
\end{equation}
\end{proposition}

\subsection{Group actions on the circle and DKN's expanding strategy}\label{sec: group actions on the circle}
We recall several results on groups of circle homeomorphisms  $G\sbs\Homeo(\SS^1).$
We include proofs of some standard facts for completeness.
Some related arguments and further discussion can be found in \cite[Chapter 2]{Na11}.

\begin{lemma}
\label{lem: finite index minimal set}
Let $G_2\sbs G_1$ be two subgroups of $\Homeo(\SS^1)$ without finite orbits on $\SS^1$.
If $G_2$ has finite index in $G_1$, then their unique minimal sets coincide.
\end{lemma}
\begin{proof}
By a standard construction in group theory, there exists a finite index subgroup $G_3\sbs G_2$ that is normal in $G_1.$
Then $G_3$ has no finite orbit either.
Let $\Lambda$ denote its unique invariant minimal set.
For every $g\in G_1,$ observe that $g\Lambda$ is also a minimal set of $gG_3g^{-1}=G_3.$
By uniqueness, it follows that $g\Lambda=\Lambda$ for all $g\in G_1.$
Hence $\Lambda$ is $G_1$-invariant and minimal.
\end{proof}
\begin{lemma}
\label{lem: invariant measure and finite orbit}
Let $G \sbs\Homeo_+(\SS^1)$ be a finitely generated subgroup.
If $G$ preserves a common invariant probability measure $\nu$ on $\SS^1,$ then either it has a finite orbit or it contains an element with an irrational rotation number.
In particular, if $G$ has no common invariant measure on $\SS^1$, then it has a unique minimal set.
\end{lemma}
\begin{proof}
For every $x \in \SS^1$ and $f \in G$, we have $\nu([x,f(x)[) = \tau(f)$, where $\tau(f)$ denotes the rotation number of $f$ (see \cite[Theorem 2.2.10]{Na11}).
Assume that every element of $G$ has a rational rotation number.
Choose $x_0 \in \supp \nu$.
Note that if $N \tau(f) \in \ZZ$ for some positive integer $N$, then $f$ preserves the finite set $E_N=\lb{\, y \in \supp\nu :  \nu([x_0,y[) \in \frac{1}{N}\ZZ\,}.$
Taking $N$ to be a common multiple of the denominators of the rotation numbers of a finite generating set of $G$, we obtain a finite set $E_N$ that is invariant under all generators, hence under $G$.
\end{proof}

The following contraction property is also well-known and is closely related to the discussion of contraction coefficients in \cite[Section 2.3]{Na11}.

\begin{lemma}\label{lem: global contraction}
	Let $G \sbs \Homeo_+(\SS^1)$ be a finitely generated subgroup without common invariant probability measure on $\SS^1,$ and let $\Lambda \sbs \SS^1$ be its unique minimal set. 
	Then there exists $\ve_1>0$ such that for any $x,y\in\Lambda$ and any $\ve> 0$ there exists $f\in G$ such that $f(B(x,\ve_1))\sbs B(y,\ve).$
\end{lemma}
We will only need this statement in $\Diff_+^2(\SS^1)$.
A proof of this special case is given in Part I \cite[Lemma 2.31]{HJX1}.

The main theorem concerns group actions by real-analytic circle diffeomorphisms.
The following theorem collects several fundamental results describing  both the group structure and the dynamics in this context.
\begin{theorem}\label{thm: real analytic with exceptional}
  Let $G\sbs\Diff_+^\omega(\SS^1)$ be a finitely generated subgroup with an exceptional minimal set.
  Then the following statements hold:
  \begin{enumerate}[\indent(1)]
      \item\label{fact: virtually free} $G$ is virtually free.
      \item\label{fact: Hector} For every $x\in\SS^1$, the stabilizer $\Stab_G(x)=\lb{g\in G:g(x)=x}$ is either trivial or infinite cyclic.
      \item\label{fact: locally discrete} $G$ is locally discrete (see Definition~\ref{def: locally discrete}).
  \end{enumerate}
\end{theorem}
Item \eqref{fact: virtually free} is due to Ghys~\cite{Ghys87}.
Item \eqref{fact: Hector} was established by Hector~\cite{Hector}, and a complete proof can also be found in~\cite{Na06}.
Finally, for Item \eqref{fact: locally discrete}, see the discussion after Definition~\ref{def: locally discrete}.
A proof can be found, for instance, in~\cite[Proposition 3.2]{Mat09}.

\paragraph{Deroin--Kleptsyn--Navas (DKN) expanding strategy.}
In order to study $C^2$ group actions on the circle, Deroin, Kleptsyn and Navas introduced the notions of property $(\star)$ and property $(\Lambda\star).$
Under these conditions, they developed an expanding strategy, drawing on ideas of Sullivan, and used it to establish several fundamental ergodic properties of smooth actions on the circle.
In their celebrated work~\cite{DKN18}, they proved that these assumptions hold in most real-analytic cases (Theorem~\ref{thm: DKN18}), thereby resolving the Ghys--Sullivan conjecture in this setting.
In the remainder of this subsection, we briefly recall the techniques they developed in~\cite{DKN09, DKN18}.

\begin{definition}\label{def: NE}
	A point $x\in\SS^1$ is called \emph{non-expandable} for the action of $G\sbs\Diff_+^1(\SS^1)$ if $g'(x)\leq 1$ for every $g\in G.$
We denote the set of non-expandable points by $\NE=\NE(G)$.
\end{definition}
\begin{definition}\label{def: property star}
	Let $G\sbs\Diff_+^1(\SS^1)$ act minimally on the circle.
	We say that $G\sbs\Diff^1(\SS^1)$ satisfies \emph{property $(\star)$} if,  for every $x\in\NE,$ there exist elements $g_+,g_-\in G$ such that $g_+(x)=g_-(x)=x$, and $x$ is isolated from the right (resp. from the left) in the set of fixed points $\Fix(g_+)$ (resp. $\Fix(g_-)$).
\end{definition}
\begin{definition}\label{def: property Lambda star}
	Let $G\sbs\Diff_+^1(\SS^1)$ be a subgroup with an exceptional minimal set $\Lambda.$
	We say that $G\sbs\Diff^1(\SS^1)$ satisfies \emph{property $(\Lambda\star)$} if for every $x\in\NE\cap\Lambda,$ there exist elements $g_+,g_-\in G$ such that $g_+(x)=g_-(x)=x$, and $x$ is isolated from the right (resp. from the left) in the set of fixed points $\Fix(g_+)$ (resp. $\Fix(g_-)$).
\end{definition}

\begin{theorem}[{\cite[Theorems A, D]{DKN09}}]\label{thm: DKN09}
Let $G\sbs\Diff_+^2(\SS^1)$ be a finitely generated subgroup satisfying property $(\star)$ or $(\Lambda\star)$.
Let $\Lambda$ denote $\SS^1$ if $G$ acts minimally or the exceptional minimal set of $G$ if it has one.
Then:
\begin{enumerate}
	\item\label{fact: NE finite} The intersection $\Lambda\cap \NE$ is finite.
	\item\label{fact: NE orbit} Denote by $G(\NE)$ the $G$-orbit of $\NE=\NE(G).$ 
      For every $x\in\Lambda\sm G(\NE)$, the set $\lb{g'(x) : g\in G}$ is unbounded.
\end{enumerate}
\end{theorem}

\begin{theorem}[{\cite{DKN18}}]\label{thm: DKN18}
Let $G\sbs\Diff_+^\omega(\SS^1)$ be a finitely generated subgroup.
Then:
	\begin{enumerate}
		\item If $G$ is free of rank at least $2$ and acts minimally on the circle, then it satisfies property $(\star)$.
		\item If $G$ admits an exceptional minimal set $\Lambda$ then it satisfies property $(\Lambda\star)$.
	\end{enumerate}
\end{theorem}

Let us explain DKN's expanding strategy for groups satisfying properties $(\star)$ or $(\Lambda\star).$
This strategy, introduced in \cite{DKN09}, allows one to expand subintervals of the circle to a fixed length using the group action, while keeping distortion uniformly controlled.
Given a subinterval, the element that performs the expansion is of course not unique.
However, since our use of the expanding strategy requires a slightly more delicate setup than in~\cite{DKN09}, we will fix once and for all the choice of such an element.
The choice will be made in \eqref{eqn: expanding def} and \eqref{eqn: expanding def 2} and we denote it by $t(n,z)$.

Let $G\subset \Diff_+^2(\SS^1)$ be a finitely generated subgroup without finite orbits, and let $\Lambda$ be its unique minimal set.
Assume that $G$ satisfies property $(\star)$ or $(\Lambda\star).$
Then $\NE\cap\Lambda$ is finite by Theorem \ref{thm: DKN09}.

For each $x\in \NE\cap\Lambda$ and each choice of sign $\sigma\in\lb{+,-},$ we first choose an interval $\wt I_x^\sigma$ and a group element $g_{x,\sigma}\in G$ as follows.
By property $(\star)$ or $(\Lambda\star)$ and the finiteness of $\NE\cap\Lambda$, we may choose a common length $\ve>0$ and, for each $x\in\NE\cap\Lambda$, elements $g_{x,+},g_{x,-}\in G$ fixing $x$, such that the intervals
\[ \wt I_x^+={]x,x+\ve]},\qquad\wt I_x^-={[x-\ve,x[}\]
are pairwise disjoint and satisfy
\[ g_{x,+}(z)>z \quad\text{for }z\in\wt I_x^+,\qquad g_{x,-}(z)<z \quad\text{for }z\in\wt I_x^-.\]

The following lemma follows from \cite[Lemma 5.7]{DKN09}.
\begin{lemma}\label{lem: DKN expand escape}
	Let $x\in\NE\cap\Lambda$ and $\sigma\in\lb{+,-}.$
	There exists a positive integer $p_{x,\sigma}$ such that the following hold.
	Let
	\begin{equation}\label{eqn: expanding J def}
		J_x^\sigma\defeq\case{&{]g_{x,+}^{-p_{x,+}-1}(x+\ve),g_{x,+}^{-p_{x,+}}(x+\ve)]}, &\sigma=+;\\ &{[g_{x,-}^{-p_{x,-}}(x-\ve),g_{x,-}^{-p_{x,-}-1}(x-\ve)[},&\sigma=-. }
	\end{equation}
	Then for every integer $n\geqslant 0$ and every point $z\in g_{x,\sigma}^{-n}(J_x^\sigma),$ we have 
	\[(g_{x,\sigma}^{n+p_{x,\sigma}})'(z)\geqslant 2. \]
\end{lemma}

\begin{definition}\label{def: I_x g_x}
	For each $x\in \NE\cap\Lambda$ and each choice of sign $\sigma\in\lb{+,-},$ we choose the element $g_{x,\sigma}\in G$ and the interval $\wt I_x^\sigma$ as above.
	We fix a choice of positive integer $p_{x,\sigma}$ given by the previous lemma.
	Let $I_x^\sigma=g_{x,\sigma}^{-p_{x,\sigma}}(\wt I_x^\sigma)$ and $J_x^\sigma$ be given by \eqref{eqn: expanding J def}.
\end{definition}

\begin{remark}\label{rem: Ix real analytic}
	When $G\sbs\Diff_+^\omega(\SS^1)$, we may choose the constant $\ve>0$ in the above construction even smaller to further guaranty that \[ \inf_{z\in I_x^\sigma} g_{x,\sigma}'(z)>1,\quad \forall x\in\NE\cap\Lambda, \, \sigma\in\lb{+,-}.\]
	This will be helpful in the proof of Lemma~\ref{lem: conformal lower local dim}.
\end{remark}

Note that we have the disjoint decomposition
\begin{equation}\label{eqn: decompose I into J}
    I_x^\sigma =\bigsqcup_{n=0}^\infty g_{x,\sigma}^{-n}(J_x^\sigma).
\end{equation}
Besides, $I_x^\sigma\sbs \wt I_x^\sigma$ and hence $I_x^{\sigma}$ are pairwise disjoint for $x\in\NE\cap\Lambda$ and $\sigma\in\lb{+,-}.$

Let $U_x$ denote the interior of $\lb{x}\cup I_x^-\cup I_x^+,$ which is an open neighborhood of $x.$
We define
\begin{equation}\label{eqn: U_NE}
  U_{\NE}\defeq\bigcup_{x\in \NE \cap \Lambda} U_x=\bigcup_{x\in \NE\cap\Lambda}\lb{x}\cup\bigcup_{x\in\NE\cap\Lambda}\bigcup_{\sigma\in\lb{+,-}}(I_{x}^\sigma )^\circ,
\end{equation}
which is an open neighborhood of $\NE \cap \Lambda$.

By definition, for each $y\in \Lambda\sm U_{\NE},$ there exists a group element $f_y\in G$ satisfying $f_y'(y)>1.$
Hence, for some open neighborhood $V_y\ni y,$ we also have $\inf_{V_y} f_y'>1.$
By the compactness of $\Lambda\sm U_{\NE},$ we may choose finitely many points $y_1,\dots,y_m$ such that $\Lambda\sm U_{\NE}\sbs \bigcup_{i=1}^m V_{y_i}.$
Together with the finiteness of  $\NE\cap\Lambda,$ this allows us to define the constant
\begin{equation}\label{eqn: expanding lower bound}
	\beta\defeq \min\lb{2,\min_{i=1,\cdots,m} \inf_{V_{y_i}}f_{y_i}'},
\end{equation}
which is strictly greater than $1$.
Set
\begin{equation}\label{eqn: expanding elements}
	\cT=\lb{g_{x,\sigma},g_{x,\sigma}^{p_{x,\sigma}}: x\in \NE\cap\Lambda ,\sigma\in\lb{+,-}}\cup\lb{f_{y_i}:i=1,\cdots,m },
\end{equation}
which is a finite subset of $G.$
We also set
\begin{equation}\label{eqn: expanding upper bound}
	L\defeq \max_{f\in \cT}\max\lb{\|f'\|,\|(f^{-1})'\|} .
\end{equation}

Let $\Lambda'=\Lambda\sm G(\NE(G)).$ Note that $\Lambda'$ is $G$-invariant.
For each $z \in \Lambda'$, we construct a sequence $(t(n,z))_{n \in \NN}$ of elements of $G$ that expands neighborhoods of $z$.
More precisely, we define a map $t : \NN\times\Lambda'\to G$ as follows.
\begin{itemize}
\item Let $t(0,z) = \Id \in G$ for all $z \in \Lambda'$.
\item Choose $t(1,z)$ from $\cT$ so that for every $z \in \Lambda'$,
  \begin{equation}\label{eqn: expanding def}
  t(1,z)\defeq \case{&g_{x,\sigma}^{p_{x,\sigma}},& \text{if $z\in J_x^\sigma$ for some } (x,\sigma) \in (\NE \cap \Lambda) \times \{+,-\};\\ &g_{x,\sigma}, & \text{if $z\in I_x^\sigma\sm J_x^\sigma$ for some } (x,\sigma) \in (\NE \cap \Lambda) \times \{+,-\};\\
  & f_{y_i},&\text{if } z\notin\ol{U_{\NE}}\text{ and }z\in V_{y_i}.}
  \end{equation}
  In the first two cases, $(x,\sigma)$ is uniquely determined by $z$, while in the third case, $z$ might lie in more than one of the sets $V_{y_i}$.
  To eliminate this ambiguity, we fix one admissible index $i$ for each such $z.$
\item Require that $(n,z) \in \NN \times \Lambda' \mapsto T^n z \defeq t(n,z) z \in \Lambda'$ is a semigroup action of $\NN$ on $\Lambda'$.
  More precisely, we require that $t$ is a $G$-valued cocycle with respect to $T$:
  \begin{equation}\label{eqn: expanding def 2}
  \forall n, m \in \NN,\; \forall z \in \Lambda'\quad   t(n + m,z) = t(n, T^mz) \circ t(m,z).
  \end{equation}
\end{itemize}
It is straightforward to check that, up to the choice in \eqref{eqn: expanding def}, this uniquely determines  well-defined maps $t$ and $T$.
In what follows, we shall see that $t(\cdot,z)$ indeed specifies elements of $G$ that expand neighborhoods of $z$.

\begin{definition}\label{def: Acceleration times}
	Let $z\in\Lambda'$. 
    We define the \textit{first acceleration time} $n_1(z)=n(z)$ as follows.
    If $z\not\in \ol{U_{\NE}}$, we set $n(z)=1.$
    Otherwise $z\in I_x^\sigma$ for some $(x,\sigma) \in (\NE \cap \Lambda) \times \{+,-\}.$ 
    By the decomposition \eqref{eqn: decompose I into J}, there exists a unique $n\geqslant 0$ such that $g_{x,\sigma}^{n} z \in J_x^\sigma.$
    We then define $n(z)=n+1.$
    Define recursively,
    \begin{equation}
    \label{eq:defn_m}
    \forall m \in \ZZ_{>0},\quad n_{m+1}(z)\defeq n_{m}(z)+n_1(T^{n_m(z)}z).
    \end{equation}
	The collection of these integers,
	\[\Acc(z)\defeq \lb{n_m(z):m\in\ZZ_{>0}}\]
	is called the set of \emph{acceleration times} for $z$, which is infinite.
\end{definition}

\begin{figure}[!ht]
\centering
	\begin{tikzpicture}[
  x=1cm,y=1cm,
  font=\small,
  >={Stealth[length=2mm,width=1.35mm]},
  orbit arrow/.style={->,line width=.55pt},
  terminal arrow/.style={->,line width=1.05pt},
  point/.style={circle,fill=black,inner sep=1.3pt},
  checkpoint/.style={circle,fill=black,inner sep=2.1pt},
  every node/.style={inner sep=2pt}
]
  \fill[black!9] (.48,-.065) rectangle (7.15,.065);
  \fill[black!19] (1.3,-.065) rectangle (2.05,.065);
  \fill[black!30] (5.1,-.065) rectangle (7.15,.065);
  \fill[black!30] (10.15,-.065) rectangle (12.35,.065);
  \draw[line width=.65pt] (.35,0) -- (7.15,0);
  \draw[densely dotted,line width=.55pt] (7.15,0) -- (10.15,0);
  \draw[line width=.65pt] (10.15,0) -- (12.35,0);
  \foreach \s in {.48,.56,.67,.82,1.03,1.3,2.05,2.7,3.7,5.1,7.15,10.15,12.35}
    \draw[line width=.45pt] (\s,-.11) -- (\s,.11);
  \draw[fill=white,line width=.55pt] (.35,0) circle[radius=1.5pt];
  \node[below=5pt] at (.35,0) {$x$};

  \node at (1.68,.43) {$g^{-n}J$};
  \node at (3.08,.43) {$\cdots$};
  \node at (4.4,.43) {$g^{-1}J$};
  \node at (6.15,.43) {$J$};
  \node at (11.25,.43) {$g^pJ$};

  \node[point] at (1.68,0) {};
  \node[point] at (2.37,0) {};
  \node[point] at (4.4,0) {};
  \node[point] at (6.15,0) {};
  \node[point] at (11.25,0) {};
  \node[below=5pt] at (1.68,0) {$z$};
  \node[below=5pt] at (2.37,0) {$gz$};
  \node[below=5pt] at (4.4,0) {$g^{n-1}z$};
  \node[below=5pt] at (6.15,0) {$g^nz$};
  \node[below=5pt] at (11.25,0) {$g^{n+p}z$};

  \draw[orbit arrow] (1.68,.78)
    to[bend left=48] node[above=1pt] {$g$} (2.37,.78);
  \draw[orbit arrow,densely dashed] (2.55,.78)
    to[bend left=17] node[above=1pt] {$\cdots$} (4.22,.78);
  \draw[orbit arrow] (4.4,.78)
    to[bend left=26] node[above=1pt] {$g$} (6.15,.78);
  \draw[terminal arrow] (6.4,.78)
    to[bend left=11]
    node[above=1pt] {$g^p$ \ (one cocycle step)} (11.25,.78);

  \draw[decorate,decoration={brace,mirror,amplitude=4pt},line width=.45pt]
    (.35,-.7) -- (7.15,-.7)
    node[midway,below=6pt] {$I=\bigsqcup_{j\geq0}g^{-j}J$};
\end{tikzpicture}
\stepcounter{theorem}
\caption{Acceleration near a non-expandable point.
Here $g=g_{x,+}$, $p=p_{x,+}$, $I=I_x^+$, and $J=J_x^+$. 
A point $z\in g^{-n}(J)\cap\Lambda'$ reaches $J$ after $n$ steps; the terminal map $g^p$ counts as one further step of $T$.
Thus $n_1(z)=n+1$ and $t(n_1(z),z)=g^{n+p}$.}\label{fg: acceleration times}
\end{figure}
\begin{lemma}\label{lem: parabolic expanding}
    Assume that $z\in I_x^\sigma\cap \Lambda'$ for some $(x,\sigma) \in (\NE \cap \Lambda) \times \{+,-\}.$ 
    Then for every $1\leqslant \ell\leqslant n(z)-1,$ we have $T^{\ell-1}z \in I_x^\sigma\sm J_x^\sigma$ and hence $t(1 ,T^{\ell-1}z) =g_{x,\sigma}.$
    For $\ell= n(z),$ we have $T^{\ell -1}z\in J_x^\sigma$ and hence $t(1 ,T^{\ell-1}z) =g_{x,\sigma}^{p_{x,\sigma}}.$
    See Figure \ref{fg: acceleration times}.
\end{lemma}
\begin{proof}
	The case $n(z)=1$ is clear.
	Suppose now that $n(z)\geqslant 2.$
    By Definition \ref{def: Acceleration times}, we have $T^0 z  = z\in I_x^\sigma\sm J_x^\sigma$, hence $t(1, T^0 z) = t(1,z)=g_{x,\sigma}$ by \eqref{eqn: expanding def}.
    We then proceed by induction on $\ell$.
    Assume that the statement holds for all $\ell' \in \{1,\dotsc, \ell-1\}$.
    Then by \eqref{eqn: expanding def 2}, we have $t(\ell - 1, z)= t(1, T^{\ell - 2} z) \dotsm t (1, z)= g_{x,\sigma}^{\ell - 1}.$
    By Definition \ref{def: Acceleration times}, $g_{x,\sigma}^{n(z)-1}z\in J_x^\sigma.$
    By the decomposition of $I_x^\sigma$ \eqref{eqn: decompose I into J}, we have
    \[
        T^{\ell - 1}z = t(\ell - 1,z) z \in g_{x,\sigma}^{-(n(z)-\ell)}(J_x^\sigma)\subset I_x^\sigma\sm J_x^\sigma,
    \]
    whenever $\ell\leqslant n(z)-1.$
    This proves the statement for $\ell\leqslant n(z)-1.$
    
    Therefore, we have $t(n(z)-1,z)=g_{x,\sigma}^{n(z)-1}$ and hence $T^{n(z)-1}z\in J_x^\sigma$ by Definition \ref{def: Acceleration times}.
\end{proof}

\begin{lemma}\label{lem: expanding basic}
Let $\beta,L$ be the constants given in \eqref{eqn: expanding lower bound} and \eqref{eqn: expanding upper bound}.
For every $z\in \Lambda'$, we have
	\begin{enumerate}
		\item\label{ppt: t'} $L^{-1}\leqslant t(1,z)'(z)\leqslant L $;
		\item\label{ppt: T' good} $t(n_m(z),z)'(z) \geqslant \beta^m$ for every positive integer $m.$
    \item\label{ppt: T' interm} $t(n_m(z),z)'(z) \geqslant \beta^{m - k} t(n_k(z), z)'(z)$ for all positive integers $m \geq k$.
  \end{enumerate}
\end{lemma}
\begin{proof}
	\eqref{ppt: t'} follows directly from the definition in \eqref{eqn: expanding def} together with the upper bound \eqref{eqn: expanding upper bound}.
  
	For \eqref{ppt: T' good}, we proceed by induction.
  For $m=1$, the goal is to show 
  \[\forall z \in \Lambda',\quad t(n(z),z)'(z) \geq \beta > 1.\]
  We distinguish two cases.
	If $z\notin\ol{U_{\NE}}$ then $n(z)=1$ and $t(n(z),z)=t(1,z)= f_{y_i}$ for some $i$ by \eqref{eqn: expanding def}.
    The goal follows from \eqref{eqn: expanding lower bound}.
	If $z\in\ol{U_{\NE}}$ then by Lemma \ref{lem: parabolic expanding} and the definition of $n(z),$ there exist $x\in \NE \cap\Lambda$ and $\sigma\in\lb{+,-}$ such that
	\[t(n(z),z)= t(1,T^{n(z)-1}z)\circ t(n(z)-1,z)=g_{x,\sigma}^{p_{x,\sigma}+n(z)-1}. \]
	Since $z\in g_{x,\sigma}^{-(n(z)-1)}(J_x^\sigma)$ by Definition \ref{def: Acceleration times}, we have $t(n(z),z)'(z)\geqslant 2$ by Lemma \ref{lem: DKN expand escape}.
  	This implies the goal $t(n(z),z)'(z) \geqslant\beta$.
	Then the induction step of \eqref{ppt: T' good} is immediate from a combination of \eqref{eq:defn_m} with \eqref{eqn: expanding def 2} and the base case.

  Finally, for \eqref{ppt: T' interm}, again taking derivatives in \eqref{eqn: expanding def 2}, write $y = T^{n_k(z)} z$.
  Then
  \[
    t(n_m(z),z)'(z) = t(n_m(z) - n_k(z), y)' (y) \cdot t(n_k(z), z)'(z).
  \]
  By a simple induction, we see that $n_m(z) = n_k(z) + n_{m-k}(y)$. This combined with \eqref{ppt: T' good} gives \eqref{ppt: T' interm}.
\end{proof}

The following will be useful when combined with distortion control.
\begin{lemma}
	There exists $C'>0$ such that for every $z\in\Lambda'$ and $n=n(z)$ the first acceleration time of $z,$ we have
	\[\sum_{\ell=0}^{n-1} t(\ell,z)'(z)\leqslant C' \, t(n,z)'(z). \]
\end{lemma}

\begin{proof}
	Since $G\sbs \Diff_+^2(\SS^1),$ there exists $M>0$ such that for all $f\in\cT$ (see \eqref{eqn: expanding elements}) and subinterval $I\sbs\SS^1$,
	\[\varkappa(f,I)\leqslant M \, |I|. \]
	If $n(z)=1,$ then either $z\notin\ol{U_{\NE}}$ or $z\in J_x^\sigma$ for some $x\in\NE\cap\Lambda$ and $\sigma\in\lb{+,-},$ and hence the inequality is trivial with $C'=1$.
	
	Otherwise, we have $n(z)\geqslant 2 $ and $z\in I_x^\sigma\sm J_x^\sigma$ for some $x\in\NE \cap \Lambda$ and $\sigma\in\lb{+,-}.$
	Then by Definition \ref{def: Acceleration times}, we have $z\in g_{x,\sigma}^{-(n-1)}(J_x^\sigma).$
	Moreover, by Lemma \ref{lem: parabolic expanding}, $t(\ell,z)=g_{x,\sigma}^\ell$ for $0\leqslant \ell\leqslant n-1.$
	We denote $J_\ell=J_{\ell,x,\sigma}\defeq g_{x,\sigma}^{-\ell}(J_x^\sigma)$ for every $\ell\geqslant 0.$
	By \eqref{eqn:vkfkI0}, we have
	\[\vk(g_{x,\sigma}^\ell,J_{n-1})\leqslant M\sum_{i=1}^{\ell}|J_{n-i}|\leqslant M,\quad \forall 0\leqslant \ell\leqslant n-1. \]
	Therefore,
	\[\sum_{\ell=0}^{n-1}(g_{x,\sigma}^\ell)'(z)\leqslant e^M \sum_{\ell=0}^{n-1}\frac{| g_{x,\sigma}^\ell(J_{n-1})|}{|J_{n-1}|}=e^M \sum_{\ell=0}^{n-1}\frac{|J_{n-1-\ell}|}{|J_{n-1}|}\leqslant \frac{e^M}{|J_{n-1}|}.\]
	On the other hand, by \eqref{eqn: distortion sum},
	\[\vk(g_{x,\sigma}^{n+p_{x,\sigma}-1},J_{n-1})\leqslant \vk(g_{x,\sigma}^{p_{x,\sigma}}, J_0)+\vk(g_{x,\sigma}^{n-1},J_{n-1})\leqslant 2M,\]
	and hence,
	\[(g_{x,\sigma}^{n+p_{x,\sigma}-1})'(z)\geqslant e^{-2M}\frac{| g_{x,\sigma}^{n+p_{x,\sigma}-1}(J_{n-1})|}{|J_{n-1}|}=e^{-2M}\frac{|g_{x,\sigma}^{p_{x,\sigma}}J_0|}{|J_{n-1}|}. \]
	That is
	\[\sum_{\ell=0}^{n-1}(g_{x,\sigma}^\ell)'(z)\leqslant \frac{e^{3M}}{|g_{x,\sigma}^{p_{x,\sigma}}J_0|} (g_{x,\sigma}^{n+p_{x,\sigma}-1})'(z)= \frac{e^{3M}}{|g_{x,\sigma}^{p_{x,\sigma}}J_x^\sigma |} t(n,z)'(z). \]
	Since there are only finitely many choices of $(x,\sigma),$ we can choose a uniform constant $C'>0$ such that the inequality holds for all $z\in\Lambda'$ and $n=n(z).$
\end{proof}

The same estimate holds at every acceleration time.
\begin{lemma} \label{lem: sum t' at Acc}
	There exists $C>0$ such that for every $z\in\Lambda'$ and every acceleration time $n\in\Acc(z),$ we have
	\[\sum_{\ell=0}^{n-1} t(\ell,z)'(z)\leqslant C \, t(n,z)'(z). \]
\end{lemma}

\begin{proof}
	Let $0< n_1(z)<n_2(z)<\cdots<n_m(z)=n$ be all acceleration times in $\Acc(z)$ which are at most $n.$
	Let $n_0=0.$
	We abbreviate $n_i(z)$ to $n_i.$
	Note that for $n_i\leqslant \ell \leqslant n_{i+1}-1,$ we have by \eqref{eqn: expanding def 2} and the chain rule,
	\[t(\ell,z)'(z)=t(\ell-n_i,T^{n_i}z)'(T^{n_i}z)\cdot t(n_i,z)'(z).\]
	Note that $n(T^{n_i}z) = n_{i+1} - n_i$. Hence by the previous lemma applied to $T^{n_i}z$,
	\[\sum_{\ell=n_i}^{n_{i+1}-1}t(\ell,z)'(z)\leqslant C'\cdot t(n_{i+1}-n_i,T^{n_i}z)'(T^{n_i}z)\cdot t(n_i,z)'(z)=C'\cdot t(n_{i+1},z)'(z).\]
Summing over $i$ and using Lemma \ref{lem: expanding basic} \eqref{ppt: T' good} and \eqref{ppt: T' interm},
\[\sum_{\ell=0}^{n-1} t(\ell,z)'(z) \leqslant C'\sum_{i=0}^{m-1} t(n_{i+1},z)'(z) \leq C' \sum_{i=0}^{m-1} \beta^{-(m - i - 1)} t(n_{m},z)'(z).\]
This proves the desired estimate with $C = \frac{C'}{1-\beta^{-1}}$.
\end{proof}

\begin{lemma}\label{lem: acceleration}
	There exists $\ve_0>0$ such that for every $z\in\Lambda'$ and $n\in\Acc(z),$ the acceleration map $t(n,z)$ satisfies 
    \[\vk(t(n,z),B(z,\ve_0 t(n,z)'(z)^{-1}))\leqslant \log 2 . \]
\end{lemma}

\begin{proof}
  Let $M = M(\cT)$ be as in Proposition~\ref{prop: distortion estimate} and $C$ as in Lemma~\ref{lem: sum t' at Acc}.
  Take $\ve_0 = \frac{\log 2}{4 C M}$ and $r=\ve_0 t(n,z)'(z)^{-1}$.
	By Lemma~\ref{lem: sum t' at Acc}, we have
	\[4 M r \sum_{\ell=0}^{n-1} t(\ell,z)'(z) \leq 4 M r C t(n,z)'(z) \leq \log 2.\]
	Hence, Proposition~\ref{prop: distortion estimate} implies that
	 \[\vk(t(n,z) ,B(z,\ve_0 t(n,z)'(z)^{-1}))\leqslant 4 M r \sum_{\ell=0}^{n-1} t(\ell,z)'(z) \leqslant\log 2.\qedhere \]
\end{proof}

\begin{definition}\label{def: rnz}
    Let $\ve_0>0$ be the constant given by Lemma \ref{lem: acceleration}.
    We define the expanding radius $r_n(z)\defeq \ve_0 [t(n,z)'(z)]^{-1}$ 
    for each $z\in\Lambda'$ and $n\geqslant 0.$
\end{definition}

We gather useful properties of $r_n(z)$ in the following lemma.
\begin{lemma}\label{lem: rnz}
	Let $\ve_0$ be the constant given by Lemma \ref{lem: acceleration} and let $z\in \Lambda'.$
	\begin{enumerate}[\indent(1)]
		\item\label{ppt: rnz decrease} The sequence of expanding radii at acceleration times $(r_n(z))_{n\in\Acc(z)}$ decreases to $0$ as $n\to+\infty$.
    \item\label{ppt: rnz bounded} The initial radius $r_0(z)=\ve_0$ and the ratio $r_n(z) / r_{n + 1}(z)$ is bounded above and below by positive constants independent of $n \in \NN$ and $z \in \Lambda'$.
    \end{enumerate}
	Moreover, if $n\in\Acc(z)$ is an acceleration time, then 
	\begin{enumerate}[\indent(1)]
        \setcounter{enumi}{2}
        \item\label{ppt: rnz vk} $\varkappa(t(n,z),B(z,r_n(z)))\leqslant \log 2$;
		\item\label{ppt: rnz derivative} $\frac{1}{2}\ve_0 r_n(z)^{-1}\leqslant t(n,z)'|_{B(z,r_n(z))}\leqslant 2\ve_0 r_n(z)^{-1}$;
		\item\label{ppt: rnz interval} $B( t(n,z)z,\ve_0/2)\sbs t(n,z)B(z,r_n(z))\sbs B( t(n,z)z,2\ve_0).$
	\end{enumerate}
\end{lemma}
\begin{proof}
	For $n = 0$, $t(0,z)= \Id$ and hence $r_0(z)=\ve_0$ by definition.
    By Lemma \ref{lem: expanding basic} \eqref{ppt: t'}, for any $n \in \NN$, we have $r_n(z)/r_{n+1}(z) = t(1,T^n z)'(T^n z) \in [L^{-1},L]$.
	By Lemma \ref{lem: expanding basic} \eqref{ppt: T' good} and \eqref{ppt: T' interm}, $(r_n(z))_{n\in\Acc(z)}$ decreases to $0$ as $n\to+\infty$.
	By Lemma \ref{lem: acceleration}, we have $\vk(t(n,z),B(z,r_n(z)))\leqslant\log 2$; the remaining two items then follow.
\end{proof}

\subsection{Ping-pong construction for semigroup actions on the circle}
For our purposes, we will construct a pair of elements in a sub-semigroup of $\Diff_+^{1 +\alpha}(\mathbb{S}^1)$ that satisfy certain structural conditions.
We refer to such a pair as a perfect ping-pong pair.
Roughly speaking, a perfect ping-pong pair consists of two diffeomorphisms that generate several independent ping-pong dynamics on disjoint subintervals of the circle.
These pairs exhibit strong hyperbolicity and produce particularly explicit dynamics, thereby providing a large ``uniformly hyperbolic'' subsystem that approximates the original action.
This approximation will be essential for establishing the variational principle of critical exponents, Proposition \ref{prop: variational principle of critical exponents}.

\begin{definition}\label{def: perfect ping-pong pair}
Let $q$ be a positive integer.
A pair $(h_1,h_2)\sbs\Diff_+^1(\SS^1)$ is called a \textit{($q$-)perfect ping-pong pair} if there exist subsets $U_1^+,U_2^+,U_1^-,U_2^-\sbs\SS^1$ satisfying the following conditions:
	\begin{enumerate}[\indent(a)]
		\item Each $h_i$ has exactly $2q$ fixed points, all of which are hyperbolic.
		\item For every attractor $a_1 \in \SS^1$ of $h_1$ there is an attractor $a_2 \in \SS^1$ of $h_2$ such that either $I = [a_1,a_2]$ or $I = [a_2,a_1]$ satisfies $I \cap \Fix(h_i) = \{a_i\}$, for $i \in \{1,2\}$.
        Moreover, the same holds for repellors. 
 		\item The set $U_1^+$ (resp. $U_2^+,U_1^-,U_2^-$) is a union of $q$ disjoint open intervals, each of which is an open neighborhood of an attractor of $h_1$ (resp. $h_2,h_1^{-1},h_2^{-1}$).
		\item The closures of the sets $U_1^+,U_2^+,U_1^-,U_2^-$ are pairwise disjoint.
		\item The collection $\lb{U_i^\pm}$ forms a ping-pong partition for $h_1$ and $h_2$ in the following sense: 
		\[h_1(\SS^1\sm U_1^-)\sbs U_1^+,~h_2(\SS^1\sm U_2^-)\sbs U_2^+,~h_1^{-1}(\SS^1\sm U_1^+)\sbs U_1^-,~h_2^{-1}(\SS^1\sm U_2^+)\sbs U_2^-.\]
		\item The maps $h_i$ and their inverses uniformly contract outside the corresponding cones:
		\[h_1'|_{\SS^1\sm U_1^-}<1,~h_2'|_{\SS^1\sm U_2^-}<1,~(h_1^{-1})'|_{\SS^1\sm U_1^+}<1,~(h_2^{-1})'|_{\SS^1\sm U_2^+}<1.\]
	\end{enumerate}
We call the sets $U_{i}^\pm$ the \textit{ping-pong cones} of the pair $(h_1,h_2).$
\end{definition}
\begin{remark}
Condition (b) is crucial and can be equivalently formulated as follows:
 $\Fix(h_1)$ and $\Fix(h_2)$ are disjoint and we may label the attractors in $\Fix(h_1) \cup \Fix(h_2)$ as $a_1, \dotsc, a_{2q}$ and the repellors as $r_1, \dotsc, r_{2q}$ so that in cyclic order we have
 \[a_1 \prec a_2 \prec r_1 \prec r_2 \prec a_3 \prec a_4 \prec \dotsi \prec a_{2q - 1} \prec a_{2q} \prec r_{2q - 1} \prec r_{2q} \prec a_1.\]
Under this labelling, each interval of the form $[a_{2i-1},a_{2i}]$ is preserved by both $h_1$ and $h_2$, and $(h_1,h_2)$ acts on it with ping-pong dynamics.
Likewise, each interval of the form $[r_{2i-1},r_{2i}]$ is preserved by both $h_1^{-1}$ and $h_2^{-1}$, and on it $(h_1^{-1}, h_2^{-1})$ exhibits ping-pong dynamics.
\end{remark}

There are two key features of a perfect ping-pong pair.
The first is that each element in the pair has only hyperbolic fixed points.
The appearance of hyperbolic fixed points in group actions on the circle goes back to Sacksteder’s theorem \cite{Sa65} for $C^2$ pseudogroups.
Later, Deroin, Kleptsyn and Navas \cite[Théorème F]{DKN07} proved that any subgroup of $ \Diff_+^1(\mathbb{S}^1) $ which does not preserve a probability measure on $\mathbb{S}^1$ must contain an element with a hyperbolic fixed point.

The second feature is that a perfect ping-pong pair generates a free subgroup of  $\Diff_+^1(\mathbb{S}^1) .$ 
In this sense, the construction of perfect ping-pong pairs can be regarded as a $C^1$-variant of Margulis’s theorem (see \cite{Mar00}), also known as Ghys’s conjecture, which asserts the existence of free subgroups in certain groups of circle homeomorphisms.

The following theorem establishes the existence of perfect ping-pong pairs in semigroups of $C^{1+\alpha}$ circle diffeomorphisms.
Its proof relies on the study of random walks on the circle developed in Part I; see \cite[Proposition 5.8]{HJX1}.
\begin{theorem}\label{thm: generate ping-pong pair}
Let $T\sbs \Diff_+^{1+\alpha}(\SS^1)$ be a finitely generated semigroup without invariant probability measures on $\SS^1.$ 
Let $\Delta$ be a minimal set of $T.$ 
Given a sufficiently short open interval $I\sbs \SS^1$ that intersects $\Delta$, there exists a perfect ping-pong pair $(h_1,h_2)\sbs T$, with ping-pong cones $(U^\pm_i)$, such that both $U_1^+$ and $U_2^+$ have a unique connected component contained in $I$, and these components have nonempty intersections with $\Delta$.
\end{theorem}

\subsection{Different notions of dimensions} 

Two notions of dimension will be used in this paper: the box dimension (also known as the Minkowski dimension) and the Hausdorff dimension. 
Let $A\sbs \SS^1$ be a subset.
The \textit{upper} and \textit{lower box dimension} of $A$ are defined by
\[\ol{\dim}_{\mr B}A\defeq\limsup_{r\to 0^+}-\frac{\log \mr{Cov}(A,r)}{\log r}\quad\text{and}\quad \ul{\dim}_{\mr B}A\defeq\liminf_{r\to 0^+}-\frac{\log \mr{Cov}(A,r)}{\log r},\]
where $\mr{Cov}(A,r)$ is the smallest number of intervals of length $r$ needed to cover  $A.$
If these two quantities agree, their common value is called the \textit{box dimension} of $A$ and is written $\dim_{\mr B}A.$
A useful basic property is that $\dim_{\mr B}A=\dim_{\mr B}\ol A,$ where $\ol A$ is the closure of $A$ (see for example \cite[Proposition 3.4]{Fal14}).

For every $s>0,$ the \textit{$s$-dimensional Hausdorff outer measure} of $A$ is defined by
\[\cH^s(A)\defeq\lim_{r\to 0^+}\cH_r^s(A),\]
where
\[\cH_r^s(A)\defeq\inf\lb{\sum_{i=1}^\infty \sb{\diam U_i}^s: A\sbs\bigcup_{i=1}^\infty U_i,~ \diam U_i\leqslant r}.\]
The \textit{Hausdorff dimension} of $A$ is the unique $\alpha\geqslant 0$ such that $\cH^s(A)=0$ for $s>\alpha$ and $\cH^s(A)=\infty$ for $s<\alpha.$
We denote this value by $\dimH A=\alpha.$
It follows directly from the definition that 
\begin{equation}\label{eqn: ineq box Hausdorff}
 	\ol{\dim}_{\mr B}A\geqslant \ul{\dim}_{\mr B}A\geqslant \dimH A.
\end{equation}

The following lemma is a handy tool to relate the dimension of a fractal set to covering numbers.
See \cite[Lemma 5.13]{HJX1}. 
Recall that for a sequence of subsets $\lb{A_n\sbs \SS^1: n\geqslant1},$ the limit superior of the sequence is
\[\limsup_{n\to+\infty} A_n=\bigcap_{N\geqslant 0}\bigcup_{n\geqslant N} A_n. \]

\begin{lemma}
\label{lem: cover and dim}
For each $n \in \NN$, let $\cE_n$ be a collection of intervals in $\SS^1$ of length at most $e^{- n}$ and let $E = \limsup_{n \to +\infty} \bigcup_{J \in \cE_n} J$.
Let $\cE_n' \sbs \cE_n$ be a maximal subcollection of pairwise disjoint intervals.
Then
\[\limsup_{n \to +\infty} \frac{1}{n} \log \#\cE_n' \geq \dimH E.\]
\end{lemma}

A common approach to studying the dimension of a set is to analyze the dimensions of measures supported on this set.
Let $\nu$ be a probability measure on $\SS^1.$
We say that $\nu$ is \textit{exact dimensional} if there exists a constant $\alpha\geqslant 0$ such that  
\begin{equation}\label{eqn: exact dim}
	\lim_{r\to 0^+}\frac{\log\nu(B(x,r))}{\log r}=\alpha,\quad\ae[\nu\mr{-}] x.
\end{equation}
In this case, $\alpha$ is called the \emph{exact dimension} of $\nu$ and is denoted by $\dim\nu.$
We also recall that the Hausdorff dimension of $\nu$ is defined by
\[\dimH\nu=\inf\lb{\,\dimH E: \nu(E)>0\,}.\]
By a classical result of Young \cite{Y82}, if $\nu$ is exact dimensional, then $\dimH\nu=\dim\nu$.
The following lemma then follows immediately.
\begin{lemma}\label{lem: measure Hausdorff}
	Let $A\sbs\SS^1$ be a Borel subset.
	If there exists an exact dimensional probability measure $\nu$ with $\nu(A)>0$, then $\dimH A\geqslant\dim \nu.$
\end{lemma}
The box dimension of a subset can also be bounded from above using measures that satisfy a lower mass estimate, see for example \cite[Theorem 5.7]{Mat95}.
\begin{lemma}\label{lem: measure box}
	Let $\nu$ be a Borel probability measure on $\SS^1$ and let $A\sbs \SS^1$ be a Borel subset. 
    If there exist constants $c>0,\alpha>0$ such that for all sufficiently small $r>0$, $\nu(B(x,r))\geqslant c\cdot r^\alpha$ for every $x\in A,$ then $\ol{\dim}_{\mr B}A\leqslant \alpha.$
\end{lemma}

\subsection{Random walks on the circle}

To study group (or semigroup) actions on the circle, we consider the random walks induced by these actions.
In this subsection, we recall some notions related to random walks on the circle that will be used in the subsequent discussion.
A more detailed study of smooth random walks on the circle, along with the associated dimension theory of stationary measures, is presented in the companion paper \cite{HJX1}.

For the purposes of this paper, we consider random walks on the circle generated by $C^2$-diffeomorphisms. 
Accordingly, we take $\mu$ to be a finitely supported probability measure on $\Diff^2_+(\SS^1).$
Recall from Definition \ref{def: stationary and conformal} that a $\mu$-stationary measure is a probability measure $\nu$ on $\SS^1$ satisfying $\mu*\nu=\nu.$
A $\mu$-stationary measure $\nu$ is said to be \textit{ergodic} if it cannot be written as a nontrivial convex combination of distinct $\mu$-stationary measures.
There are two key quantities of interest associated with such random walks:
\begin{itemize}
	\item \textit{Lyapunov exponent}: Let $\nu$ be an ergodic $\mu$-stationary measure. 
    The Lyapunov exponent for $(\mu,\nu)$ is
	\begin{equation}\label{eqn: LE def}
		\lambda(\mu,\nu) \defeq\iint \log f'(x) \dd \mu(f)\dd\nu(x).
	\end{equation}
	\item \textit{Entropy}: Here we consider the random walk entropy of $\mu$. 
    For a positive integer $n$, let $\mu^{*n}$ denote the $n$-fold convolution of $\mu$, which is again a finitely supported probability measure on $\Diff_+^2(\SS^1).$ The \textit{Shannon entropy} of $\mu^{*n}$ is given by 
	\[H(\mu^{*n}) = -\sum_{g\in\supp\mu^{*n}} \mu^{*n}(g)\log\mu^{*n}(g).\]
	The \textit{random walk entropy} of $\mu$ is then defined to be the limit
	\[h_{\mr{RW}}(\mu)\defeq \lim_{n\to\infty}\frac{1}{n}H(\mu^{*n}).\]
	The existence of the limit is guaranteed by sub-additivity, as shown in \cite{KV83}. 
	In this paper, we only need to compute the random walk entropy in the following special case. 
	The proof is elementary and therefore omitted.
	\vspace{0.2cm}
	\begin{lemma}\label{lem: random walk entropy free}
		Let $\mu$ be a uniform measure on a finite set $\cS\sbs\Diff_+^2(\SS^1)$ and assume that $\cS$ freely generates a free semigroup in $\Diff_+^2(\SS^1)$. 
		Then $h_{\mr{RW}}(\mu)=\log\#\cS.$
	\end{lemma}
\end{itemize}

Let $T_\mu$ denote the semigroup generated by the support of $\mu.$
The following theorem, established in Part I \cite{HJX1}, provides a dimension formula for ergodic $\mu$-stationary measures.
The first part, concerning the negativity of the Lyapunov exponent, follows directly from the Avila–Viana invariance principle \cite{AV10}.
A related result in this direction is due to Malicet \cite{Mal17}, who generalized the invariance principle to random walks by circle homeomorphisms.
\begin{theorem}\label{thm: stationary measure dimension}
	Assume that $T_\mu$ preserves a closed interval $I\sbs \SS^1$ and there is no common $T_\mu$-invariant probability measure on $I.$
	Let $\nu$ be an ergodic $\mu$-stationary probability measure supported on $I.$
	Then the Lyapunov exponent $\lambda(\mu,\nu)$ is negative.
	
	Moreover, $\nu$ is exact dimensional and one of the following holds:
	\begin{enumerate}[\indent(1)]
		\item\label{alt: dimension formula} either $\dim\nu=\frac{h_{\mr{RW}}(\mu)}{|\lambda(\mu,\nu)|}$;
		\item\label{alt: locally nondiscrete} or there exists a closed subinterval $J\sbs I$ with $J^\circ \cap \supp\nu\ne\vn$ and two sequences $(g_n)_{n\geqslant 1},(f_n)_{n\geqslant 1}\sbs T_\mu$ with $g_n\ne f_n$ for all $n$, such that $f_n(J) \sbs g_n(I)$, and the maps $(g_n^{-1}f_n)|_J$ converge to the identity map on $J$ in the $C^1$-topology.
	\end{enumerate}
\end{theorem}

\section{Dynamical critical exponents}
\label{sec: 3}

\subsection{Dynamical critical exponents on subsets}
Let $G\sbs\Diff_+^1(\SS^1)$ be a subgroup. 
Recall the definition of the dynamical critical exponent of a set from Definition \ref{def: DCE on sets}.
The next lemma collects some basic properties of dynamical critical exponents.

\begin{lemma}\label{lem: union critical exp}
Let $\Delta\sbs\SS^1$ be a subset.
Then: 
\begin{enumerate}
\item\label{ppt: DCE monotone} If $\Delta'$ is a subset of $\Delta$, then $\delta(G,\Delta') \leq \delta(G,\Delta)$.
\item\label{ppt: DCE union} If $\lb{\Delta_i : i\in\cI}$ is a family of subsets of $\Delta$ with $\Delta=\bigcup_{i\in\cI}\Delta_i$, then
\[\delta(G,\Delta)=\sup_{i\in\cI}\delta(G,\Delta_i).\]
\item\label{ppt: DCE closure} If $\ol\Delta$ denotes the closure of $\Delta$, then $\delta(G,\Delta)=\delta(G,\ol\Delta).$
\item\label{ppt: DCE invariance}  For every $g\in G$, $\delta(G,\Delta)=\delta(G,g\Delta)$.
\item\label{ppt: DCE orbit} If $G\Delta$ denotes $\bigcup_{g\in G} g\Delta$, then
\[\delta(G,\Delta)=\delta(G,G\Delta)= \delta(G,\ol {G\Delta}).\]
\end{enumerate}
\end{lemma}
\begin{proof}
The items \eqref{ppt: DCE monotone} and \eqref{ppt: DCE invariance} are immediate from the definition.

For \eqref{ppt: DCE union}, the inequality $\delta(G,\Delta)\geq \sup_{i\in\cI}\delta(G,\Delta_i)$ is immediate since $\Delta_i\sbs\Delta.$ 
Conversely, fix $\ve>0$ and let $\cK=\cK(\ve)$ be a finite $\ve/2$-dense set of $\Delta.$
For every $x\in\Delta,$ there exists $y\in\cK$ such that $B(x,\ve)\sps B(y,\ve/2).$ 
Hence
\[\limsup_{n\to\infty}\frac{1}{n}\log\# \lb{g\in G:\exists x\in\Delta,g'|_{B(x,\ve)}\geq e^{-n}}\leq\max_{j\in\cJ}\delta(G,\Delta_j)\leq\sup_{i\in\cI}\delta(G,\Delta_i) ,\]
where $\cJ=\cJ(\ve)$ is a finite subset of $\cI$ such that $\cK\sbs\bigcup_{j\in\cJ}\Delta_j.$
Letting $\ve\to0$ gives \eqref{ppt: DCE union}.

For \eqref{ppt: DCE closure}, it suffices to show $\delta(G,\ol\Delta)\leqslant \delta(G,\Delta).$
Given $\ve>0$ and $x\in\ol\Delta$, there exists $y\in \Delta$ such that $B(x,\ve)\sps B(y,\ve/2).$
Then for every $\ve>0$ and $n\geqslant 1,$ we have
\[\# \lb{g\in G:\exists x\in\ol\Delta,g'|_{B(x,\ve)}\geq e^{-n}}\leq \# \lb{g\in G:\exists y\in\Delta,g'|_{B(y,\ve/2)}\geq e^{-n}} .\]
This implies $\delta(G,\ol\Delta)\leqslant\delta(G,\Delta)$ by the definition of $\delta(G,\cdot).$

Finally, item \eqref{ppt: DCE orbit} is an immediate consequence of \eqref{ppt: DCE union}, \eqref{ppt: DCE closure} and \eqref{ppt: DCE invariance}.
\end{proof}

The following is a corollary of part \eqref{ppt: DCE orbit} of Lemma \ref{lem: union critical exp}.
Recall that $\delta(G,x)=\delta(G,\lb{x})$ is the pointwise dynamical critical exponent at $x\in\SS^1.$

\begin{corollary}\label{cor: pointwise DCE}
    Assume that $G$ has no finite orbits, so that it preserves a unique minimal set $\Lambda.$
	Then for every $x\in\SS^1,$ $\delta(G,x)=\delta(G,\ol{Gx})\geqslant \delta(G,\Lambda).$
	In particular, for every nonempty subset $\Delta\sbs\SS^1,$ $\delta(G,\Delta)\geqslant\delta(G)$.
\end{corollary}
In view of this, if $G$ acts minimally on the circle, then $\delta(G,\Delta) = \delta(G)$ for every nonempty subset $\Delta \sbs \SS^1$. 
Thus, $\delta(G,\Delta)$ provides additional information only when $G$ preserves an exceptional minimal set.

The following is a strengthened form of the fact that $\delta(G,x)=\delta(G)$ for $x\in\Lambda$: it shows that in the definition of $\delta(G)$, the value $\delta(G)$ is already obtained at a fixed positive radius.
\begin{lemma}\label{lem: realize DCE at epsilon_1}
    Assume that $G$ is finitely generated and has no common invariant measure on $\SS^1,$ so that it preserves a unique minimal set $\Lambda$ by Lemma \ref{lem: invariant measure and finite orbit}.
	Let $\ve_1>0$ be the constant given by Lemma \ref{lem: global contraction}.
	Then for every $x\in\Lambda,$ we have 
	 \[ \delta(G)=\limsup_{n \to +\infty} \frac{1}{n} \log\# \lb{ g\in G : g'|_{B(x,\ve_1)}\geq e^{-n} }. \]
\end{lemma}
\begin{proof}
By the definition of $\delta(G)$, for every $\delta'<\delta(G)$ there is $\ve > 0$ such that 
\[ \limsup_{n \to +\infty} \frac{1}{n}\log \# \lb{\, g\in G:\exists y \in \Lambda,\, g'|_{B(y,\ve)}\geq e^{-n}\,} \geq \delta'.\]
Choose a finite $\frac{\ve}{2}$-dense subset $\{x_1,\dotsc,x_k\}\sbs \Lambda$.
Then for any $y \in \Lambda$, one has $B(x_i,\ve/2) \sbs B(y,\ve)$ for some $i$. 
By the pigeonhole principle, there exists $i \in \{1,\dotsc,k\}$ such that
\[ \limsup_{n \to +\infty} \frac{1}{n} \log\# \lb{\, g\in G: g'|_{B(x_i,\ve/2)} \geq e^{-n}\,} \geq \delta'.\]
By Lemma \ref{lem: global contraction}, there exists $f \in G$ such that $fB(x,\ve_1) \sbs B(x_i,\ve/2)$.
For any $g\in G$, if $ g'|_{B(x_i,\ve/2)}\geq e^{-n}$, then $(gf)'|_{B(x,\ve_1)}\geqslant c\cdot e^{-n}$ for some $c>0$ depending only on $f.$
Therefore,
\[ \limsup_{n \to +\infty} \frac{1}{n} \log\# \lb{ g\in G: g'|_{B(x,\ve_1)} \geq e^{-n}}\geq \delta'.\]
Letting $\delta'\to\delta(G)$, the lemma follows.
\end{proof}

\paragraph{Pointwise dynamical critical exponents.}
By Lemma \ref{lem: union critical exp} \eqref{ppt: DCE union}, $\delta(G,\Delta)=\sup_{x\in\Delta}\delta(G,x)$ for any subset $\Delta\sbs\SS^1.$
Thus, the collection of pointwise critical exponents governs the behavior of dynamical critical exponents on all subsets.

Let $G\sbs\Diff_+^2(\SS^1)$ preserve an exceptional minimal set $\Lambda.$
We recall the concept of \textit{wandering points}, which allows an equivalent description of pointwise critical exponents.

\begin{definition}\label{def: wandering}
A point $x\in\SS^1$ is called \textit{wandering} if there exists an open neighborhood $B(x,\ve)$ such that the images $gB(x,\ve)$ are pairwise disjoint for $g\in G.$
\end{definition}

\begin{proposition}\label{prop: wandering}
	Let $x\in\SS^1$ be a wandering point.
	For $s>0,$ consider the Poincar\'e series
	\[P(x,s)\defeq \sum_{g\in G}g'(x)^s. \]
	Assume further that $G$ is finitely generated.
	Then the critical exponent of convergence of $P(x,s)$ is exactly $\delta(G,x)$.
	Moreover, $P(x,1)$ converges and hence $\delta(G,x)\leq 1.$
\end{proposition}
\begin{proof}
	Let $B(x,\ve)$ be a neighborhood of $x$ such that the intervals $gB(x,\ve)$ are pairwise disjoint for $g\in G.$ 
	Let $\cG$ be a finite generating set of $G$ and let $\cS=\cG\cup\lb{\Id}\cup\cG^{-1}$ be the associated symmetric generating set.
	Define
	\[M=M(\cS)=\max\lb{\max_{g\in\cS}\|g'\|_{C^0},\max_{g\in\cS}\|\log g'\|_{\mr{Lip}}}.\]
	By \eqref{eqn:vkfkI0}, for every $g\in G,$   
	\[\vk(g,B(x,\ve))\leqslant M\sum_{h\in G}|hB(x,\ve)|\leqslant M. \]
	Hence for every $\ve'\leq \ve,$
    \[\forall n \in \NN,\quad g'|_{B(x,\ve')} \geq e^{-n} \text{ implies } g'(x) \geq e^{-n} \text{ implies } g'|_{B(x,\ve')} \geq e^{-(n + M)},\]
    leading to
	\[\limsup_{n\to+\infty}\frac{1}{n}\log\#\lb{g\in G:g'|_{B(x,\ve')}\geq e^{-n}}=\limsup_{n\to+\infty}\frac{1}{n}\log\#\lb{g\in G:g'(x)\geq e^{-n}},\]
	and the right-hand side is precisely the critical exponent of convergence of $P(x,s)$.
    This proves the desired equality. 
	Finally, 
    \[P(x,1)=\sum_{g\in G} |g'(x)| \leqslant e^M\sum_{g\in G}\frac{|gB(x,\ve)|}{|B(x,\ve)|}\leqslant (2\ve)^{-1}e^M ,\]
    and hence $\delta(G,x)\leqslant 1.$
\end{proof}

Recall that the \textit{$\omega$-limit set} of a point $x$ for the $G$-action is 
\[\omega(x)\defeq\lb{y:\exists(g_n)\sbs G\text{ pairwise distinct and }g_nx \to y}.\]
A point $x$ is said to be \textit{recurrent} if $x\in\omega(x).$

\begin{lemma}\label{lem: wandering}
Let $G$ and $x$ be as above. 
Consider the following conditions.
	\begin{enumerate}[\indent (1)]
		\item\label{equi: 1} $x$ is wandering;
		\item\label{equi: 2} $x$ is nonrecurrent;
		\item\label{equi: 3} $x\notin \Lambda$ and $x$ is not fixed by any nontrivial element in $G.$
	\end{enumerate}
We have \eqref{equi: 1} $\Rightarrow$ \eqref{equi: 2} $\Rightarrow$ \eqref{equi: 3}.
If $x\notin \Lambda,$ let $J$ be the connected component of $\SS^1\sm\Lambda$ containing $x$. 
The reverse implication $\eqref{equi: 3}\Rightarrow\eqref{equi: 1}$ holds whenever $\Stab_G(J)$ is trivial or infinite cyclic; in particular, it holds if $G\sbs\Diff_+^\omega(\SS^1)$.
\end{lemma}
\begin{proof}
	\eqref{equi: 1} implies \eqref{equi: 2} directly from the definition. 

  Since $G$ acts minimally on $\Lambda,$ a nonrecurrent point cannot lie in $\Lambda.$
	Moreover, if $x$ is fixed by some $f\ne\Id\in G$, then taking $g_n=f^n$ gives $f^nx=x$ so $x\in\omega(x) .$
	Thus \eqref{equi: 2} implies \eqref{equi: 3}.
	
	It remains to show that \eqref{equi: 3} implies \eqref{equi: 1}.
  Let $J$ be the connected component of $\SS^1\sm\Lambda$ containing $x.$
	If $G \sbs \Diff_+^\omega(\SS^1)$, by Theorem \ref{thm: real analytic with exceptional} \eqref{fact: Hector}, $\Stab_G(J)=\Stab_G(\partial J)$ is either trivial or infinite cyclic.
	If $\Stab_G(J)$ is trivial, choose  $B(x,\ve)\sbs J.$
	Then the intervals $gB(x,\ve)\sbs gJ$ are pairwise disjoint.
	Otherwise $\Stab_G(J)=\pair{f}$ for some $f$ and $f$ does not fix $x$.
  We can choose an open interval $B(x,\ve)$ such that the intervals $f^nB(x,\ve)$ are pairwise disjoint for $n\in\ZZ.$
	If for some distinct $g_1\ne g_2\in G$ we had $g_1B(x,\ve)\cap g_2B(x,\ve)\ne\vn$, then $g_2^{-1}g_1\in \Stab_G(J)=\pair{f}$ and hence $g_2^{-1}g_1 B(x,\ve)\cap B(x,\ve)\ne\vn,$ contradicting the choice of $B(x,\ve).$
	Therefore, $x$ is wandering.
\end{proof}
\subsection{Groups with parabolic elements}
The main proposition of this subsection is Proposition \ref{prop: parabolic critical exponent}, which provides a strict lower bound on the dynamical critical exponent for subgroups $G\sbs\Diff_+^2(\SS^1)$ that admit a parabolic fixed point on the minimal set.
In particular, when specialized to $G \subset \Diff_+^\omega(\SS^1)$, this implies item (4) in the main theorem.

Let $f : I\to\RR$ be a continuous map, where $I\sbs \RR$ is an interval. 
We say that $f$ is \textit{contracting on $I$} if $|f(J)|<|J| $ for every compact subinterval $J\sbs I.$
We begin by recalling a classical estimate describing the dynamics near a parabolic fixed point.
The following is due to Firmo \cite{Fir88}; see also Eynard-Bontemps--Navas \cite[Theorem B]{EBN}. 

\begin{theorem}\label{thm: Firmo}
	Let $k\geqslant 1$ be an integer, and let $f_1,f_2:[0,1]\to[0,1]$ be real analytic maps with $f_i'(z)>0$ for every $z\in[0,1]$ and $i=1,2$.
	Assume that, for each $i$, the point $0$ is a fixed point of $f_i$ of multiplicity $k+1$, and that $f_i$ is contracting on $]0,1]$.
	Then there exists a $C^1$-diffeomorphism $\psi:[0,1]\to [0,1]$ such that $\psi\circ f_1=f_2\circ \psi.$
\end{theorem}
Using this theorem, the dynamics near a parabolic fixed point can always be reduced to a standard model.
For each $k\geqslant 1,$ consider the real analytic map
\[\phi_k:[0,1]\to [0,1],\quad z\mapsto z(1+z^k)^{-1/k}. \]
Then $0$ is a fixed point of $\phi_k$ of multiplicity $(k+1)$ and $\phi_k$ is contracting on $]0,1].$
Moreover, for every integer $m\geqslant 1 $ and $z\in {]0,1]},$ one has
\begin{equation}\label{eqn: standard parabolic}
	\phi_k^m(z)= (z^{-k}+m)^{-\frac{1}{k}},\quad (\phi_k^m)'(z)= z^{-k-1}(z^{-k}+m)^{-\frac{k+1}{k}}=\sb{\frac{z^{-k}}{z^{-k}+m}}^{\frac{k+1}{k} } .
\end{equation}

\begin{lemma}\label{lem: around parabolic}
	Let $I=[x,y]$ be a closed interval and let $f:I\to I$ be a real analytic map with $f'(z)>0$ for every $z\in I$.
	Assume that $x$ is a fixed point of $f$ of multiplicity $(k+1)$ for some integer $k\geqslant1$, and that $f$ is contracting on ${]x,y]}.$
	Fix $y_0\in {]x,y]}.$
	For each $n\geqslant 1,$ let $y_n=f^n(y_0)$ and $I_n=[y_n,y_{n-1}].$
	Then there exists a constant $C>0$, depending only on $f$ and $y_0$, such that for all $n\geqslant 1$, all  $m\geqslant 0$, and all points $z\in I_n$, 
	\begin{enumerate}[\indent (1)]
		\item\label{ppt: para 1} $d(x,z)\asymp_C n^{-1/k}$ and $d(fz,z)\asymp_C n^{-(k+1)/k}.$
		\item\label{ppt: para 2} $(f^m)'(z) \asymp_{C} \sb{\frac{n}{n+m}}^{(k+1)/k}\geqslant (m+1)^{-(k+1)/k} .$
		\item\label{ppt: para 3} $\vk(f^m,[fz,z])\leqslant C$.
	\end{enumerate}
\end{lemma}
\begin{proof}
	By Theorem \ref{thm: Firmo}, there exists a $C^1$ diffeomorphism $\psi:[x,y_0]\to [0,1]$ such that $\psi\circ f=\phi_k\circ\psi$ on $[x,y_0].$
    Set $C=\sup_{z\in[x,y_0]}\max\lb{|\psi'(z)|,|\psi'(z)|^{-1}}.$

    Let $z\in I_n$ for some $n\geqslant 1$, and denote $z'=\psi(z).$
	Since
	\[\psi(y_n)=\psi\circ f^n(y_0) = \phi_k^n(1)=(n+1)^{-\frac{1}{k} },\]
	we have $z'\in \psi(I_n)=[(n+1)^{-1/k},n^{-1/k}].$
	Therefore,
	\[d(x,z)\asymp_C d(0,z')\asymp n^{-\frac{1}{k} },\text{ and } d(fz,z)\asymp_C d(\phi_k(z'),z')\asymp n^{-\frac{k+1}{k} }. \]
	Using $ (z')^{-k}\in [n,n+1]$ and the formula \eqref{eqn: standard parabolic}, for every $m\geqslant 0$, we have
	\[ (f^m)'(z)\asymp_C (\phi_k^m)'(z')\asymp\sb{\frac{n}{n+m}}^{\frac{k+1}{k} }\geqslant (m+1)^{-\frac{k+1}{k} }.   \]
    
	Finally, for $z_1,z_2\in [fz,z],$ set $\zeta_i=\psi(z_i)^{-k}.$ 
	Then $\zeta_1,\zeta_2\in[n,n+2].$
	By the choice of $C$,
	\begin{align*}
		|\log (f^m)'(z_1)-\log (f^m)'(z_2)|&\leqslant |\log (\phi_k^m)'(\psi(z_1))-\log (\phi_k^m)'(\psi(z_2))|+4\log C\\
		&=\frac{k+1}{k}\left\lvert\log \frac{\zeta_1}{\zeta_1+m}-\log \frac{\zeta_2}{\zeta_2+m} \right\rvert+4\log C\\
		&\leqslant \frac{k+1}{k}\sb{\log \frac{n+2}{m+n+2}-\log \frac{n}{m+n}}+4\log C\leqslant 3+4\log C.
	\end{align*}
	Replacing $C$ with $\max\lb{C, 3+4\log C}$ if necessary, we obtain $\vk(f^m,[fz,z])\leqslant C$.
\end{proof}

For a group with a parabolic fixed point $x\in\Lambda$ of multiplicity $(k+1)$, the lower bound on derivatives given by part \eqref{ppt: para 2} of Lemma \ref{lem: around parabolic} implies that $\delta(G)=\delta(G,z)\geqslant k/(k+1)$ for some $z\in\Lambda$ close to $x.$
To obtain the strict inequality required in Proposition \ref{prop: parabolic critical exponent}, one needs to increase the dynamical critical exponent by incorporating additional elements of the group into the argument.
This step is partly inspired by the classical works of Beardon \cite{Beardon} and of Patterson \cite{Pat} for Fuchsian groups.

\begin{proposition}\label{prop: parabolic critical exponent}
	Let $G\sbs \Diff_+^2(\SS^1)$ be a finitely generated subgroup without invariant probability measures on $\SS^1$, and let $\Lambda$ be its unique minimal set.
	Assume that $G$ contains a nontrivial element $f \in G$ with a fixed point $x\in \Lambda$, and that $f$ is real analytic in a neighborhood of $x.$
	If $x$ is a fixed point of $f$ of multiplicity $(k+1)$ with $k\geqslant 1$ then $\delta(G)>k/(k+1).$
\end{proposition}

\begin{proof}[Proof of Proposition \ref{prop: parabolic critical exponent}]
Without loss of generality, assume that $\Lambda$ accumulates at $x$ from the right.
Replacing $f$ by $f^{-1}$ if necessary, we may also assume that $f$ is real analytic and contracting on some interval $I={]x,x'[}\subset \SS^1$.
We now apply Theorem \ref{thm: generate ping-pong pair} to find a pair of elements $(h_1,h_2)\sbs G$ with the following properties:
\begin{enumerate}[\indent(a)]
	\item\label{ass: hyp pair a} each of $h_1,h_2$ has only hyperbolic fixed points on $\SS^1$, and in particular each of them has only finitely many fixed points;
	\item\label{ass: hyp pair b}  each $h_i$ has at least one fixed point in $I$;
	\item\label{ass: hyp pair c}  the sets of fixed points of $h_1$ and $h_2$ are disjoint.
\end{enumerate}
By \eqref{ass: hyp pair c}, at least one of $h_1,h_2$ does not fix $x.$
We denote such an element by $h.$
By \eqref{ass: hyp pair b}, $h$ has fixed points on $I$ and we let $y$ be the leftmost one.
Replacing $h$ by $h^{-1}$ if necessary, we may assume that $y$ is an attracting fixed point.
Since $h$ does not fix $x$, we may further replace $h$ by a sufficiently large power so that $h$ strictly contracts the interval $[x,y]$ and satisfies $h([x,y])\cap f([x,y])=\vn$.
Thus $f$ and $h$ freely generate a free semigroup $T$ on $[x,y]$ exhibiting ping-pong dynamics.

For each $g\in \Diff^1_+(\SS^1)$ that preserves the interval $[x,y],$ we consider its restriction to $[x,y]\sbs\SS^1$ and define its co-norm by
\[\gamma(g) = \inf_{z\in [x,y]}g'(z).\]
The co-norm is super-multiplicative:
$\gamma(g_1g_2)\geq \gamma(g_1)\gamma( g_2)$ for all such $g_1,g_2.$
Given a semigroup $H$ consisting of such diffeomorphisms and an exponent $s\geq 0$, we define the associated Poincar\'e series 
\[\vp(H,s)\defeq \sum_{g\in H\sm\lb{\Id}} \gamma(g)^s.\]

Let $\pair f_{\mr s}$ and $\pair h_{\mr s}$ denote the sub-semigroups generated by $f$ and $h$, respectively.
By the freeness of $T=\pair{f,h}_{\mr s}$, we have 
\begin{align*}
\vp(T,s)&\geq \sum_{j \geq 1} \sum_{n_1,\dotsc,n_j \geq 1} \sum_{m_1,\dotsc,m_j \geq 1} \gamma(f^{n_1}h^{m_1} \dotsm f^{n_j}h^{m_j})^s \\
&\geq \sum_{j \geq 1} \sum_{n_1,\dotsc,n_j \geq 1} \sum_{m_1,\dotsc,m_j \geq 1} \gamma( f^{n_1})^s \gamma(h^{m_1})^s \dotsm \gamma( f^{n_j})^s \gamma(h^{m_j})^s\\
&=\sum_{j\geq 1}\vp(\pair f_{\mr s},s)^j\vp(\pair h_{\mr s},s)^j.
\end{align*}
In particular, $\vp(T,s)$ diverges whenever $\vp(\pair f_{\mr s},s)\vp(\pair h_{\mr s},s)\geq 1$.
Recall that $f$ is real analytic and contracting on $I$.
By Lemma \ref{lem: around parabolic} \eqref{ppt: para 2}, for every $k/(k+1)<s\leqslant1,$
\[\vp(\pair f_{\mr s},s)\gg_{f,I} \sum_{n\geq 1} n^{-\frac{k+1}{k}s} \gg \frac{k}{(k+1)s-k}.\]
Thus $\vp(\pair f_{\mr s},s) \to +\infty$ as $s\downarrow\frac{k}{k+1}.$ 
On the other hand, since $h$ contracts $[x,y]$ as well,
\[\vp(\pair h_{\mr s},s) \geq \vp(\pair h_{\mr s},1)>0, \quad \forall 0<s\leqslant 1. \] 
Hence there exists $s_0>\frac{k}{k+1}$ such that $\vp(\pair f_{\mr s},s_0)\vp(\pair h_{\mr s},s_0)>1$ and therefore $\vp(T,s_0)$ diverges.
Consequently,
\[\limsup_{n\to\infty}\frac{1}{n}\log\#\lb{g\in T:g'|_{[x,y]}\geq e^{-n}}\geq s_0>\frac{k}{k+1}.\]

Since $\Lambda \cap {]x,y[} \neq \vn$, choose $z\in\Lambda\cap {]x,y[}$ and an open interval $B(z,\ve_0)\sbs [x,y].$
By the definition of $\delta(G)$, we have
\[\delta(G)\geq\limsup_{n\to\infty}\frac{1}{n}\log\#\lb{g\in G:g'|_{B(z,\ve_0)}\geq e^{-n}}\geq s_0>\frac{k}{k+1}.\qedhere\]
\end{proof}

\section{The Hausdorff dimension of minimal sets}
\label{sec: 4}

\subsection{Upper bound for the Hausdorff dimension}
\begin{proposition}\label{prop: Hausdorff upper bound}
	Let $G\sbs \Diff_+^2(\SS^1)$ be a finitely generated subgroup without finite orbits, and let $\Lambda$ be its unique minimal set. 
	Assume that $G$ satisfies property $(\star)$ or $(\Lambda\star)$. 
	Then $\delta(G)\geqslant \dim_{\rm H}\Lambda.$
\end{proposition}

\begin{proof}
Let $\Lambda'=\Lambda\sm G(\NE).$
Since $\Lambda\cap G(\NE)$ is countable by Theorem \ref{thm: DKN09} \eqref{fact: NE finite}, we have $\dimH\Lambda'=\dimH\Lambda$. 
We now invoke DKN's expanding strategy, detailed in Section \ref{sec: group actions on the circle}.
For $m\geqslant 1,$ set
\[\cE_m\defeq\lb{B(z,r_n(z)):z\in\Lambda',n\in\Acc(z), r_n(z)\in[e^{-(m+2)} ,e^{-(m+1)}] }.\]
Since $\Acc(z)$ is infinite and $r_n(z)\to 0$ along $n\in\Acc(z)$ by Lemma \ref{lem: rnz}, each $z\in\Lambda'$ lies in $\bigcup_{I\in\cE_m} I$ for infinitely many $m.$
Therefore,
\[\Lambda'\sbs \limsup_{m\to\infty}\bigcup_{I\in\cE_m} I.  \]
Note that $|I|\leqslant e^{-m}$ for every $I\in\cE_m.$
For every $m>0,$ let $\cF_m\sbs \Lambda'\times\ZZ_{>0}$ be a finite subset of  maximal cardinality such that
\[\cF_m\sbs\lb{(z,n) \in \Lambda' \times \ZZ_{> 0} : n\in\Acc(z),\, r_n(z)\in[e^{-(m+2)}, e^{-(m+1)}]} \]
and the intervals $ B(z,r_n(z))$ are pairwise disjoint for all $(z,n)\in \cF_m$.
Applying Lemma \ref{lem: cover and dim} yields
\begin{equation}\label{eqn: limsup F_m}
	\limsup_{m\to\infty}\frac{1}{m}\log\#\cF_m\geqslant\dimH\limsup_{m\to\infty}\bigcup_{I\in\cE_m} I\geqslant \dimH\Lambda'=\dimH\Lambda. 
\end{equation}

For each $(z,n)\in\cF_m$, Lemma \ref{lem: rnz} implies that
\[t(n,z)'|_{B(z,r_n(z))}\leqslant 2\ve_0\cdot r_n(z)^{-1}\leqslant 2\ve_0\cdot e^{m+2} ,\]
and $t(n,z)B(z,r_n(z))\sps B(t(n,z)z,\ve_0/2)$.
Recall that $\ve_0 > 0$ is a constant given by Lemma \ref{lem: acceleration} depending only on $G$. 
Consequently,
\[(t(n,z)^{-1})'|_{B(t(n,z)z,\ve_0/2)}\geqslant \frac{1}{2\ve_0} \cdot e^{-(m+2)}.\]

However, the elements $t(n,z)$ may not be  distinct for $(z,n)$ ranging in $\cF_m.$
The next step is to find a large subset on which they are distinct. 
Let $\cK$ be a finite $\ve_0/2$-dense subset of $\Lambda'.$
For each $(z,n)\in\cF_m$, choose $x\in\cK$ with $d(x,t(n,z)z)<\ve_0/2$, so that $x\in t(n,z)B(z,r_n(z)).$
By the pigeonhole principle, for every $m>0$, there exists a subset $\cF_m'\sbs\cF_m$ and a point $x_m\in\cK$ such that 
\begin{enumerate}[\indent (a)]
	\item \label{ass: a} $x_m\in t(n,z)B(z,r_n(z))$ for every $(z,n)\in\cF_m'$, and
	\item \label{ass: b} $\#\cF_m'\geqslant (\#\cK)^{-1}\#\cF_m.$
\end{enumerate}
Since the intervals $ B(z,r_n(z))$ are pairwise disjoint for $(z,n)\in \cF_m$, condition \eqref{ass: a} forces the elements $t(n,z)$ to be pairwise distinct.
Hence, using \eqref{ass: b} and \eqref{eqn: limsup F_m}, we have 
\begin{align*}
	\delta(G)&\geqslant  \limsup_{m\to\infty}\frac{1}{m}\log\# \lb{g\in G:\exists x\in\Lambda ,g'|_{B(x,\ve_0/2)}\geq e^{-m}}\\
	&=\limsup_{m\to\infty}\frac{1}{m}\log\# \lb{g\in G:\exists x\in\Lambda ,g'|_{B(x,\ve_0/2)}\geq \frac{1}{2\ve_0}\cdot e^{-(m+2)}}\\
	&\geqslant \limsup_{m\to\infty}\frac{1}{m}\log\#\cF_m'= \limsup_{m\to\infty}\frac{1}{m}\log\#\cF_m\geqslant \dimH\Lambda.\qedhere
\end{align*}
\end{proof}

\subsection{Approximation by uniformly hyperbolic subsystems}
The argument establishing the upper bound on $\dimH\Lambda$ has an additional application: it yields an approximation of the original system by uniformly hyperbolic subsystems.
More precisely, one obtains finitely generated sub-semigroups whose generators form contracting IFSs on a subinterval and whose attractors have Hausdorff dimensions arbitrarily close to $\dimH\Lambda.$
This approximation result is a consequence of the preceding construction and is of independent interest.
The ideas involved will also be used later in Section \ref{sec: conformal Hausdorff} in the study of conformal measures.

\begin{theorem}\label{thm: approximation by UHS}
	Let $G\sbs \Diff_+^2(\SS^1)$ be a finitely generated subgroup without invariant probability measures on $\SS^1$, and let $\Lambda$ be its unique minimal set. 
	Assume that $G$ satisfies property $(\star)$ or $(\Lambda\star)$. 
	Then for every $\ve>0$ there exist a finite set $\lb{g_i}_{i=1}^\ell\sbs G$ and a closed interval $I_0\sbs\SS^1$ such that
	\begin{enumerate}
		\item\label{ass: contracting IFS} $\lb{g_i}_{i=1}^\ell$ forms a \textbf{contracting IFS} on $I_0$: for every $1\leqslant i\leqslant \ell,$ $g_i(I_0)\sbs I_0$ and $(g_i)'|_{I_0}<1;$
		\item\label{ass: open set condition} $\lb{g_i|_{I_0}}_{i=1}^\ell$ satisfies the following separation condition: $g_i(I_0)\cap g_j(I_0)=\vn$ for every $i\ne j;$
		\item\label{ass: dimension approximation} the attractor $\Lambda_0$ of the IFS $\lb{g_i|_{I_0}}_{i=1}^\ell,$ 
		\begin{equation}
        \label{eqn: Lambda0}
		\Lambda_0=\bigcap_{n=1}^\infty\bigcup_{1\leqslant i_1,\cdots,i_n\leqslant \ell}g_{i_n}\circ\cdots\circ g_{i_1}(I_0),    
		\end{equation}
		is contained in $\Lambda$ and satisfies $\dimH\Lambda_0\geqslant \dimH\Lambda-\ve.$
	\end{enumerate}
\end{theorem}

Recall from the proof of Proposition \ref{prop: Hausdorff upper bound} that we constructed a sequence of finite sets $\cF_m$.
For each $m,$ the associated family of pairwise disjoint intervals $\cE_m'\defeq\lb{B(z,r_n(z)):(z,n)\in\cF_m}$ satisfies the following properties:
\begin{enumerate}[\indent (a)]
	\item Every $I\in\cE_m'$ has its center in $\Lambda$ and satisfies $|I|\geqslant 2e^{-(m+2)}.$ 
	\item For every $I= B(z,r_n(z))\in\cE_m',$ the corresponding element $g=t(n,z)\in G$ satisfies $C^{-1}e^{m}\leqslant g'|_I\leqslant Ce^{m}$, where $C=10\ve_0^{-1}.$
	\item $\limsup_{m\to\infty}\frac{1}{m}\log\#\cE_m'\geqslant\dimH\Lambda.$
\end{enumerate}
To deduce Theorem \ref{thm: approximation by UHS}, it now suffices to establish the following lemma.

\begin{lemma}\label{lem: large attractor}
	Let $G\sbs \Diff_+^2(\SS^1)$ be a finitely generated subgroup without invariant probability measures on $\SS^1$, and let $\Lambda$ be its unique minimal set.
    Let $\alpha >0$ and $C>1$. 
	Assume that for each $m$ we are given a finite family $\cE_m'$ of pairwise disjoint open intervals such that:
	\begin{enumerate}[\indent(a)]
		\item\label{ass: disjoint intervals} Every $I\in\cE_m'$ has its center in $\Lambda$ and satisfies $|I|\geqslant C^{-1}e^{-m}.$ 
		\item\label{ass: expanding elements} For every $I\in\cE_m',$ there exists $g=g_{I,m}\in G$ such that $C^{-1}e^{m}\leqslant g'|_I\leqslant Ce^{m}$.
		\item\label{ass: dimension approximation'} $\limsup_{m\to\infty}\frac{1}{m}\log\#\cE_m'\geqslant\alpha .$
	\end{enumerate}
	Then for every $\ve>0,$ there exist finitely many elements $g_1,\cdots,g_\ell\in G$ and a closed subinterval $I_0\sbs\SS^1$ such that conditions \eqref{ass: contracting IFS} and \eqref{ass: open set condition} of Theorem \ref{thm: approximation by UHS} hold together with
	\begin{enumerate}[\indent(3')]
		\item the attractor $\Lambda_0$ of $\lb{g_i|_{I_0}}_{i=1}^\ell,$ given by \eqref{eqn: Lambda0}, is contained in $\Lambda$ and satisfies $\dimH\Lambda_0\geqslant \alpha -\ve.$
	\end{enumerate}
\end{lemma}

\begin{proof}
	Let $\ve_1>0$ be the constant given by Lemma \ref{lem: global contraction} and set $\ve_1'= C^{-2}\ve_1/10$.
    Let $\cK$ be a finite $\ve_1'$-dense subset of $\Lambda.$
	By the derivative bound in \eqref{ass: expanding elements}, every $I\in \cE_m'$ has length at most $Ce^{-m}.$
	Hence for sufficiently large $m$ and for every $I\in \cE_m',$ since the center of $I$ is contained in $\Lambda,$ we may choose  $x_1=x_{1,I}\in\cK$ such that $I\sbs B(x_1,\ve_1').$ 
	Similarly, combining \eqref{ass: disjoint intervals} with \eqref{ass: expanding elements}, for every $I\in \cE_m',$ we have $g_{I,m}(I)\sps B(x,C^{-2}/2)$ for some $x\in\Lambda$.
    Therefore, $g_{I,m}(I)\sps B(x_2,\ve_1')$ for some $x_2=x_{2,I}\in\cK.$
    
    Choose a strictly increasing sequence $(m_i)_{i\geqslant1}$ of positive integers such that
\[\lim_{i\to\infty}\frac{1}{m_i}\log\#\cE_{m_i}'=\limsup_{m\to\infty}\frac{1}{m}\log\#\cE_m'.\]
	By the pigeonhole principle, we can find $x_1,x_2\in\cK$ for which the set
	\[\cE_m''\defeq\lb{I\in\cE_m': I\sbs B(x_1,\ve_1')\text{ and } B(x_2, \ve_1')\sbs g_{I,m}(I)}\]
	satisfies $\#\cE_{m_i}''\geqslant(\#\cK)^{-2}\cdot\#\cE_{m_i}'$ for infinitely many indices $i.$ 
	We fix such $x_1,x_2$ for the remainder of this proof. 
	By assumption \eqref{ass: dimension approximation'}, we have 
	\begin{equation}\label{eqn: E_m''}
		\limsup_{m\to\infty}\frac{1}{m}\log\#\cE_m''=\limsup_{m\to\infty}\frac{1}{m}\log\#\cE_m'\geqslant\alpha .
	\end{equation}
	
	Let $I_0$ denote the closure of $B(x_1,\ve_1').$
	By Lemma \ref{lem: global contraction}, there exists $f\in G$ such that $ f^{-1}B(x_1,2\ve_1')\sbs B(x_2,\ve_1').$
	Therefore, for every $I\in\cE_m'',$ 
	\[fg_{I,m}(I_0)\sps (fg_{I,m}) B(x_1,\ve_1')\sps (fg_{I,m})(I)\sps fB(x_2,\ve_1')\sps B(x_1,2\ve_1')\sps \ol{B(x_1,\ve_1')}=I_0 . \]
	Consequently, $(fg_{I,m})^{-1}I_0\sbs I\sbs I_0,$ and in particular the intervals $ (fg_{I,m})^{-1}I_0$ are pairwise disjoint for $I\in\cE_m''.$
	Moreover, the derivative bounds in \eqref{ass: expanding elements} imply that $ {C'}^{-1}e^{-m}\leqslant ((fg_{I,m})^{-1})'|_{I_0}\leqslant C' e^{-m},$ for some constant $C'>0$ depending only on $C$ and $f.$
	Thus, for sufficiently large $m>0,$ the family
	\[\cG_m\defeq\lb{(fg_{I,m})^{-1}:I\in\cE_m''} \]
	defines a contracting IFS on $I_0$ and satisfies the separation condition in \eqref{ass: open set condition}.
	
	To estimate the Hausdorff dimension of the attractor, we consider the uniform measure $\mu$ on $\cG_m.$
	It defines a random walk on the circle. 
	Let $\nu$ be the unique $\mu$-stationary measure on $I_0$, whose support is the attractor $\Lambda_0.$
	We apply Theorem \ref{thm: stationary measure dimension} to estimate the dimension of $\nu.$
	By the separation condition \eqref{ass: open set condition}, $\cG_m$ does not preserve a common invariant measure on $I_0$ and alternative \eqref{alt: locally nondiscrete} in Theorem \ref{thm: stationary measure dimension} is excluded.
    Hence $\nu$ is exact dimensional and $\dim\nu =\frac{ h_{\mr{RW}}(\mu)}{|\lambda(\mu,\nu)|}$.
	By the definition of the Lyapunov exponent $\lambda(\mu,\nu)$ in \eqref{eqn: LE def},
	\[|\lambda(\mu,\nu)|\leqslant \sup_{f\in\cG_m,~x\in I_0} |\log f'(x)|\leqslant m+\log C'.\]
    Moreover, $\cG_m$ freely generates a free semigroup, so by Lemma \ref{lem: random walk entropy free} we have $h_{\mr{RW}}(\mu)=\log\# \cG_m$.
	Applying Lemma \ref{lem: measure Hausdorff}, we obtain
	\[\dim\Lambda_0\geqslant\dim\nu=\frac{ h_{\mr{RW}}(\mu)}{|\lambda(\mu,\nu)|}\geqslant \frac{\log\#\cG_m}{m+\log C'}=\frac{\log\#\cE_m''}{ m+\log C'}.\]
	Finally, by \eqref{eqn: E_m''}, for arbitrarily small $\ve>0$, one can choose $m$ large enough so that the attractor $\Lambda_0$ constructed above satisfies $\dimH\Lambda_0\geqslant\alpha-\ve$.
\end{proof}
\begin{remark}
    To keep the argument within the framework used here, we applied Theorem \ref{thm: stationary measure dimension} in the last step, although its use is not essential at this point.
    Indeed, for contracting IFSs satisfying the open set condition (implied by \eqref{ass: open set condition}), the Hausdorff dimension of the attractor is well understood in the conformal setting; see, for example, \cite{Pat96, FengHu}.
    In the next subsection,  Theorem \ref{thm: stationary measure dimension} will become an indispensable ingredient.
\end{remark}

\subsection{Lower bound for the Hausdorff dimension: a variational principle for critical exponents}
Our next goal is to show that the Hausdorff dimension of the minimal set is bounded below by the dynamical critical exponent.
\begin{proposition}\label{prop: Hausdorff lower bound} 
	Let $G\sbs \Diff_+^2(\SS^1)$ be a finitely generated subgroup without invariant probability measures on $\SS^1$, and let $\Lambda$ be its unique minimal set. 
	Assume that $G$ is locally discrete and word hyperbolic.
	Then $\delta(G)\leqslant \dim_{\rm H}\Lambda.$
\end{proposition}
The notion of word-hyperbolic groups will be recalled in Appendix~\ref{app}.
Here, let us only mention that all finitely generated virtually free groups are word-hyperbolic.

This inequality is a direct consequence of the variational principle for dynamical critical exponents proved below, combined with Lemma~\ref{lem: measure Hausdorff}.

\begin{proposition}[Variational principle for the dynamical critical exponent]\label{prop: variational principle of critical exponents}
	Let $G\sbs\Diff_+^2(\SS^1)$ be a finitely generated word-hyperbolic subgroup.
	Assume that $G$ is locally discrete and does not preserve any invariant probability measure on $\SS^1.$
	Let $\Lambda$ be the unique minimal set of $G.$ 
    
	Then for every $\ve>0,$ there exist a finitely supported probability measure $\mu$ on $G$ and a $\mu$-stationary measure $\nu$ supported on $\Lambda$ such that $\dim_{\mr H}\nu\geq \delta(G)-\ve.$
\end{proposition}

\begin{proof}[Proof of Proposition \ref{prop: variational principle of critical exponents}.]
	By Theorem \ref{thm: generate ping-pong pair}, there exists a pair of elements $h_1,h_2\in G$ and a closed interval $I_0\sbs\SS^1$ such that the following properties hold:
	\begin{enumerate}[\indent (a)]
		\item $I_0\cap\Lambda\ne\vn$;
		\item both $h_1,h_2$ preserve $I_0$ and contract $I_0;$
		\item the intervals $h_1(I_0),h_2(I_0)$ are disjoint;
		\item $h_1$ and $h_2$ freely generate a free group.
	\end{enumerate}
	Let $F$ be the set of all positive words of length three in $h_1$ and $h_2$.
	Since $h_1$ and $h_2$ freely generate a free group, we have $\#F=8$ and $\gamma_1^m\ne\gamma_2^n$ for all distinct $\gamma_1,\gamma_2\in F$ and all $(m,n)\in\ZZ^2\setminus\{(0,0)\}$.
    The subset $F\sbs G$ serves as the input for the following geometric-group-theoretic proposition, whose proof is given in Appendix \ref{app}.
    \begin{proposition}\label{prop: free semigroup}
		Let $\Gamma$ be a word-hyperbolic group.
		Let $F\sbs \Gamma$ be a finite subset of cardinality $\#F \geqslant 5$.
		Assume that for each pair of distinct elements $\gamma_1 \neq \gamma_2 \in F$ and for every pair of integers $(m,n)\ne(0,0),$ $\gamma_1^m\ne \gamma_2^n.$
		Then there exist a constant $c>0$ and a map $\ell : F \to \ZZ_{>0}$ such that the following holds.

		For any finite subset $A\sbs \Gamma$, there exist $\gamma_1\ne\gamma_2 \in F$ and a subset $A_0 \sbs A$ of cardinality $\#A_0 \geq c \#A - 1$, such that the elements
		\[
		a\gamma_i^{\ell(\gamma_i)},
		\qquad (a,i)\in A_0\times\{1,2\},
		\]
		are pairwise distinct and freely generate a free sub-semigroup of $\Gamma$.
    \end{proposition}
	Applying this proposition to the group $\Gamma = G$, we obtain a constant $c>0$ and, for each $\gamma\in F$, a positive integer $\ell(\gamma)$.
	Let $M>0$ be such that $(\gamma^{\ell(\gamma)})'(x)\geqslant M^{-1}$ for every $x\in\SS^1$ and $\gamma\in F.$
	
	Assume $\delta(G)>0.$
	For every positive integer $n,$ set
	\[A_n\defeq\lb{g\in G: g(I_0)\sbs I_0,g'|_{I_0}\geqslant e^{-n} }.\]
	We do not assume a priori that $A_n$ is finite.
	We therefore first work with an arbitrary finite subset $A_n'\subset A_n$ of sufficiently large cardinality.
	Then Proposition \ref{prop: free semigroup} gives a subset $A_{n,0}\sbs A_n'$ with $\#A_{n,0}\geqslant c\cdot\# A_n'-1$ and two elements $\gamma_1,\gamma_2\in F$ such that 
	\[\wt A_n\defeq\lb{a\gamma_i^{\ell(\gamma_i)}: i=1,2,\text{ and }a\in A_{n,0}}\]
	freely generates a free semigroup.
	Let $\mu_n$ be the uniform measure on $\wt A_n.$
	
	By construction of $h_1,h_2,$ both $\gamma_1$ and $\gamma_2$ preserve $I_0$ and have disjoint images.
	Hence every element of $\wt A_n$ also preserves $I_0$, and the family $\wt A_n$ does not preserve any common invariant probability measure on $I_0$.
	Since $\wt A_n\subset G$, the set $\wt A_n$ preserves $I_0\cap\Lambda$.
	Hence there exists an ergodic $\mu_n$-stationary measure $\nu_n$ supported on $I_0\cap\Lambda.$

    Since $G$ is $C^1$-locally discrete and $\nu_n$ is supported on $\Lambda,$ alternative \eqref{alt: locally nondiscrete} in Theorem \ref{thm: stationary measure dimension} is excluded, and hence 
    \[\dimH\nu_n=\frac{h_{\mr{RW}}(\mu_n)}{|\lambda(\mu_n,\nu_n)|}.\]
	By Lemma \ref{lem: random walk entropy free},
	\[h_{\mr{RW}}(\mu_n)=\log\#\wt A_n\geqslant\log\# A_{n,0}\geqslant \log(c\#A_n'-1) .\]
	Since $\supp\nu_n\sbs I_0,$ we have $f'(x)\geqslant M^{-1}e^{-n}$ for every $f\in\supp\mu_n=\wt A_n$ and $x\in\supp\nu_n.$
	Therefore, the definition of Lyapunov exponents \eqref{eqn: LE def} yields $|\lambda(\mu_n,\nu_n)|\leqslant n+\log M.$
	Thus 
	\[\dimH\nu_n\geqslant \frac{\log(c\#A_n'-1)}{n+\log M}.\]
	Since $\dimH\nu_n\leqslant1$, the preceding inequality also yields
	$\#A_n'\leqslant \frac{Me^n+1}{c}.$
	This bound is independent of the choice of $A_n'$, so $A_n$ is finite. 
	By the growth estimate below, we may therefore take $A_n'=A_n$ in the preceding construction for all sufficiently large $n$.
	
	Note that the constants $c$ and $M$ are independent of $n$ and the next lemma ensures the required growth of $\#A_n$. Therefore, the proposition follows by choosing $\mu=\mu_n,$ $\nu=\nu_n$ for some sufficiently large $n.$
	\begin{lemma}
		We have $\limsup_{n\to\infty}\frac{1}{n}\log\# A_n\geqslant\delta(G) .$
	\end{lemma}
	\begin{proof}
		Let $\ve_1>0$ be the constant given by Lemma \ref{lem: global contraction}.
		Set $h_0=h_1h_2$, and let $x$ be the unique fixed point of $h_0$ in $I_0.$
		Then $x$ lies in the interior of $I_0$ and $\lim_{m\to+\infty} h_0^my=x$ for every $y\in I_0.$
		Since $I_0\cap \Lambda\ne\vn$ and $h_0\Lambda=\Lambda$, it follows that $x\in\Lambda.$
        Choose $m\geqslant 1$ such that $h_0^m I_0\sbs B(x,\ve_1).$
		Let $\ve'>0$ be sufficiently small so that $\ve'<\ve_1$ and $B(x,\ve')\sbs I_0$.

		Let $\lb{x_1,\cdots,x_k}$ be a $\frac{\ve_1}{2}$-dense subset of $\Lambda.$
		For each $1\leqslant i\leqslant k,$ Lemma \ref{lem: global contraction} provides a map $f_i\in G$ such that $f_iB(x_i,\ve_1)\sbs B(x,\ve')\sbs I_0.$
		Fix these maps for what follows. 
		Choose a positive integer $\ell$ such that $2^\ell\geqslant 2\ve_1^{-1},$ and let $A=\lb{h_1,h_2}^{*\ell},$ which has cardinality $2^\ell.$
		Since the images $h(I_0)$ are pairwise disjoint for $h\in A$, the intervals  $gh_0^m h(I_0)$ are pairwise disjoint for any fixed $g\in G$.
		
		Now assume that $g\in G$ satisfies $g'|_{B(x,\ve_1)}\geqslant e^{-n}.$
        Among the $2^\ell$ disjoint intervals $g h_0^m h(I_0)$ with $h\in A,$ at least one has length less than $2^{-\ell}\leqslant \ve_1/2.$
		Since each such interval intersects $\Lambda$, there exists $ i$ such that $B(x_i,\ve_1)\supset g h_0^m h(I_0).$ 
		Thus $f_igh_0^mhI_0\sbs I_0$ by our choice of $f_i.$
		
		Define $\ol g \defeq f_igh_0^mh$ and let $ B_n$ be the image of $\lb{g\in G: g'|_{B(x,\ve_1)}\geqslant e^{-n}}$ under the map $g\mapsto \ol g.$
		This map is at most $k2^\ell$-to-$1$.
		Hence 
		\[\# B_n\geqslant k^{-1}2^{-\ell}\#\lb{g\in G: g'|_{B(x,\ve_1)}\geqslant e^{-n}}.\]
		Since $h_0^mhI_0\sbs h_0^mI_0\sbs B(x,\ve_1),$ we have $\ol g'|_{I_0}\geqslant c\cdot e^{-n}$ for some constant $c>0$ depending only on $\{f_i\},$ $h_0,$ and the elements of $A.$
		Hence
		\[ B_n\sbs \lb{g\in G:g(I_0)\sbs I_0, g'|_{I_0}\geqslant e^{-(n+\lceil-\log c\rceil )}}= A_{n+ \lceil-\log c\rceil }.\]
		By Lemma \ref{lem: realize DCE at epsilon_1}, 
        \[\limsup_{n\to\infty}\frac{1}{n}\log\# A_n\geqslant \limsup_{n\to\infty}\frac{1}{n}\log\#\lb{g\in G: g'|_{B(x,\ve_1)}\geqslant e^{-n}}=\delta(G),\]
        which completes the proof.
	\end{proof}
\end{proof}

\section{Conformal measures}
\label{sec: 5}

Let $G\sbs\Diff_+^2(\SS^1)$ be a finitely generated subgroup without finite orbits and let $\Lambda$ be the unique minimal set of $G.$
Recall from Definition \ref{def: stationary and conformal} \eqref{it: conformal} that a probability measure $\nu$ on $\SS^1$ is called $\delta$-conformal for $G$ if for every $g\in G$ one has 
\begin{equation}\label{eqn: conformal def}
	g^{-1}_*\nu=(g')^\delta\nu,
\end{equation}
where the right-hand side denotes the multiplication of the measure $\nu$ by the function $x \in \SS^1 \mapsto g'(x)^\delta \in \RR$.

Throughout this section we focus on the case in which $\Lambda$ is exceptional. 
As emphasized in the introduction, the main object of interest here is an atomless conformal measure supported on the minimal set $\Lambda$, since such a measure reveals the key dimensional properties of $\Lambda.$
Our goal is therefore to construct such a conformal measure and investigate its dimensional properties.

We also recall a fundamental result of Deroin--Kleptsyn--Navas concerning the uniqueness and exponent of conformal measures.
\begin{theorem}[{\cite[Theorem F]{DKN09}}]\label{thm: DKN09 thmF}
	Let $G$ be a finitely generated subgroup of $\Diff_+^2(\SS^1)$ with an exceptional minimal set $\Lambda.$
	Assume that $G$ satisfies property $(\Lambda\star).$
	Then there exists at most one atomless conformal measure supported on $\Lambda,$ and if such a measure exists, then its exponent belongs to $]0,1[.$
\end{theorem}

\subsection{Existence of atomless conformal measures}\label{sec: 5.1}
Let $G\sbs\Diff_+^\omega(\SS^1)$ be a finitely generated subgroup with an exceptional minimal set $\Lambda.$
Our goal is to construct atomless conformal measures.
Our starting point is the classical Patterson--Sullivan construction for Fuchsian groups.
Several issues arise if we naively follow the Patterson--Sullivan construction.

First, Fuchsian groups act naturally on the unit disk and their circle action is the induced action on the boundary of the disk.
For general subgroups $G\sbs\Diff_+^\omega(\SS^1)$, the action on the circle does not necessarily extend to the unit disk.
To solve this issue, we will let the complement $\SS^1 \setminus \Lambda$ play the role of the unit disk.
Given a wandering point $x \in \SS^1 \setminus \Lambda$ (see Definition \ref{def: wandering} and Lemma \ref{lem: wandering}),
we consider the associated Poincar\'e series
\[P(x,s) \defeq \sum_{g\in G}g'(x)^s, \quad \text{for } s \in \RR.\]
Let $\delta$ denote the critical exponent of $P(x,s)$, that is, the minimal $\delta \in \RR$ such that $P(x,s)$ converges for all $s > \delta$.
Whenever $s > \delta$ define
\[\nu_s=\frac{1}{P(x,s)}\sum_{g\in G}g'(x)^s\Dir_{gx}\]
where $ \Dir_{gx}$ is the Dirac measure at $gx$.
This is a probability measure on $Gx \subset \mathbb{S}^1$.

For the purpose of illustrating the ideas, let us assume that $P(x,\delta) = +\infty$.
In this case, the family $\lb{\nu_s}_{s>\delta}$ has weak-$*$ accumulation points as $s\downarrow \delta $, and any such limit is a probability measure supported on $\omega(x)$.
By construction, it is $\delta$-conformal for $G$.

However, at this point, a second issue arises.
The action of $G$ on $\SS^1 \setminus \Lambda$ might not be properly discontinuous. 
In particular, if $J$ is the connected component containing $x,$ then $J$ may contain a point $y \in J$ fixed by nontrivial elements of $G$.
In such situations, $\omega(x)$ may be $\Lambda \cup G y$ and the limit measure may have an atom at $y$.
So this naive construction does not guarantee that the limit measure is atomless and supported on $\Lambda$.
To overcome this issue, our innovative idea is to modify the action outside $\Lambda$, simplifying the gap stabilizers while leaving the dynamics near $\Lambda$ unchanged.

\paragraph{$C^\infty$-representations.}
To perform such modification, we need to move beyond the analytic category.
This leads to considering $C^\infty$-representations on the circle.
\begin{definition}
Let $G\sbs \Diff^\omega_+(\SS^1)$ be a subgroup.
By a \emph{$C^\infty$-representation} of $G$, we mean an abstract group homomorphism $\iota : G \to \Diff^\infty_+(\SS^1)$.

If $G$ has an exceptional minimal set $\Lambda,$ we say the $C^\infty$-representation $\iota$ is \emph{$\Lambda$-inert} if for every $f \in G$, the diffeomorphisms $f$ and $\iota(f)$ coincide on a neighborhood of $\Lambda$.
\end{definition}
Note that a $\Lambda$-inert $C^\infty$-representation must be faithful, that is, $\iota$ is a group isomorphism from $G$ to its image.
The image $\iota(G)$ inherits from $G$ some of its properties.
\begin{lemma}\label{lem: Lambda-analytic subgroup}
	Let $G\sbs \Diff^\omega_+(\SS^1)$ be a finitely generated subgroup with an exceptional minimal set $\Lambda$.
  Let $\iota$ be a $\Lambda$-inert $C^\infty$-representation of $G$.
  Then the group $\iota(G)$ has the same minimal set $\Lambda$ as $G$, and additionally
	\begin{enumerate}[\indent (a)]
		\item every $g\in \iota(G)$ is real analytic on a neighborhood of $\Lambda$;
		\item for every $x\in \Lambda,$ the stabilizer $\Stab_{\iota(G)}(x)$ is cyclic;
		\item $\iota(G)$ satisfies property $(\Lambda\star)$.
	\end{enumerate}
\end{lemma}

\begin{proof}
	By definition, $\iota(f)$ agrees with $f$ on a neighborhood of $\Lambda$ for every $f\in G.$ 
  In particular, for every $x\in\Lambda$, we have $\iota(G)x=Gx$, which is dense in $\Lambda$.
	Thus $\Lambda$ is also the minimal set of $\iota(G).$
	Moreover, for every $x\in\Lambda,$ we have $\Stab_{\iota(G)}(x)=\iota(\Stab_G(x))$, which is cyclic by Theorem \ref{thm: real analytic with exceptional} \eqref{fact: Hector}.
	
  It remains to verify property $(\Lambda\star).$
	Since $\iota(g)$ and $g$ coincide on a neighborhood of $\Lambda$ for every $g\in G$, their derivatives agree at every $x\in\Lambda$.
	Hence $\NE(\iota(G))\cap\Lambda =\NE(G)\cap\Lambda .$ 
	Because $G$ satisfies property $(\Lambda\star)$ by Theorem \ref{thm: DKN18}, each $x\in\NE(G)\cap\Lambda $ is fixed by some nontrivial element $f\in G.$ 
	Then $\iota(f)$ fixes $x$, and the fixed point $x$ is isolated since $\iota(f)$ agrees with $f$ near $x.$
	Hence $\iota(G)$ also satisfies property $(\Lambda\star).$ 
\end{proof}

\begin{definition}
A subgroup of $\Diff_+^\infty(\SS^1)$ is called \emph{$\Lambda$-analytic} if it satisfies the properties of $\iota(G)$ in Lemma \ref{lem: Lambda-analytic subgroup}.
\end{definition}

In the category of $\Lambda$-analytic subgroups, we have a criterion for admitting atomless conformal measures supported on $\Lambda$.

\begin{definition}\label{def: fundamental pair}
	Let $G\sbs\Diff_+^\infty(\SS^1)$ be a subgroup with an exceptional minimal set $\Lambda.$ 
	A point $x \in \SS^1 \sm \Lambda$ is said to be a \emph{gentle wandering point} if, writing $J$ for the connected component containing $x$, one of the following holds:
	\begin{enumerate}[\indent(a)]
		\item the stabilizer $\Stab_G(J)\defeq\lb{g\in G: g(J)=J}$ is trivial; or
		\item $\Stab_G(J)=\pair{f}$ is infinite cyclic, and its generator $f$ has at most one fixed point in $J$, which, if it exists, is hyperbolic and distinct from $x$.
	\end{enumerate}
\end{definition}

By Lemma~\ref{lem: wandering}, a gentle wandering point is indeed wandering.
Before showing that having a gentle wandering point is a sufficient condition to have an atomeless conformal measure, let us show that every finitely generated subgroup $G\sbs\Diff^\omega_+(\SS^1)$ with an exceptional minimal set admits such a modification after passing to a finite-index subgroup.

\begin{proposition}\label{prop: construct C^inf modif}
	Let $G\sbs \Diff^\omega_+(\SS^1)$ be a finitely generated subgroup with an exceptional minimal set $\Lambda.$ 
	Then there exists a finite index subgroup $H<G$ which admits a $\Lambda$-inert $C^\infty$-representation $\iota : H \to \Diff^\infty_+(\SS^1)$ such that the action of $\iota(H)$ has a gentle wandering point.
\end{proposition}
\begin{proof}
	By Theorem \ref{thm: real analytic with exceptional} \eqref{fact: virtually free}, $G$ is virtually free.
	Passing to a finite index subgroup if necessary, we may assume without loss of generality that $G$ itself is free. 
	Let $J$ be any connected component of $\SS^1\sm\Lambda.$ 
	If $\Stab_G(J)$ is trivial, then every $x \in J$ is a gentle wandering point and we can take $\iota = \Id$.
	
	Now assume that $\Stab_G(J)$ is nontrivial. 
	By Theorem \ref{thm: real analytic with exceptional} \eqref{fact: Hector}, we have $\Stab_G(J)=\Stab_G(\partial J)=\pair{f}$ for some nontrivial element $f\in G$.
	We will use the following lemma, which is a special case of Hall's theorem \cite{Ha49} (see also Stallings's article \cite[Corollary 6.3]{St83}).
	\begin{lemma}
		Let $G$ be a finitely generated free group and let $f\in G$ be a nontrivial element. 
		Then there exists a finite index subgroup $H<G$ (hence again a finitely generated free group) such that $f\in H$ and $f$ is a free generator of $H.$
	\end{lemma}
	Applying this lemma, we pass to a finite index subgroup $H$ containing $f$ such that $f$ is a free generator of $H$ and $\Stab_H(J)=\pair{f}.$
    Since $f$ has only finitely many fixed points in $J,$ we may choose a closed subinterval $J'\subset J$ whose interior contains all of them.
    The modification will be performed only on $J'.$ Choose $f_0\in \Diff_+^\infty(\SS^1)$ so that
	\begin{enumerate}[\indent (a)]
		\item $f_0|_{\SS^1\sm J'}=f|_{\SS^1\sm J'}$, and
		\item $f_0|_{J'}$ has at most one fixed point. More precisely, if both endpoints of $J$ are attracting fixed points of $f$, then we make sure $f_0$ has only one repelling hyperbolic fixed point. If both are repelling fixed points of $f$, then we choose $f_0$ to have a unique attracting hyperbolic fixed point. Otherwise, we can choose $f_0$ with no fixed point on $J'$.
	\end{enumerate}
	
Let $\cG$ be a free generating set of $H$ containing $f.$
By the universal property of free groups, there exists a representation $\iota : H \to \Diff_+^\infty(\SS^1)$ such that $\iota(f) = f_0$ and $\iota(g) = g$ for all $g \in \cG \sm \{f\}$.
  By construction, each $g \in H$ agrees with $\iota(g)$ on a neighborhood of $\Lambda$.
  Hence $\iota$ is $\Lambda$-inert.
	
	By Lemma \ref{lem: Lambda-analytic subgroup}, $\Lambda$ is also the unique minimal set of $\iota(H)$.
	Since $\partial J\sbs\Lambda,$ we obtain
	\[\Stab_{\iota(H)}(J)=\Stab_{\iota(H)}(\partial J)=\iota(\Stab_{H}(\partial J))=\iota(\pair{f})= \pair{f_0}.\]
	Choose $x\in J\sm \Fix(f_0).$
	Then $x$ is a gentle wandering point.
\end{proof}

\paragraph{The construction of atomless conformal measures.}
Here is the aforementioned sufficient condition for the existence of an atomless conformal measure.
\begin{proposition}\label{prop: construct conformal measures}
	Let $G$ be a finitely generated subgroup of $\Diff_+^\infty(\SS^1)$ with an exceptional minimal set $\Lambda$.
    Assume that $G$ is $\Lambda$-analytic and has a gentle wandering point $x \in \SS^1 \sm \Lambda$.
	Then there exists an atomless $\delta$-conformal measure supported on $\Lambda$ with $\delta=\delta(G,x).$
\end{proposition}
\begin{proof}
	We consider the Poincar\'e series
	\[P(x,s)\defeq \sum_{g\in G}g'(x)^s.\]
	Since $x$ is a wandering point, the critical exponent of $P(x,s)$ coincides with $\delta=\delta(G,x)$, and $P(x,1)$ converges by Proposition \ref{prop: wandering}.
	
	\begin{lemma}[{\cite[Lemma 3.1]{Pat76}}]\label{lem: conformal 1}
		If $P(x,s)$ converges at $s=\delta$, then there exists a decreasing function $\psi :{]0,\infty[}\to {]0,\infty[}$ such that the series
		\begin{equation}\label{eqn: tilde P}
			\wt P(x,s)=\sum_{g\in G}\psi \bigl(g'(x)\bigr) g'(x)^s
		\end{equation}
		diverges for $s\leq \delta$ and converges for $s>\delta$. 
        Moreover, for every $\ve>0$ there exists $t_0=t_0(\ve)>0$ such that
      \begin{equation}\label{eq:psidilation}
      \forall \lambda \in ]0,1[, \; \forall t \in ]0,t_0[,\quad \psi(\lambda t)\leq \lambda^{-\ve} \psi(t). 
      \end{equation}
	\end{lemma}
	If $P(x,\delta)$ converges, we use the function $\psi$ and the constant $t_0$ given by the lemma.
	If $P(x,\delta)$ diverges, we simply set $\psi\equiv 1$ and $t_0\equiv+\infty$ so that $\wt P(x,s) = P(x,s)$.
	With this choice, \eqref{eq:psidilation} always holds.
	For every $\ve>0,$ define 
	\begin{equation}\label{eqn: E_epsilon}
		E_\ve\defeq\lb{g\in G: g'(x)\geqslant t_0(\ve)}.
	\end{equation}
	Since $P(x,1)= \sum_{g\in G}g'(x)$ converges, the set $E_\ve$ is finite.

	We now consider the modified Poincar\'e series $\wt P(x,s)$ given by \eqref{eqn: tilde P}.
	For every $s > \delta$, define the probability measure
	\[\nu_s= \frac{1}{\wt P(x,s)}\sum_{g\in G}\psi\bigl(g'(x)\bigr)  g'(x)^s \Dir_{g(x)}.\]
    
	In the following, for finite measures $\eta_1,\eta_2,$ we write $\eta_2 \leq \eta_1$ if $\eta_1-\eta_2$ is a positive measure. 

	\begin{lemma}\label{lem: nu s}
		For every $\ve>0,s>\delta$ and $f\in G,$ we have
		\begin{equation}
		\label{eqn: nu s conformal}
		f^{-1}_*\nu_s \leq \max\{1, (f')^{-\ve}\} (f')^s \nu_s + \eta_{f,\ve,s}
		\end{equation}
		where $\eta_{f,\ve,s}$ is a positive measure supported on $E_\ve \cdot x$ with total mass
		\begin{equation}\label{eqn: eta f e s}
			 \frac{1}{\wt P(x,s)} \sum_{g \in E_\ve} \psi((fg)'(x)) (fg)'(x)^s.
		\end{equation}
	\end{lemma}
	\begin{remark}
		In fact, if $P(x,\delta)$ diverges, then by the choice of $t_0 \equiv +\infty$, we have $E_\ve=\vn$ and $\nu_s$ is genuinely $s$-conformal, i.e. $ f_*^{-1}\nu_s=(f')^s\nu_s.$
		In the general case, \eqref{eqn: nu s conformal} asserts that $\nu_s$ is almost $s$-conformal.
	\end{remark}
	\begin{proof}
	By the definition of $E_\ve$ in \eqref{eqn: E_epsilon} and  by the monotonicity of $\psi$ and \eqref{eq:psidilation}, we have
	\[ \psi\bigl( (fg)'(x) \bigr)= \psi\bigl(f'(gx) g'(x) \bigr) \leq \max\{1, f'(gx)^{-\ve}\} \psi(g'(x)),\quad\forall g \in G \sm E_\ve.\]
	Using the change of variables $g\mapsto fg,$ we have
		\begin{align*}
		f^{-1}_*\nu_s &= \frac{1}{\wt P(x,s)} \sum_{g\in G} \psi(g'(x)) g'(x)^s \Dir_{(f^{-1}g)(x)}\\
    &=\frac{1}{\wt P(x,s)} \sum_{g\in G} \psi((fg)'(x))(fg)'(x)^s \Dir_{g(x)}\\
		&\leqslant\frac{1}{\wt P(x,s)} \sum_{g\in G\sm E_\ve } \max\{1, f'(gx)^{-\ve}\}\psi(g'(x)) f'(gx)^s g'(x)^s \Dir_{g(x)}\\
		&\qquad \qquad +\frac{1}{\wt P(x,s)} \sum_{g \in E_\ve} \psi\bigl((fg)'(x)\bigr)(fg)'(x)^s \Dir_{g(x)}\\
		&\leqslant \max\{1, (f')^{-\ve}\} (f')^s \nu_s + \frac{1}{\wt P(x,s)} \sum_{g \in E_\ve} \psi\bigl((fg)'(x)\bigr)(fg)'(x)^s \Dir_{g(x)}.
		\end{align*}
	We denote the error term by $\eta_{f,\ve,s}$, which is supported on $E_\ve\cdot x$ and has total mass given by \eqref{eqn: eta f e s}. 
    This yields \eqref{eqn: nu s conformal}.
	\end{proof}

	By the Banach--Alaoglu theorem, there exists a sequence $s_n \downarrow \delta$ such that $\nu_{s_n}$  converges to a probability measure $\nu$ in the weak-$*$ topology.
	
\begin{lemma}\label{lem: conformal 2}
    The limit $\nu$ is a $\delta$-conformal measure.
\end{lemma}

\begin{proof} 
    Given $f \in G$, we show that $f^{-1}_* \nu = (f')^\delta \nu$.
    Let $\ve>0.$
    We apply Lemma~\ref{lem: nu s}.
    Since the sum in \eqref{eqn: eta f e s} is uniformly bounded and $\wt P(x,s)\to+\infty$ as $s\downarrow \delta,$ the total mass of $\eta_{f,\ve,s}$ tends to $0$ as $s\to \delta.$
    Letting $s_n\downarrow\delta$ in \eqref{eqn: nu s conformal}, we obtain $f^{-1}_*\nu \leq \max\{1, (f')^{-\ve}\} (f')^\delta \nu.$
    Since this holds for every $\ve>0,$ it follows that
    \[f^{-1}_*\nu \leq (f')^\delta \nu.\]
    Applying the same inequality to $f^{-1}$ yields $f_*\nu \leq ({f^{-1}}')^\delta \nu,$
    which is equivalent to
    \[\nu\leqslant f^{-1}_*\bigl(({f^{-1}}')^\delta \nu \bigr)=\frac{1}{(f')^\delta}f_*^{-1}\nu. \]
    Combining with the previous inequality, we conclude that the equality $ (f')^\delta \nu=f^{-1}_*\nu$ holds for every $f\in G.$
    Thus $\nu$ is a $\delta$-conformal measure for $G.$
\end{proof}

\begin{lemma}\label{lem: conformal 3}
    The measure	$\nu$ is supported on $\Lambda$.
\end{lemma}
\begin{proof}
    We start by showing that $\nu$ is supported on $\omega(x)$.
    If $y \notin \omega(x)$, then it has an open neighborhood $I$ for which the set $\lb{\,g \in G : gx \in I\,}$ is finite.
    Their contribution to the sum $\wt P(x,s)$ is uniformly bounded for $s \in {]\delta,\delta+1]}$. 
    Since $\wt P(x,s)\to +\infty$ as $s \downarrow \delta$, it follows that  $\nu(I) = 0$.
    Hence $y \not\in \supp\nu.$
    Therefore, $\supp \nu \sbs \omega(x)$.

    Now we make use of the properties of $x$ being a gentle wandering point.
    Indeed, by Definition \ref{def: fundamental pair}, $\omega(x)$ must be either the minimal set $\Lambda$ or $\Lambda\cup Gy,$ where $y\in J$ is a hyperbolic fixed point of some element $f\in G.$ 
    Suppose for contradiction that $\nu(Gy)>0.$ 
    Since each point of $Gy$ is isolated in $\omega(x)$, there exists $g\in G$ such that $\nu(\{gy\})>0.$
    Since $G$ has no finite orbits, the existence of an atom implies that the $\delta$-conformal measure $\nu$ is not $G$-invariant, and hence $\delta>0$.
    Note that $g f g^{-1}$ fixes $gy$.
    We have
    \[\nu(\lb{gy})= [(gfg^{-1})'(gy)]^\delta\nu(\lb{gy})=f'(y)^\delta\nu(\lb{gy}).\]
    This is impossible, because $f'(y)\ne 1$ for a hyperbolic fixed point.
    Therefore $\nu(Gy)=0$, and we conclude that  $\nu$ is supported on $\Lambda$.
\end{proof}
    
    It remains to show that $\nu$ is atomless. 
    Suppose $ y \in \SS^1 $ is an atom of $\nu$.
    Therefore, $\nu$ cannot be $G$-invariant and hence $\delta>0.$
    Then for every $g \in G$, the conformality relation gives
    \[\nu(\{g y\}) = g'(y)^\delta \nu(\{y\}).\]
    Hence $\{g'(y) : g \in G\}$ is bounded.
    By Lemma~\ref{lem: conformal 3}, we know that $y \in \Lambda$.
    Therefore, by Theorem \ref{thm: DKN09} \eqref{fact: NE orbit}, this implies $y\in G(\NE\cap\Lambda)$.
    
    Recall that $G$ satisfies property $(\Lambda\star)$.
    Therefore, the point $y$ must in fact be a parabolic fixed point of some  $f \in G.$
    To exclude such atoms, we use the following estimate:
    \begin{lemma}\label{lem: atomless estimate}
    	If $f\in G$ and $y\in\Lambda$ is a parabolic fixed point of $f,$ then 
    	\[\lim_{r\to 0}\sup_{s>\delta} \nu_s (B(y,r))=0.\]
    \end{lemma}
    \begin{proof}
    	Assume that $y$ is a fixed point of $f$ of multiplicity $k+1$ with $k\geqslant 1.$
    	By Corollary \ref{cor: pointwise DCE} and Proposition \ref{prop: parabolic critical exponent}, 
    	\[\delta=\delta(G,x)\geqslant\delta(G)>\frac{k}{k+1}. \]
    	Since $y\notin Gx,$ we have $\nu_s(\lb{y})=0$ for every $s.$ 
        It therefore suffices to show that 
    	\begin{equation}\label{eqn: nu s y right}
    		\lim_{r\to 0}\sup_{s>\delta} \nu_s ({]y,y+r[})=0.
    	\end{equation}
    	The same for ${]y-r,y[}$ will follow by symmetry.
    	
    	Replacing $f$ by $f^{-1}$ if necessary, we may assume that $f$ is real analytic and contracting on $]y,y+r_0[$ for some $r_0 > 0.$
    	Choose $\ve=\frac{1}{2}(\delta-\frac{k}{k+1})>0.$
    	Let $E_\ve\sbs G$ be the finite subset defined in \eqref{eqn: E_epsilon}.
    	Let $y_0\in {]y,y+r_0[}$ be a point sufficiently close to $y$ so that $E_\ve x\cap [fy_0,y_0]=\vn.$
    	For $n\geqslant 1,$ set $y_n=f^n y_0$ and $I_n=[y_n,y_{n-1}].$
    	By Lemma \ref{lem: nu s} followed by Lemma \ref{lem: around parabolic} \eqref{ppt: para 2}, there exists a constant $C>0$, depending only on $f$ and $y_0,$ such that for all $n\geqslant 1$ and all $s>\delta$, 
    	\begin{align*}
    		\nu_s(I_{n+1})&=\nu_s(f^n I_1)\leq\int_{I_1}(f^n)'(z)^{s-\ve}\dd \nu_s(z)\leq\int_{I_1}(f^n)'(z)^{\delta -\ve}\dd \nu_s(z)\\
    		&\ll_C  n^{-\frac{k+1}{k}(\delta -\ve)} \nu_s(I_1) = n^{-1-\frac{k+1}{k}\ve}\nu_s(I_1),
    	\end{align*}
    	where the last equality follows from the choice of $\ve$.
    	Let $0<r<d(y,f(y_0)),$ and let $N=N(r)>0$ be the unique positive integer such that $y+r\in{]f^{N+1}y_0,f^Ny_0]}.$
    	Note that $N(r)$ tends to infinity as $r\to 0.$
    	Since $\bigcup_{n \geq N}I_{n+1} = {]y,f^N(y_0)]}$, we obtain
    	\[ \sup_{s > \delta} \nu_s\bigl( {]y,y+r]} \bigr) \ll_{C} \sup_{s>\delta}\nu_s(I_1)\sum_{n \geq N} n^{-1-\frac{k+1}{k}\ve} \ll \sum_{n \geq N} n^{-1-\frac{k+1}{k}\ve} .\]
    	Since the right-hand side is the tail of a convergent series, we obtain \eqref{eqn: nu s y right}.
    \end{proof}
    Since $\nu$ is the weak* limit of $\lb{\nu_{s_n}}$ with $s_n\downarrow \delta,$ for every $r>0$ we have
    \[\nu(B(y,r))\leqslant\liminf_{n\to \infty }\nu_{s_n}(B(y,r))\leqslant  \sup_{s>\delta} \nu_s (B(y,r)).\]
    Letting $r\to 0$ and using Lemma \ref{lem: atomless estimate}, we obtain  $\lim_{r\to 0}\nu(B(y,r))= 0$. 
    Hence $y$ cannot be an atom of $\nu.$
    We conclude that $\nu$ is an atomless $\delta$-conformal measure for $G$ whose support is $\Lambda$, which completes the proof of Proposition \ref{prop: construct conformal measures}.
\end{proof}

Combined with Proposition \ref{prop: construct C^inf modif}, Proposition \ref{prop: construct conformal measures} yields an atomless conformal measure supported on $\Lambda$ for some finite index subgroup of $G.$ 
This already suffices for studying the dimensional properties of $\Lambda.$ 
However, the existence of atomless conformal measures for the original group $G$ is also of independent interest. 
The next lemma shows that the construction in fact produces a conformal measure for $G$ itself.
\begin{lemma}\label{lem: conformal measure finite index}
	Let $G\sbs\Diff_+^\omega(\SS^1)$ be a finitely generated subgroup with an exceptional minimal set $\Lambda,$ and let $H<G$ be a finite index subgroup.
	Assume that $\nu$ is an atomless conformal measure for $H$ supported on $\Lambda.$
	Then $\nu$ is also a conformal measure for $G.$
\end{lemma}
\begin{proof}
	Let $\cM(\SS^1)$ denote the space of positive Radon measures on $\SS^1.$
	To show that an $H$-conformal measure is also $G$-conformal, we introduce the following map:
	\[\pi : G\times \cM(\SS^1)\to \cM(\SS^1),\quad (g,\eta)\mapsto g_*(|g'|^\delta\eta).\]
    It is straightforward to verify that $\pi$ defines an action of $G$ on $\cM(\SS^1)$.
  By definition, a probability measure $\eta\in\cM(\SS^1)$ is  a $\delta$-conformal measure with respect to some subgroup $H' < G$ if and only if it is fixed by $H'$ under the $\pi$-action.
	
	Let now $\nu$ be an atomless $\delta$-conformal measure for $H$ supported on $\Lambda.$ 
    Equivalently, $\nu$ is  $H$-fixed under $\pi$.
	For any $g\in G,$ let $H_g=gHg^{-1}\cap H,$ which is also a finite index subgroup of $G.$
	Then both $\nu$ and $\pi(g,\nu)$ are $H_g$-fixed under $\pi$.
	Since $H_g$ has finite index in $G,$ Lemma \ref{lem: finite index minimal set} implies that $\Lambda$ is the unique minimal set of $H_g$ on $\SS^1$.
    Since $H_g$ satisfies property $(\Lambda\star)$ by Theorem \ref{thm: DKN18}, Theorem \ref{thm: DKN09 thmF} ensures the uniqueness of atomless conformal measure supported on $\Lambda$.
    It follows that there exists $c(g)>0$ such that $\pi(g,\nu)=c(g)\cdot \nu.$
	
	This defines a group homomorphism $c:G\to \RR_{>0}$. Since $H\sbs\ker c$ has finite index in $G$, the image $c(G)$ is a finite subgroup of $\RR_{>0},$ hence trivial.
	Therefore, $\pi(g,\nu)=\nu$ for every $g\in G,$ or equivalently, $\nu$ is a $\delta$-conformal measure for $G.$
\end{proof}

\begin{corollary}\label{cor: construct conformal measures}
	Let $G\sbs\Diff_+^\omega(\SS^1)$ be a finitely generated subgroup with an exceptional minimal set $\Lambda.$
	Then $G$ admits an atomless conformal measure supported on $\Lambda.$
\end{corollary}
\begin{proof}
By Proposition \ref{prop: construct C^inf modif}, there exist a finite index subgroup $H<G$ and a $\Lambda$-inert $C^\infty$-representation $\iota : H \to \Diff^\infty_+(\SS^1)$ such that $\iota(H)$ has a gentle wandering point.
By Lemma \ref{lem: Lambda-analytic subgroup}, $\iota(H)$ is $\Lambda$-analytic.
Thus, by Proposition \ref{prop: construct conformal measures}, $\iota(H)$ admits an atomless conformal measure $\nu$ supported on $\Lambda$. 
For each $h\in H$, since $\iota(h)$ agrees with $h$ on a neighborhood of $\Lambda$, we have $\iota(h)'(x)=h'(x)$ for all $x \in \Lambda$.
It follows that $\nu$ is also conformal with respect to $H$.
By Lemma \ref{lem: conformal measure finite index}, $\nu$ is in fact $G$-conformal.
\end{proof}

\subsection{Conformal exponent and Hausdorff dimension of the minimal set}
\label{sec: conformal Hausdorff}

The main result of this subsection is the following, which establishes the equality between the conformal exponent and the Hausdorff dimension of $\Lambda.$

\begin{proposition}\label{prop: conformal dimension}
	Let $G\sbs\Diff_+^2(\SS^1)$ be a finitely generated subgroup with an exceptional minimal set $\Lambda.$ 
	Assume that $G$ satisfies property $(\Lambda\star).$ 
	If there exists an atomless $\delta$-conformal measure supported on $\Lambda$, then $\delta=\dim_{\mr H}\Lambda.$
\end{proposition}

Let $G\sbs\Diff_+^\omega(\SS^1)$ be a finitely generated subgroup with an exceptional minimal set $\Lambda.$
By Corollary \ref{cor: construct conformal measures}, we have constructed an atomless $\delta$-conformal measure supported on $\Lambda.$ 
Combined with the exponent bound in Theorem \ref{thm: DKN09 thmF} and the identity in Proposition \ref{prop: conformal dimension}, we obtain:
\begin{corollary}\label{cor: dimH<1}
	Let $G\sbs\Diff_+^\omega(\SS^1)$ be a finitely generated subgroup with an exceptional minimal set $\Lambda.$
	Then $0<\dimH\Lambda<1.$
\end{corollary}

We will prove Proposition \ref{prop: conformal dimension} by establishing two inequalities in Lemmas \ref{lem: delta geq dim} and \ref{lem: delta leq dim}.
For the rest of this subsection, let $G$ and $\Lambda$ be as in Proposition~\ref{prop: conformal dimension}. 
Let $\nu$ be a (not necessarily atomless) $\delta$-conformal measure supported on $\Lambda$.
Since $\supp\nu$ is $G$-invariant, we have $\supp\nu=\Lambda.$
Write $\Lambda'=\Lambda\sm G(\NE).$
Recall the constant $\ve_0 > 0$ from Lemma \ref{lem: acceleration}, and for each $z\in\Lambda'$ and $n>0$, the radius $r_n(z)$ given in Definition \ref{def: rnz}.

A common tool used in both these lemmas is an estimate of the conformal measure of certain well-chosen intervals.
\begin{lemma}\label{lem: conformal measure good interval}
	There exists a constant $c'=c'(G,\nu)>0$ such that the following holds.
	For every $z\in\Lambda'$, every $n\in\Acc(z)$ (Definition \ref{def: Acceleration times}), and the interval $I= B(z, r_n(z)),$ we have 
	\[c'|I|^\delta\leqslant \nu(I)\leqslant {c'}^{-1}|I|^\delta.\]
\end{lemma}
\begin{proof}
	By conformality and Lemma \ref{lem: rnz} \eqref{ppt: rnz vk},
	\[\nu(t(n,z)I)=\int_I [t(n,z)']^\delta\dd\nu\asymp_{\ve_0} [r_n(z)]^{-\delta} \nu(I)\asymp |I|^{-\delta}\nu(I). \]
	Hence $\nu(I)\asymp |I|^\delta \nu(t(n,z)I).$
    It therefore suffices to show that $\nu(t(n,z)I)\asymp_{\nu} 1.$
	By Lemma \ref{lem: rnz} \eqref{ppt: rnz interval},
	\[ B(t(n,z)z,\ve_0/2)\sbs t(n,z)I \] 
	Since $t(n,z)z\in \Lambda=\supp\nu,$ we have 
	\[\nu(t(n,z)I)\geqslant \inf_{x\in\supp\nu}\nu(B(x,\ve_0/2))>0,\]
	while of course $\nu(t(n,z)I)\leqslant 1.$ Combining these bounds yields the desired estimate.
\end{proof}

\begin{lemma}
\label{lem: delta geq dim}
We have $\delta\geq\dim_{\mr H}\Lambda.$
\end{lemma}
\begin{proof}
Note that $\dim_{\mr H}\Lambda'=\dim_{\mr H}\Lambda$ by Theorem \ref{thm: DKN09} \eqref{fact: NE finite}.
	We aim to show that $\cH_r^\delta(\Lambda')$ is uniformly bounded in $r,$ which implies that the Hausdorff outer measure $\cH^\delta(\Lambda')$ is finite and therefore $\dimH\Lambda'\leqslant\delta.$ 
	
	Given $r>0,$ consider the collection of intervals
	\[\cE_r\defeq\lb{B(z,r_n(z)): z\in \Lambda',n\in\Acc(z),r_n(z)<r/10 }.\]
	Then $\cE_r$ forms a cover of $\Lambda'$ by intervals of length at most $ r/5.$
	By the Vitali covering lemma, there exists a countable, pairwise disjoint subcollection $\cE_r'\sbs\cE_r$ such that 
	\[\Lambda'\sbs \bigcup_{I\in\cE_r}I\sbs \bigcup_{I\in\cE_r'}(5I). \]
	Hence $\lb{5I:I\in\cE_r'}$ is a cover of $\Lambda'$ by intervals of length at most $r,$ and thus
	\[\cH_r^\delta(\Lambda')\leqslant \sum_{I\in\cE_r'}|5I|^\delta \ll \sum_{I\in\cE_r'}|I|^\delta.\]
	
	By Lemma \ref{lem: conformal measure good interval}, there exists $c'>0$ independent of $r$ such that $\nu(I)\geqslant c' |I|^\delta.$
	Combined with the disjointness among $\cE'_r$, we obtain
	\[\cH_r^\delta(\Lambda')\ll\sum_{I\in\cE_r'}|I|^\delta \ll_{\ve_0,c'} \sum_{I\in\cE_r'}\nu(I)\leqslant 1. \]
	Thus $\cH^\delta(\Lambda')=\lim_{r\to 0^+}\cH_r^\delta(\Lambda')\ll_{\ve_0,c'}1. $
  This concludes the proof of the lemma.
\end{proof}

\begin{lemma}
\label{lem: delta leq dim}
If $\nu$ is moreover atomless, then $\delta\leq\dim_{\mr H}\Lambda.$
\end{lemma}
\begin{proof}
	For each positive integer $m,$ we consider the collection of intervals
	\[\cE_m\defeq\lb{B(z,r_n(z)/10): z\in \Lambda',n\in\Acc(z), r_n(z)\in [e^{-(m+1)},e^{-m}] }.\]
	Let $\cE_m' \sbs \cE_m$ be a finite subcollection of maximum cardinality consisting of pairwise disjoint intervals.
    Then for every $J\in \cE_m$ there exists $I\in\cE_m'$ with $J\cap I\ne\vn.$
    Since $|J|\leqslant e|I|,$ we have $J\sbs 10I.$
	Thus
	\[\bigcup_{I\in \cE_m}I\sbs \bigcup_{I\in\cE_m'} (10I).  \]
	By Lemma \ref{lem: rnz}, for each $z$ in $\Lambda'$, $\Acc(z)$ is infinite and $r_n(z)$ decreases to $0$ along $n\in\Acc(z).$
	This implies that each $z\in\Lambda'$ lies in $ \bigcup_{I\in \cE_m}I$ for infinitely many $m>0.$
	
	Since $\nu$ is atomless, we have $\nu(\Lambda')=1.$
	Therefore,
	\begin{equation}\label{eqn: sum conformal measure}
		\sum_{m=1}^\infty \sum_{I\in\cE_m'}\nu(10I)\geqslant \sum_{m=1}^\infty \nu\sb{\bigcup_{I\in\cE_m'} (10I) }\geqslant \sum_{m=1}^\infty\nu\sb{\bigcup_{I\in \cE_m}I}= +\infty,
	\end{equation}
	Note that for every $I\in\cE_m',$ the interval $10I$ is of the form $B(z,r_n(z)).$
	By Lemma \ref{lem: conformal measure good interval}, there exists $c'>0$ such that 
	\[ \nu(10I)\leqslant {c'}^{-1}\cdot |10I|^{\delta}\ll {c'}^{-1}e^{-\delta m},\quad \forall I\in\cE_m'.\]
	Combining this with \eqref{eqn: sum conformal measure}, we obtain 
    \[\sum_{m=1}^\infty e^{-\delta m}\#\cE_m'\gg_{c'}\sum_{m=1}^\infty \sum_{I\in\cE_m'}\nu(10I) = +\infty,\] 
    and hence $\limsup_{m\to\infty}\frac{1}{m}\log\#\cE_m'\geqslant\delta.$
	By Lemma \ref{lem: rnz} \eqref{ppt: rnz derivative}, the family $\cE_m'$ satisfies the hypotheses of Lemma \ref{lem: large attractor} with $C=20\ve_0^{-1} ,\alpha =\delta,$ and $g_I=t(n,z)$ for each $I=B(z,r_n(z)/10)\in\cE_m'$.
	Thus Lemma \ref{lem: large attractor} yields a subset $\Lambda_0\sbs\Lambda$ of  Hausdorff dimension at least $\delta-\ve$ for arbitrarily small $\ve>0.$
	This would lead to a contradiction if $\delta$ were greater than $\dimH\Lambda.$
\end{proof}

\subsection{Lower mass bounds for conformal measures and box dimension of the minimal set}
In this subsection, we establish an anti-Frostman estimate for the conformal measure.
This will be used in the study of the box dimension of the exceptional minimal set.

Let $G\sbs\Diff_+^\omega(\SS^1)$ be a finitely generated subgroup with an exceptional minimal set $\Lambda$, which satisfies property $(\Lambda\star)$ by Theorem \ref{thm: DKN18}.
Let $\Lambda'=\Lambda\sm G(\NE)$, so that $\dim_{\mr H}\Lambda'=\dim_{\mr H}\Lambda$ by Theorem \ref{thm: DKN09} \eqref{fact: NE finite}.
Let $\nu$ be an atomless conformal measure supported on $\Lambda$, whose existence was proved in Corollary \ref{cor: construct conformal measures}.
Let $\delta$ be its conformal exponent. 
Our aim in this subsection is to prove the following proposition.

\begin{proposition}\label{prop: conformal lower local dim}
	There exists a constant $c_0>0$ such that for all sufficiently small $r>0$ and all $z\in\Lambda',$ we have $\nu(B(z,r))\geqslant c_0\cdot r^\delta.$ 
\end{proposition}

\begin{corollary}\label{cor: box}
	The box dimension of $\Lambda$ exists and satisfies $\dim_{\mr B}\Lambda=\dim_{\mr H}\Lambda.$
\end{corollary}
\begin{proof}
	By Lemma \ref{lem: measure box}, we have $\ol{\dim}_{\mr B}\Lambda'\leqslant \delta.$
	By Proposition \ref{prop: conformal dimension}, $\delta=\dimH\Lambda.$
	Combining these with \eqref{eqn: ineq box Hausdorff} yields 
	\[\dimH\Lambda = \dimH\Lambda'\leqslant  \ul{\dim}_{\mr B}\Lambda'\leqslant \ol{\dim}_{\mr B}\Lambda'\leqslant \delta=\dimH\Lambda.  \]
	Hence the box dimension of $\Lambda'$ exists and coincides with $\dimH\Lambda'.$
	Finally, since $\Lambda$ is the closure of $\Lambda',$ we obtain $\dim_{\mr B}\Lambda =\dim_{\mr B}\Lambda'=\dimH\Lambda.$
\end{proof}

It remains to prove Proposition \ref{prop: conformal lower local dim}.
By Lemma~\ref{lem: rnz}, the sequence of radii $\lb{r_n(z)}_{n \in \NN}$ satisfies $r_0(z)=\ve_0$, $\liminf_{n\to\infty} r_n(z)=0$ and the ratios $r_n(z)/r_{n+1}(z)$ are uniformly bounded for $n\in\NN$ and $z\in\Lambda'.$
Thus the proof of Proposition~\ref{prop: conformal lower local dim} reduces to the following lemma.
\begin{lemma}\label{lem: conformal lower local dim}
	There exists a constant $c_0'>0$ such that for every $z\in\Lambda'$ and $ n\geqslant 1$, if $ r=r_n(z)$ is small enough then $\nu(B(z,r))\geqslant c_0' \, r^\delta.$
\end{lemma}

We first explain the idea of the proof of Lemma \ref{lem: conformal lower local dim}.
We retain the notation from Section \ref{sec: group actions on the circle}.
In the previous subsection, Lemma \ref{lem: conformal measure good interval} shows that whenever $n \in \Acc(z)$, one has $\nu(B(z,r)) \asymp r^\delta$. 

When $n\notin\Acc(z)$, the preceding argument does not provide uniform distortion control for $t(n,z)$ on $B(z,r_n(z))$.
Instead, we consider the image of  $B(z,r_n(z))$ under the map $t(n',z)$, where $n' \in \Acc(z)$ is the largest element less than $n$.
By Lemma \ref{lem: parabolic expanding}, there exist $x \in \NE \cap \Lambda$ and $\sigma \in \{+,-\}$ such that $t(n'+\ell-1, z)z \in I_x^\sigma\sm J_x^\sigma$ and $t(n'+\ell,z) = g_{x,\sigma}^\ell \circ t(n',z)$ for all $1 \leq \ell \leq n-n'$.
Since $x$ is a parabolic fixed point of $g_{x,\sigma}$ and $t(n',z)(B(z,r_n(z)))$ is an interval close to $x$, estimating $\nu(B(z,r_n(z)))$ reduces to understanding the conformal measure near a parabolic fixed point in $\Lambda$.

\paragraph{Conformal measure around parabolic fixed points.}
Let $f\in G\sbs\Diff_+^\omega(\SS^1)$ be an element with a parabolic fixed point $x\in\Lambda$ of multiplicity $k+1$.
Assume that $f$ is contracting on an  interval ${]x,y]}$ with ${]x,y]}\cap \Lambda\ne\vn.$
Fix a point $y_0\in {]x,y]}$.
For each positive integer $m,$ let $y_m=f^m(y_0)$ and $I_m=[y_m,y_{m-1}],$  as in Lemma \ref{lem: around parabolic}.
\begin{lemma}\label{lem:nuIm}
	There exists a constant $c_1 >0$ depending only on $\nu, f$ and the interval ${]x,y_0]}$ such that
    \[
    \forall m\geqslant 1, \quad \nu(I_m)\geqslant c_1 \abs{I_m}^\delta.
    \]
\end{lemma}
\begin{proof}
	By conformality and Lemma \ref{lem: around parabolic} \eqref{ppt: para 2}, we have
	\[\nu(I_m)=\int_{I_1} (f^{m-1})'(s)^\delta \dd\nu(s)\gg_{C} \nu(I_1) \, m^{-\frac{k+1}{k}\delta},\]
	where $C>0$ depends only on $f$ and on ${]x,y_0]}.$
    By  Lemma \ref{lem: around parabolic} \eqref{ppt: para 1}, we have
    \( m^{-\frac{k+1}{k}\delta} \gg_{C} \abs{I_m}^\delta\).
\end{proof}

\begin{lemma}
	There exist constants $c_2,C_1>0,$ depending only on $\nu, f$ and the interval ${]x,y_0]}$, such that for every $m\geqslant 2,$ every $z\in I_m,$ and every radius $ C_1 m^{-(k+1)/k}\leqslant r\leqslant 1 ,$ we have $\nu(B(z,r))\geqslant c_2 \, r^\delta.$
\end{lemma}
\begin{proof}
    Let $\cN$ be the set of $n \geq 1$ such that $I_{n} \sbs [z, z+ r[$.
    Since $\nu$ is atomless, we have
    \[\nu(B(z,r)) \geq \nu([z, z+ r[) \geq \sum_{n \in \cN} \nu(I_n).\]
    In view of Minkowski's inequality ($(a+b)^\delta\leqslant a^\delta+b^\delta$) and Lemma~\ref{lem:nuIm}, it suffices to show that
    \begin{equation}\label{eq:sumImr}
    \sum_{n \in \cN} \abs{I_n} \gg_{f, x, y_0} r.        
    \end{equation}

    By Lemma \ref{lem: around parabolic} \eqref{ppt: para 1}, we have
    \( \abs{I_n} \asymp_C n^{-\frac{k+1}{k}} \) for every $n \geq 1$, where $C>0$ depends only on $f$ and on ${]x,y_0]}$.
    Thus, choosing $C_1$ sufficiently large, we can ensure that $\cN$ is nonempty with $m - 1 \in \cN$.

    Let $n_0 = \min \cN$. 
    If $n_0 = 1$, that is, $1 \in \cN$, then \(r \leq 1\) already implies \eqref{eq:sumImr}.
    Otherwise, we have $n_0 \geq 2$, meaning that $I_{n_0 - 1} \notin \cN$.
    It follows that
    \[
      [z,z+r[ \subset \bigcup_{n \in \cN \cup\{n_0 - 1, m\}} I_n.
    \]
    This also implies \eqref{eq:sumImr} because $\abs{I_{n_0 - 1}} \ll_f \abs{I_{n_0}}$ and $\abs{I_{m}} \leq \abs{I_{m - 1}}$ by the fact that $f$ contracts.
\end{proof}

\begin{corollary}\label{cor: parabolic measure}
	For every $c>0$ there exists $c_3>0$ depending only on $c,\nu,f$ and the interval ${]x,y_0]},$ such that for every $m\geqslant 2,$ every $z\in I_m\cap\Lambda ,$ and every radius $ c \, m^{-(k+1)/k}\leqslant r\leqslant 1 ,$ we have $\nu(B(z,r))\geqslant c_3\,  r^\delta.$
\end{corollary}
\begin{proof}
	It suffices to prove the corollary for $c>0$ sufficiently small. 
	By Lemma \ref{lem: around parabolic} \eqref{ppt: para 1}, for $c>0$ small enough we have $B(z,c \, m^{-(k+1)/k})\sbs {]fz,f^{-1}z[}\sbs I_{m+1}\cup I_m\cup I_{m-1}.$ 
	By Lemma \ref{lem: around parabolic} \eqref{ppt: para 2}, 
	\[(f^{m-2})'|_{I_1\cup I_2\cup I_3}\asymp_{C} m^{-\frac{k+1}{k}},\]
    where $C>0$ depends only on $f$ and $]x,y_0]$.
	Therefore, there exists $\ve>0$, depending only on $c$ and $C$, such that
	\[f^{-(m-2)}B(z,c \, m^{-\frac{k+1}{k}})\sps B(f^{-(m-2)}z,\ve).\]
	Since $\supp\nu=\Lambda$, there exists $c_3'>0$, depending only on $\nu$ and $\ve$, such that $\nu(B(z',\ve))>c_3'$ for all $z'\in\Lambda.$  
	Set $z'\defeq f^{-(m-2)}z \in\Lambda.$ 
	We obtain
	\[\nu(B(z,c \, m^{-\frac{k+1}{k}}))\geqslant \int_{B( z',\ve)} [(f^{m-2})'(s)]^\delta\dd\nu(s)\gg_{C} \nu( B( z',\ve))\cdot m^{-\frac{k+1}{k}\delta}\geqslant c_3'\cdot m^{-\frac{k+1}{k}\delta}.\]
	Let $c_2,C_1>0$ be the constants given by the previous lemma.
	Then there exists $0<c_3\leqslant c_2,$ depending only on $c,\nu,f$ and ${]x,y_0]},$ such that  
	\[\nu(B(z,r))\geqslant c_3 \, r^\delta,\quad \forall r \in  \bigl[c m^{-(k+1)/k}, C_1 m^{-(k+1)/k}\bigr].\]
	For $r\geqslant C_1 m^{-(k+1)/k},$ the conclusion follows directly from the previous lemma.
\end{proof}

\begin{proof}[Proof of Lemma \ref{lem: conformal lower local dim}.]
	Since $G\sbs\Diff_+^\omega(\SS^1),$ we apply the expanding strategy in Section \ref{sec: group actions on the circle} with the convention mentioned in Remark \ref{rem: Ix real analytic}.
	Let $n'\leqslant n $ be the largest integer in $\Acc(z)$, with the convention that $n'=0$ if $[1,n]\cap\Acc(z)=\vn$.
	If $n'=n,$ the estimate follows from Lemma \ref{lem: conformal measure good interval}.
	
	We henceforth assume that $n'<n$.
    Set $z'= t(n',z)z$ and $n''\defeq n - n' \geq 1$.
    By Definition \ref{def: Acceleration times}, the next acceleration time of $z$ is $n(z')+n'$, which is strictly larger than $n$ by our choice of $n'$. 
    Then the first acceleration time of $z'$ satisfies $n_1(z')=n(z')\geqslant  n-n'+1\geqslant 2.$
    This implies that $z'\in I_x^\sigma\sm J_x^\sigma$ for some $x\in\NE\cap\Lambda$ and $\sigma\in\lb{+,-}$ by Definition \ref{def: Acceleration times}. 
    Then by Lemma \ref{lem: parabolic expanding}, for every $1\leqslant \ell \leqslant n''\leqslant n(z')-1 ,$ one has $t(\ell-1,z')z'\in I_x^\sigma$ and $t(\ell,z')=g_{x,\sigma}^\ell.$
    
    By our choice that $g_{x,\sigma}'|_{I_x^\sigma}>1,$ we have $r=r_n(z)<r_{n'}(z)$ and hence $\vk(t(n',z),B(z,r))\leqslant \log 2 $.
    Let $r'\defeq t(n',z)'(z)\cdot  r/2.$
	Then 
	\[t(n',z)B(z,r)\sps B(z',r'). \]
	By Definition \ref{def: rnz}, we have
	\[r^{-1}= r_n(z)^{-1}=\ve_0^{-1}t(n,z)'(z)=\ve_0^{-1}\cdot t(n'',z')'(z')\cdot t(n',z)'(z),\]
	and hence
	\begin{equation}\label{eqn: r'}
	r'= \frac{\ve_0}{2}\cdot[t(n'',z')'(z')]^{-1}=\frac{\ve_0}{2}\cdot [(g_{x,\sigma}^{n''})'(z')]^{-1}.
	\end{equation}
    
    Recall that our initial choice of the side intervals ensures that $g_{x,\sigma}^{-1}$ is contracting on $I_x^\sigma$. Moreover, $x$ is a parabolic fixed point of $g_{x,\sigma}$.
    By \eqref{eqn: decompose I into J}, we have the disjoint decomposition
    \[I_x^\sigma=\bigcup_{\ell= 0}^\infty J_\ell ,\]
    where $J_\ell =J_{\ell,x,\sigma}\defeq g_{x,\sigma}^{-\ell}J_x^\sigma.$
    Assume that $z'\in J_m$.
    Then $m=n(z')-1\geqslant n''$ by Definition \ref{def: Acceleration times}.
    Applying Lemma \ref{lem: around parabolic} \eqref{ppt: para 2} to $f=g_{x,\sigma}^{-1}$ and $I_{\ell}=J_{\ell-1}$, there exists a constant $C=C_{x,\sigma}>0$ depending only on $g_{x,\sigma}$ and $I_x^\sigma$ such that 
    \[ [(g_{x,\sigma}^{n''})'(z')]^{-1}=(g_{x,\sigma}^{-n''})'( g_{x,\sigma}^{n''}z')\asymp_C \sb{\frac{m-n''+1 }{m+1}}^{\frac{k+1}{k}}\gg (m+1)^{-\frac{k+1}{k}},\]
    where $k + 1$ is the multiplicity of $g_{x,\sigma}$ at the fixed point $x.$
    Combining this with \eqref{eqn: r'} yields $r'\gg_{\ve_0,C} (m+1)^{-(k+1)/k}.$
    
    Note that $z'\in J_m$ with $m\geqslant n''\geqslant 1.$
    Taking $I_\ell=J_{\ell-1}$ in Corollary \ref{cor: parabolic measure}, there exists $c_3>0,$ depending only on $\ve_0,\nu$ and $I_x^\sigma,g_{x,\sigma},$ such that 
	\[\nu(B(z',r'))\geqslant c_3\cdot {r'}^{\delta}.\]
	Finally, use again that $\vk(t(n',z),B(z,r))\leqslant \log 2$ and recall that $t(n',z)^{-1}B(z',r')\sbs B(z,r)$ with $r'=t(n',z)'(z)\cdot  r/2.$
	We have
	\begin{align*}
	\nu(B(z,r))\geqslant& \int_{B(z',r')} [(t(n',z)^{-1})'(s)]^\delta\dd\nu(s)\geqslant \frac{1}{2} \cdot [t(n',z)'(z)]^{-\delta}\cdot \nu(B(z',r'))\\
	\label{eqn: measure case 2}\geqslant & \frac{c_3}{2} [t(n',z)'(z)]^{-\delta}\cdot(r')^\delta\geqslant \frac{c_3}{4}\cdot r^\delta,
	\end{align*}
	Since $\NE(G)\cap\Lambda$ is finite, only finitely many such constants $c_3$ arise.
    Let $c_0'$ be the minimum of these values $c_3/4$ and the constant $c'$ from Lemma \ref{lem: conformal measure good interval}.
    This completes the proof.
\end{proof}

\section{Proofs of the main results and further discussions}\label{sec: 6}

\begin{proof}[Proof of the Main Theorem]
    By Lemma \ref{lem: invariant measure and finite orbit} and Denjoy's theorem, $G$ does not preserve a probability measure on $\SS^1.$
    By Theorems \ref{thm: DKN18} and \ref{thm: real analytic with exceptional}, $G$ is virtually free, locally discrete and satisfies property $(\Lambda\star).$
    
    Statement \eqref{mainthm: 1} is given by Corollary \ref{cor: box}. 
    Statement \eqref{mainthm: 2} follows from a combination of Propositions \ref{prop: Hausdorff upper bound} and \ref{prop: Hausdorff lower bound}, and statement \eqref{mainthm: 3} is Corollary \ref{cor: dimH<1}.
    
    The first part of \eqref{mainthm: 4} follows from Proposition \ref{prop: parabolic critical exponent}, which, together with \eqref{mainthm: 3}, implies that the multiplicity of fixed points on $\Lambda$ is uniformly bounded.
    To control the multiplicity of fixed points outside $\Lambda,$ \cite[Corollary 1.18]{DKN18} implies that $\SS^1\sm\Lambda$ is a union of finitely many orbits of intervals.
    By Theorem \ref{thm: real analytic with exceptional} \eqref{fact: Hector}, every nontrivial element of $G$ fixing a point outside $\Lambda$ is conjugate to a nonzero power of an element of a fixed finite subset of $G$.
    Hence the multiplicity is uniformly bounded.
\end{proof}

\begin{proof}[Proof of Theorem \ref{thm: C infty example}.]
As mentioned in the introduction, the construction is based on the existence of a parabolic fixed point of infinite multiplicity.
The following lemma is parallel to Lemma \ref{lem: around parabolic}.
\begin{lemma}\label{lem: appd lem cal}
There exists an odd $C^\infty$ function $\phi: [-1/100, 1/100]\to (-1/100, 1/100)$ 
such that the following hold:
\begin{enumerate}
    \item $\phi'(0)=1$ and $|\phi(x)|<|x|$ for every $x\ne 0$.
    \item $(\phi^n)'(x)\gg \frac{1}{n(1+\log n)^2}$ for every $x\in [-1/100,1/100]$ and every $ n\geq 1.$
\end{enumerate}
\end{lemma}
\begin{proof}
    Define $\phi:[-1/100,1/100]\to [-1/100,1/100]$ by 
    \[\phi(x)=\case{&x-x^2e^{-1/x},& x>0;\\ &0,&x=0; \\&-\phi(-x),& x<0.} \]
    Then $\phi$ is $C^\infty$, odd,  satisfies $\phi'(0)=1$ and $|\phi(x)|<|x|$ for every $x\ne 0.$

    Let $x_0=1/100$ and define $x_n=\phi^n(x_0).$
    Set $h(x)=e^{1/x}$ and write $z_n=h(x_n).$ 
    A direct computation shows that for $z\geq z_0=e^{100}$,
    \[\wt \phi(z)\defeq h\circ \phi\circ h^{-1}(z)=\exp\mb{\sb{\frac{1}{\log z}- \frac{1}{z(\log z)^2}}^{-1} },\]
    so that $z_n={\wt \phi}^{n}(z_0).$
    Then $z_n\to +\infty$ as $n\to+\infty.$
    Moreover, since $\lim_{z\to+\infty} \wt \phi(z)-(z+1)=0,$ we have $|z_{n+1}-z_n| \to 1,$ and consequently $z_n=(1+o(1))n.$
    It follows that
    \[x_n=\frac{1}{\log ( (1+o(1))n)},\text{ and } x_n-x_{n+1}=x_n^2e^{-1/x_n}=\frac{1 }{(1+o(1))n\cdot (\log n)^2 }.\]
    In view of \eqref{eqn:vksumI}, for every $m,n\geqslant 0, $ we have
    \[\vk(\phi^n,[x_{m+1},x_m])\leqslant\sum_{i=0}^{n-1}\vk(\phi,[x_{m+i+1},x_{m+i}])\leqslant M\sum_{i=0}^{n-1}( x_{m+i}-x_{m+i+1})\leqslant Mx_0=\frac{M}{100}, \]
    where $M$ is a constant depending only on $\phi.$
    Therefore,
    \begin{equation*}\label{eqn: infty multiple fixed point}
        (\phi^n)'(x)\asymp_M \frac{\phi^n(x_m)-\phi^n(x_{m+1})}{x_m-x_{m+1}} = \frac{x_{m+n}-x_{m+n+1} }{x_m-x_{m+1}}\gg \frac{1}{n(1+\log n)^2},\quad\forall x\in[x_{m+1},x_m],
    \end{equation*}
    for every $n\geqslant 1$ and $m\geqslant 0$.
    Since $(0,1/100]=\bigcup_{m\geqslant 0}[x_{m+1},x_m]$ and $\phi$ is odd, this proves the lemma.
\end{proof}

    We now construct a subgroup $G\sbs\Diff_+^\infty(\SS^1)$ using the function given by Lemma \ref{lem: appd lem cal}.
    Let $U_1^\pm, U_2^\pm\sbs\SS^1$ be four open intervals of length $1/10$ with pairwise disjoint closures. 
    Choose a point $a_1\in U_1^+$ such that $B(a_1,1/50)\subset U_1^+$, and define $f_0:B(a_1,1/100)\to \SS^1$ by $f_0(x)=\phi(x-a_1)+a_1$, where $\phi$ is the function given in Lemma \ref{lem: appd lem cal}. 
    Then $f_0$ is a $C^\infty$ diffeomorphism with $a_1$ as its unique fixed point, which is topologically attracting.
    To obtain ping-pong dynamics, we choose diffeomorphisms $f_1,f_2\in\Diff_+^\infty(\SS^1)$ satisfying:
    \begin{enumerate}[\indent (a)]
        \item\label{ass: C^infty example a} $f_1|_{B(a_1,1/100)}=f_0$;
        \item\label{ass: C^infty example b} for each $i=1,2$, the map $f_i$ has exactly two fixed points: a topologically attracting fixed point $a_i\in U_i^+$ and a topologically repelling fixed point $r_i\in U_i^-$.
        Moreover, $a_2,r_1,r_2$ are hyperbolic;
        \item\label{ass: C^infty example c} for $i=1,2,$ we have $f_i(\SS^1\sm U_i^-)\sbs U_i^+$ and $f_i^{-1}(\SS^1\sm U_i^+)\sbs U_i^-.$
    \end{enumerate}
   
   Let $G\defeq\pair{f_1,f_2}\sbs\Diff_+^\infty(\SS^1)$. 
   By \eqref{ass: C^infty example b} and \eqref{ass: C^infty example c}, $G$ is a free group freely generated by $f_1,f_2$ and preserves an exceptional minimal set $\Lambda\sbs U_1^+\cup U_1^-\cup U_2^+\cup U_2^-.$
   Moreover, $G$ admits no invariant probability measure on $\SS^1.$
   Now we show the local discreteness of $G.$
   \begin{lemma}
   The ping-pong property \eqref{ass: C^infty example c} implies that $G$ is $C^1$-locally discrete. 
   \end{lemma}
   \begin{proof}
   Let $K=\ol{U_1^+\cup U_1^-\cup U_2^+\cup U_2^-}$ and let $x_0\notin K.$
   Then by \eqref{ass: C^infty example c}, $G$ has a unique minimal set $\Lambda\sbs K$ and for every $g\ne\Id\in G,$ $gx_0\in K.$
   
   Let $I$ be an arbitrary open interval intersecting $\Lambda.$
   Since $\Lambda\sbs \ol{G x_0},$ there exists $h\in G$ such that $hx_0\in I.$
   For each $g\ne\Id\in G$, we have $g(hx_0)=h(h^{-1}gh)x_0\in hK.$
   Therefore, we have $d(g(hx_0),hx_0)\geqslant d(hK,h x_0),$ where the right hand side is a positive constant independent of the choice of $g\ne\Id\in G.$
   Hence there is no sequence of distinct elements of $G$ that $C^0$-converges to $\Id $ on $I,$ which proves that $G$ is locally discrete. 
   \end{proof}
   
   Since $f_1^n(x)\to a_1$ for every $x\ne r_1$ as $n\to +\infty,$ we have $a_1\in\Lambda.$ 
   By Proposition \ref{prop: Hausdorff lower bound} and Corollary \ref{cor: pointwise DCE},
    \[\dimH\Lambda\geqslant\delta(G)=\delta(G,a_1)\geqslant\delta(\pair{f_1},a_1).\]
    By Lemma \ref{lem: appd lem cal}, we have $(f_1^n)'|_{B(a_1,1/100)}\gg \frac{1}{n(1+\log n)^2}$. 
    Therefore 
    \[\delta(\pair{f_1},a_1)\geqslant \limsup_{m\to+\infty}\frac{1}{m}\log\# \lb{n\geqslant 0: (f_1^n)'|_{B(a_1,1/100)}\geqslant e^{-m} }=1, \]
    which implies $\dimH\Lambda=1.$
\end{proof}

\begin{remark}\label{rem: no conformal measures}
	The maps $f_1,f_2$ can moreover be chosen so that 
	\begin{enumerate}[\indent(d)]
		\item $f_1'<1$ on $(\SS^1\sm U_1^-)\sm\{a_1\}$, $f_2'<1$ on $\SS^1\sm U_2^-$, and $(f_i^{-1})'<1$ on $\SS^1\sm U_i^+$ for $i=1,2$.
	\end{enumerate}
    With this additional requirement, the group $G$ also satisfies property $(\Lambda\star)$.
    Indeed, the ping-pong dynamics together with this additional requirement implies that every point of $\Lambda\sm\{a_1\}$ is expanded by one of $f_1^{\pm1},f_2^{\pm1}$.
    Thus $\NE\cap\Lambda\subset\{a_1\}$, and $a_1$ is an isolated fixed point of $f_1$.
    Consequently, $G$ does not admit any atomless conformal measure supported on $\Lambda$.
    Otherwise, the conformal exponent $\delta$ would coincide with $\dimH\Lambda=1$ by Proposition~\ref{prop: conformal dimension}, which contradicts Theorem~\ref{thm: DKN09 thmF}.
\end{remark}

\begin{proof}[Proof of Corollary~\ref{cor: orbit closure classifications}]
    If $G$ has no finite orbits, then either $G$ acts minimally or $G$ has an exceptional minimal set $\Lambda$ with $\dimH\Lambda=\delta(G)$. 
    The first situation is covered by case \eqref{orbit closure case 2}.
    In the latter situation, Theorem~\ref{thm: real analytic with exceptional} \eqref{fact: Hector} implies that the orbit closure of any point $x$ is of the form $Gx\cup \omega(x)$, where $\omega(x)$ consists of $\Lambda$ together with at most two $G$-orbits.
    This is included in case~\eqref{orbit closure case 3}.

    Now suppose that $G$ has a finite orbit $F=\lb{x_i:i\in\ZZ/n\ZZ},$ where $x_0,\cdots,x_{n-1}$ are arranged in cyclic order and $G$ acts transitively on $F.$ 
    For each $g\in G$, there exists a unique translation number $\tau(g)\in\ZZ/n\ZZ$ such that $g(x_i)=x_{i+\tau(g)}.$ 
    This defines a surjective group homomorphism $\tau:G\to\ZZ/n\ZZ$.
    Let $G_1=\ker\tau$, which is a normal subgroup of $G$. 
    If $G_1$ is not isomorphic to $\ZZ$, then by \cite[Proposition~3.5]{Na06}, every orbit closure of $G$ is either a finite set, a finite union of closed intervals, or the whole circle.
    In this case, the action again falls within case~\eqref{orbit closure case 2}.

    We now consider the remaining case where $G_1=\langle f\rangle$ for some nontrivial element $f\in G$ fixing every point of $F$. 
    Then each $G$-orbit is a finite union of $\langle f\rangle$-orbits, and hence every $G$-orbit closure is either finite or countable.
    Since $G/G_1\cong \ZZ/n\ZZ$, there exists an element $g\in G$ such that $g(x_i)=x_{i+1}$ for every $i\in\ZZ/n\ZZ.$ 
    As $g^n\in\langle f\rangle$, there exist integers $s,m$ such that 
    \[g f g^{-1} = f^s\quad \text{and}\quad g^n = f^m.\]
    Moreover, every element of $G$ can be written in the form $f^k g^\ell$ with $k,\ell\in\ZZ$.
    Since $g^n$ commutes with $f$, we have $f = g^nfg^{-n}=f^{s^n}$.
    Since $f$ has infinite order, it follows that $s^n=1$, and therefore only the following two cases can occur:
    \begin{enumerate}
        \item If $s=1$, then $G$ is abelian. 
        More precisely, $G\cong \ZZ\times \ZZ/k\ZZ$ where $k=\gcd(m,n).$
        \item If $s=-1$, then $n=2k$ for some positive integer $k$.
        Conjugating the relation $g^n=f^m$ by $g$, we obtain
        \[f^m = g^n = gf^mg^{-1} = f^{ms} = f^{-m}.\]
        Since $f$ is not torsion, this implies $m=0$.
        In this case, $G$ admits the presentation $\pair{a,b\mid bab^{-1}=a^{-1},b^{2k}=1}$.\qedhere
    \end{enumerate}
\end{proof}
\begin{remark}\label{rem: finite orbit example}
    The case $G\cong \pair{a,b\mid bab^{-1}=a^{-1},b^{2k}=1}$ can indeed be realized as a subgroup of $\Diff_+^\omega(\SS^1).$ 
    Let $g$ be the rotation $x\mapsto x+\tfrac{1}{2k}$ on $\SS^1=\RR/\ZZ$.
    We now construct $f$. 
    Consider the real-analytic vector field $X(x)=\sin(2k\pi x)$ on $\RR/\ZZ,$ which satisfies $X(x+1/2k)=-X(x).$ 
    Let $\phi_t$ be the flow generated by $X(x).$ 
    Then $\phi_t\in\Diff_+^\omega(\SS^1)$ for all $t$, and $\phi_t$ fixes the finite set $\lb{\ell/2k:\ell=0,\dots,2k-1}.$ 
    Moreover, the symmetry of $X$ implies that $\phi_t\circ g=g\circ \phi_{-t}.$ 
    Setting $f=\phi_1$, we obtain $g f g^{-1}=f^{-1}$.
    Therefore, $f$ and $g$ generate a subgroup of $\Diff_+^\omega(\SS^1)$ realizing the desired presentation. 
\end{remark}

To prove Corollary \ref{cor: free orbit closure}, we first establish the following more general statement. 
Let $\Gamma$ be a finitely generated group with a finite generating set $\cG,$ and let $\cS=\cG\cup \cG^{-1}\cup\lb{\Id}$ be the associated symmetric generating set.
The \textit{growth rate} of $\Gamma$ with respect to $\cG$ is defined by
\[\omega_\cG(\Gamma)\defeq\limsup_{n\to\infty} \frac{1}{n}\log\#\cS^{*n}.\]

\begin{corollary}
    Let $\Gamma$ be a finitely generated group with a finite generating set $\cG,$ and let $\rho:\Gamma\to \Diff_+^\omega(\SS^1)$ be a group isomorphism onto its image $G=\rho(\Gamma)$.
    Assume that $\rho(\gamma)'(x)\leqslant e^{\omega_\cG(\Gamma)}$ for all $\gamma\in\cG\cup\cG^{-1}$ and all $x\in \SS^1$.
    Then every $G$-orbit closure is a submanifold. 
\end{corollary}
\begin{proof}
    If $\omega_\cG(\Gamma)=0$ then $\rho(\gamma)$ is a rotation for every $\gamma\in\Gamma,$ and the conclusion is clear.
    We therefore assume $\omega_\cG(\Gamma)>0.$ 
    In this case, $G$ cannot fall into case \eqref{orbit closure case 1}.
    It remains to exclude case \eqref{orbit closure case 3} in Corollary \ref{cor: orbit closure classifications}.
    
    By assumption, for every $n\geqslant 1$ and every $g\in \rho(\cS^{*n}),$ we have $g'(x)\geqslant e^{-n\omega_\cG(\Gamma)}$ for all $x\in\SS^1,$ where $\cS= \cG\cup \cG^{-1}\cup\lb{\Id}.$
    Therefore, 
    \begin{align*}
        \delta(G)&\geqslant \limsup_{n\to\infty}\frac{1}{n\omega_\cG(\Gamma)}\log\#\lb{g\in G: g'(x)\geqslant e^{-n\omega_\cG(\Gamma)},\forall x\in\SS^1 }\\
        &\geqslant\frac{1}{\omega_\cG(\Gamma)} \limsup_{n\to\infty}\frac{1}{n}\log\#\cS^{*n}= 1.
    \end{align*}
    This excludes case \eqref{orbit closure case 3}, in which one has $\delta(G)<1.$
\end{proof}

\begin{proof}[Proof of Corollary \ref{cor: free orbit closure}]
    The first part follows directly from the previous corollary, together with the fact that $\omega_{\cG}(\Gamma)=\log (2\# \cG-1)$ if $\Gamma$ is freely generated by $\cG.$
        
    We now show that this bound is sharp.
    For every $\ve > 0$, one can construct diffeomorphisms
    $\{f_1,\dots,f_n\}\subset \Diff_+^\omega(\SS^1)$ admitting a ping-pong partition given by $2n$ open intervals $\lb{U_i^{\pm}}_{i=1,\cdots,n}$ with pairwise disjoint closures, such that
    \begin{enumerate}[\indent(a)]
        \item $f_i^{\pm 1}(\ol U\sm U_i^{\mp})\sbs U_i^\pm$ for all $i=1,\cdots,n$, where $U=\bigcup_{i=1}^n (U_i^+\cup U_i^-)$, and
        \item $\max \{ \|({f_i}^{-1})' \|_{C^0},  \|f_i'\|_{C^0} \}\leq  (2n-1)+\ve$.
    \end{enumerate}
    By the ping-pong property, the set $\lb{f_1,\cdots,f_n}$ freely generates a free group and preserves an exceptional minimal set $\Lambda\sbs U=\bigcup_{i=1}^n (U_i^+\cup U_i^-).$
\end{proof}

Finally, we comment on statement \eqref{mainthm: 4} and on our definition of the dynamical critical exponent.
The assumption that a parabolic fixed point lies in $\Lambda$ is essential to obtain a lower bound on $\dim_{\mr H}\Lambda$.
This necessity is illustrated by the following example.

\begin{proposition}\label{prop: parabolic outside Lambda}
    For every $\ve>0$ and every $k\geqslant 1,$ there exists a finitely generated subgroup $G\sbs\Diff_+^\omega(\SS^1)$ with an exceptional minimal set $\Lambda$ such that $\dim_{\mr H}\Lambda<\ve$,
    while $G$ contains an element admitting a fixed point of multiplicity at least $(k+1)$.
\end{proposition}
\begin{proof}
    We construct two diffeomorphisms $f_1,f_2\in \Diff_+^\omega(\SS^1)$ with the following properties:
	\begin{itemize}
		\item $f_1$ has three fixed points: one hyperbolic attracting fixed point, one hyperbolic repelling fixed point and one parabolic fixed point of multiplicity $2k$;
		\item $f_2$ has two fixed points: one hyperbolic attracting and one hyperbolic repelling.
	\end{itemize}
    More precisely, define $g_0:\RR/\ZZ\to\RR/\ZZ$ by
	\[g_0(x)=x+\ve' \sin^{2k}(\pi x)\cos (2\pi x),\]
	where $\ve'>0$ is sufficiently small so that $g_0\in\Diff_+^\omega(\SS^1)$. Then $0$ is a fixed point of $g_0$ of multiplicity $2k$, while $1/4$ is a hyperbolic attracting fixed point and $3/4$ is a hyperbolic repelling fixed point.
    Let $h_0\in\Diff_+^\omega(\SS^1)$ have exactly two fixed points: a hyperbolic repelling one at $3/8$ and a hyperbolic attracting one at $5/8$.
    Fix $\delta=1/100$, and define $U_1^+=B(1/4,\delta),U_1^-=B(3/4,\delta), U_2^+=B(5/8,\delta),U_2^-=B(3/8,\delta).$
    These intervals have pairwise disjoint closures.
    Set $U=U_1^+\cup U_1^-\cup U_2^+\cup U_2^-.$ 
    For a positive integer $N,$ define  $f_1=g_0^N$ and $f_2=h_0^N.$
    Assume that $N$ is sufficiently large so that
    \begin{enumerate}[\indent (a)]
        \item $f_i^{\pm 1}(U\sm U_i^\mp)\sbs U_i^\pm$ for $i=1,2$, and
        \item $(f_i^{\pm 1})'|_{U\sm U_i^\mp}<e^{-2/\ve }.$
    \end{enumerate}
    Let $G$ be the group generated by $f_1,f_2.$
    Then $G$ admits a ping-pong partition given by $\{U_i^\pm\}_{i=1,2},$ and  hence preserves an exceptional minimal set $\Lambda\sbs \ol U.$
    Using the ping-pong property, one shows that for every $\ve'>0$ there exists $L=L(\ve')>0$ such that
    \[\lb{g\in G:\exists x \in \Lambda,\, g'|_{B(x,\ve')} \geq e^{-2n/\ve} }\sbs \lb{\Id, f_1,f_2,f_1^{-1},f_2^{-1}}^{*(n+L)} \]
    for all $n\geqslant 0.$
    It follows that $\dimH\Lambda=\delta(G)\leqslant\frac{\log 3}{2\ve^{-1}}<\ve,$ which completes the proof. 
\end{proof}
\begin{remark}\label{rem: delta(G,S1)}
    Proposition \ref{prop: parabolic outside Lambda} also provides an example in which the strict inequality $\delta(G,\SS^1)>\dimH\Lambda$ holds. 
    Indeed, applying Lemma~\ref{lem: around parabolic} on a contracting side of a parabolic fixed point of multiplicity at least $k+1$, and using Lemma~\ref{lem: union critical exp}, we obtain
	\[\delta(G,\SS^1)\geqslant\frac{k}{k+1}>\ve>\dimH\Lambda,\]
	provided that $0<\ve<k/(k+1)$.
    
    More generally, applying the argument of Proposition \ref{prop: construct conformal measures} directly to the original group
    $G\subset \Diff_+^\omega(\SS^1)$ yields an upper bound for $\delta(G,x)$.
    As a consequence, one obtains the following explicit formula:
    \[\delta(G,\SS^1)=\max\lb{\dimH\Lambda,\sup_{f\ne\Id\in G}\frac{k(f)}{k(f)+1}},\]
    where $k(f)+1$ denotes the maximal multiplicity of a fixed point of $f$, and $k(f)=0$ if $f$ has no fixed points. 
    In particular, we still have the strict inequality $\delta(G,\SS^1)<1$ whenever $G$ preserves an exceptional minimal set. 
\end{remark}

\appendix
\section{Free semigroups in hyperbolic groups}\label{app}
This appendix is joint work with Wenyuan Yang.
The main goal is to prove Proposition \ref{prop: free semigroup} used in the main part of this article.
More precisely, we prove Proposition~\ref{prop: free semigroup-loxodromic} below, which expresses the assumption on $F$ in Proposition~\ref{prop: free semigroup} in terms of properness.

Recall that a map between discrete spaces is proper if and only if the preimage of every finite set is finite.
\begin{proposition}\label{prop: free semigroup-loxodromic}
Let $\Gamma$ be a word-hyperbolic group.
Let $F\sbs \Gamma$ be a finite subset of cardinality $\#F \geqslant 5$.
Assume that for each pair of distinct elements $\gamma_1 \neq \gamma_2 \in F$, the map $(n,m) \in \ZZ^2 \mapsto \gamma_1^n \gamma_2^m \in \Gamma$ is proper.
Then there exist a constant $c>0$ and a map $\ell : F \to \ZZ_{>0}$ such that the following holds.

For any finite subset $A\sbs \Gamma$, there exist $\gamma_1\ne\gamma_2 \in F$ and a subset $A_0 \sbs A$ of cardinality $\#A_0 \geq c \#A - 1$, such that the elements
\[
a\gamma_i^{\ell(\gamma_i)},
\qquad (a,i)\in A_0\times\{1,2\},
\]
are pairwise distinct and freely generate a free sub-semigroup of $\Gamma$.
\end{proposition}
\begin{remark}
The assumption on $F$ says precisely that its elements are loxodromic and pairwise independent, that is, no two of them share a fixed point on the Gromov boundary.
We have phrased the assumption in terms of properness to keep the proof self-contained.
\end{remark}

The rest of this appendix is dedicated to the proof of Proposition \ref{prop: free semigroup-loxodromic}.
Let us first explain the idea behind the proof. 
If the set $A$ consists of elements of nearly equal word lengths, variants of this proposition have been established in several contexts; see, for example, \cite[Lemma 2.19]{Yang19}.
To deal with the case where $A$ consists of elements with uncontrolled word length, we introduce a ``prefix argument''.
To illustrate the idea, we can consider a simple model case: constructing a free sub-semigroup in a free semigroup.

Let $T$ be a free semigroup with the free generating set $\cT.$ Namely, $T$ consists of all finite words over the alphabet $\cT$ (including the empty word as the neutral element).
Assume that $\#\cT\geqslant 2.$
For elements $t_1,t_2\in T,$ $t_1$ is called a \textit{prefix} of $t_2$ if there exists $t_3\in T$ such that $t_2=t_1t_3.$
A sufficient criterion for a finite subset $A\sbs T$ to be a free generating set is that for every $t_1\ne t_2\in A,$ $t_1$ is not a prefix of $t_2.$
Given a finite subset $A\sbs T,$ we aim to modify $A$ to obtain another subset $A'$ that freely generates a free semigroup.
We consider the subset $A''\sbs A$ given by
\[A''\defeq\lb{a\in A: \forall t\in\cT,~ at\text{ is a prefix of some element in }A}. \]
For every $a\in A'',t\in\cT,$ let $\psi(a,t)$ be an element in $A$ of the least word length such that $at$ is a prefix of $\psi(a,t).$
A key observation is that the map $\psi:A''\times\cT\to A$ is injective.
Hence $\#A''\leqslant \# A/\#\cT\leqslant \#A/2.$
We define
\[A'\defeq\lb{at:a\in A,t\in \cT,~ at\text{ is not a prefix of any element in }A }.\]
Then $A'$ freely generates a free semigroup and $\#A'\geqslant \#(A\sm A'')\geqslant \#A/2.$

\subsection{Preliminaries}
\paragraph{Word-hyperbolic group.}
Let $\Gamma$ be a finitely generated group.
Denote by $e$ the neutral element.
We equip $\Gamma$ with a left-invariant word metric.
This requires choosing a finite symmetric generating set $\cG$.
The word length of an element $g \in \Gamma$ is 
\[
\abs{g} = \min\bigl\{\, k \geq 0 : \exists (g_1, \dotsc, g_k) \in \cG^k,\, g = g_1 \dotsc g_k \,\bigr\}.
\]
The word metric $d$ is defined by
\[
\forall x, y \in \Gamma,\quad d(x,y) = \abs{x^{-1} y}.
\]
A path in $\Gamma$ is a sequence $(x_0, \dotsc, x_n)$ of elements in $\Gamma$ such that for every $i \in \{1,\dotsc,n\}$, $d(x_{i-1}, x_i) = 1$.
If a path $(x_0, \dotsc, x_n)$ satisfies $d(x_0, x_n) = n$, then we say that it is a geodesic path and the underlying set $\{x_0, \dotsc, x_n\}$ is called a geodesic.
For $x, y \in \Gamma$, we will denote by $[x,y] \subset \Gamma$ a geodesic from $x$ to $y$.
Such a geodesic need not be unique.
In what follows, this non-uniqueness will not affect the arguments.
We will simply treat $[x,y]$ as an arbitrary geodesic connecting these two points.
For a subset $A \subset \Gamma$ and a radius $\delta \geq 0$, we write 
\[
\cN_\delta(A) = \left\{\, x \in \Gamma : d(x, A) \leq \delta \,\right\}
\]
to denote the closed $\delta$-neighborhood of $A$.
Slightly abusively, we also write $\cN_\delta(g) = \cN_\delta(\{g\})$ for an element $g \in \Gamma$.

Recall that a finitely generated group $\Gamma$ is said to be \emph{word-hyperbolic} if for any finite symmetric generating set $\cG$, there is some $\delta \geq 0$ such that the Rips condition holds:
\[\forall x, y, z \in \Gamma,\quad [x,z] \subset \cN_\delta\bigl( [x,y] \cup [y,z] \bigr).\]
Note that the choice of the generating set $\cG$ affects the parameter $\delta$.
In what follows, we fix once and for all a finite symmetric generating set $\cG$ and say that $\Gamma$ is \emph{$\delta$-hyperbolic}.
We also assume without loss of generality that $\delta\geq 1$.

We will make use of the following well-known fact.
\begin{lemma}
\label{lem: cyclic qc}
Let $\gamma$ be an element of a word-hyperbolic group $\Gamma$. 
Then $\pair{\gamma} := \lb{\, \gamma^n : n \in \ZZ \,}$ is quasi-convex, that is, there exists $r \geq 0$ such that $\bigcup_{n \in \ZZ} [e, \gamma^n] \subset \cN_r(\pair{\gamma})$.
\end{lemma}
\begin{proof}
If $\pair{\gamma}$ is finite, then it is quasi-convex.
Otherwise, by~\cite[Chapter III.$\Gamma$, Corollary 3.10]{BH99}, the map $n \in \ZZ \mapsto \gamma^n \in \Gamma$ is a quasi-isometric embedding. Then by~\cite[Chapter III.$\Gamma$, Corollary 3.6(2)]{BH99}, the subgroup $\pair{\gamma}$ is quasi-convex in $\Gamma$.
\end{proof}

\paragraph{Gromov product and $(\tau,D)$-chains}
Recall that the Gromov product between two points $x,y \in \Gamma$ with respect to a base point $z \in \Gamma$ is 
\[
\langle x, y \rangle_z = \frac{ d(x,z) + d(y,z) - d(x,y) }{2}.
\]
Clearly, it is invariant under left-translations:
\[
\forall g \in \Gamma,\quad \langle gx, gy \rangle_{gz} = \langle x, y \rangle_z.
\]
Informally, for $g, h \in \Gamma$, the product $\langle g^{-1}, h \rangle_e = \langle e, gh \rangle_g$ reflects how many letters get cancelled when we multiply $g$ by $h$.

\begin{definition}\label{def: tau,D-chain}
Let $\tau \geq 0$ and $D \geq 1$ be parameters.
A sequence $(x_0,\dotsc,x_n)$ of points in $\Gamma$ is said to be a \emph{$(\tau,D)$-chain} if
\begin{itemize}
	\item $\left\langle x_{i-1},x_{i+1}\right\rangle_{x_i}\leqslant \tau$ for every $1\leqslant i\leqslant n-1,$ and
	\item $d(x_{i-1},x_i)\geqslant D$ for every $1\leqslant i\leqslant n.$
\end{itemize}
\end{definition}

We will need the following well-known lemma.
\begin{lemma}
\label{lem: tau,D-chain}
Given $\delta \geq 1$ and $\tau \geq 0$, there exist constants $D=D(\delta, \tau) > 0$ and $C_0=C_0(\delta, \tau)>0$ such that for every $(\tau, D)$-chain $(x_0, \dotsc, x_n)$ in a $\delta$-hyperbolic group, we have
\[
\lb{x_i : 0 \leq i \leq n} \subset \cN_{C_0}\bigl([x_0, x_n]\bigr).
\]
\end{lemma}
This is essentially contained in \cite[Chapitre 5, Th\'eor\`eme 4.21]{GhysdelaHarpe}; see also \cite[Lemma 3.8]{Gou22}, which gives a shorter proof.
\begin{proof}
Take $D=2\tau+16\delta+1$.
By~\cite[Lemma 3.8]{Gou22},\footnote{Gou\"ezel uses
the Gromov-product definition of hyperbolicity.
Our Rips condition implies this definition with constant
$4\delta$, since we have normalized $\delta\ge1$.}
we have $\pair{x_0,x_n}_{x_i}\le\tau+8\delta$ for every $i$.
The elementary distance estimate for vertex geodesics gives
\[
d(x_i,[x_0,x_n])
\leq \pair{x_0,x_n}_{x_i}+2\delta+1
\leq \tau+10\delta+1.
\]
Thus we may take $C_0=\tau+10\delta+1$.
\end{proof}

\subsection{Constructing free subsemigroups}
In the model case above, we illustrated the use of prefixes in a free semigroup.
For the general setting of word-hyperbolic groups, we need to relax the notion of prefixes.
\begin{definition}
Let $C, M \geqslant 0$ be parameters and let $a,\gamma,x \in \Gamma$.
We call the pair $(a,\gamma)$ a \emph{$(C,M)$-prefix} of $x$ if there exists $y \in \cN_M(x)$ so that
\[
 \{a, a\gamma\} \subset \cN_C\bigl([e, y]\bigr).
\]
\end{definition}

Note that in a free semigroup, the prefix relation restricted to the set of all prefixes of a given element is a linear order.
The following lemma is an analogue for $(C,M)$-prefixes.
\begin{lemma}
\label{lem:prefix-linear}
Let $\Gamma$ be a $\delta$-hyperbolic group.
Let $C, M, \tau, L$ be parameters satisfying 
\[
M \geq C + \delta \quad \text{and} \quad L > 2 (C + M + \tau) + \delta.
\]

Let $a,a',\gamma,\gamma',b \in \Gamma$.
Assume
\begin{enumerate}
\item both $(a,\gamma)$ and $(a',\gamma')$ are $(C,M)$-prefixes of $b$,
\item $\abs{\gamma'} \leq 2L$, $\abs{\gamma} \geq L$ and $\pair{a^{-1}, \gamma}_{e} \leq \tau$,
\item $\abs{a'} \leq \abs{a}$ and $d(a,a') \geq 5L$.
\end{enumerate}
Then $(a',\gamma')$ is a $(C,M)$-prefix of $a$ and $\abs{a} < \abs{b}$.
\end{lemma}

\begin{proof}
By the definition, there are $y, y' \in \cN_M(b)$ such that $\{a , a \gamma\} \subset \cN_C\bigl([e, y]\bigr)$ and $\{a' , a' \gamma'\} \subset \cN_C\bigl([e, y']\bigr)$.
Let $u \in [e,y]$ such that $d(a,u) \leq C$.
Let $v \in [e,y]$ such that $d(a\gamma, v) \leq C$.

\medskip
\noindent\textbf{Step 1:} We show that $\abs{u} < \abs{y} - (2M + \delta)$.

Indeed, assumption (2) implies
\[
\abs{a} + \abs{\gamma} - \abs{a \gamma} \leq 2 \tau.
\]
It follows that
\[
\abs{a \gamma} - \abs{a} \geq L - 2 \tau,
\]
and then by the triangle inequality
\[
\abs{v} - \abs{u} \geq L - 2 \tau - 2C.
\]
Since $\abs{v} \leq \abs{y}$, the conclusion follows from the lower bound on $L$.

\medskip
\noindent\textbf{Step 2:} Find a point $z \in [e,y']$ with $d(z,a) \leq M$.

By Step 1, we have $d(u,y) > d(y,y')+\delta$ hence $d(u, [y,y']) > \delta$.
But hyperbolicity asserts that $u \in [e,y] \subset \cN_\delta([e,y']) \cup \cN_\delta([y,y'])$.
Necessarily, $d(u,[e,y'])\le\delta$.
Let $z \in [e,y']$ be such that $d(u,z) \leq \delta$.
Then $d(z,a) \leq C + \delta \leq M$.

\medskip
\noindent\textbf{Step 3:} We show that for any $g \in \cN_C([e,y'])$ with $\abs{g} \leq \abs{a} + 2L$ and $d(g,a) \geq 3 L$, we must have\footnote{To be rigorous, we choose a geodesic $[e,z]$ inside the choice of $[e,y']$.}
$d(g , [e,z])  \leq C$.

Indeed, let $w \in [e,y']$ such that $d(g, w) \leq C$.
It suffices to show that $\abs{w} \leq \abs{z}$.
Assume for a contradiction that $\abs{w} > \abs{z}$.
Then $\abs{w} = \abs{z} + d(w,z)$.
But the triangle inequality gives
\[
d(w,z) \geq d(g,a) - d(g, w) - d(a,z) \geq 3 L - C - M,
\]
\[
\abs{w} \leq \abs{g} + C \leq \abs{a} + 2L + C \quad \text{and} \quad \abs{z} \geq \abs{a} - d(a,z) \geq \abs{a} - M.
\]
The combination of these inequalities contradicts the assumption on $L$.

\begin{figure}[!ht]
	\centering
	\begin{tikzpicture}[
  x=1cm,y=1cm,
  font=\small,
  line cap=round,line join=round,
  geodesic/.style={draw=black,line width=.75pt},
  projection/.style={draw=black!65,densely dotted,line width=.7pt},
  endpointlink/.style={draw=black!60,dashed,line width=.65pt},
  point/.style={circle,fill=black,inner sep=1.5pt},
  every node/.style={inner sep=2pt}
]
\coordinate (e) at (0,2.5);
\coordinate (p) at (2,2.1);
\coordinate (q) at (3.8,2.1);
\coordinate (z) at (5.8,2.1);
\coordinate (u) at (5.8,3.05);
\coordinate (v) at (8.65,3.5);
\coordinate (y) at (10.3,4.55);
\coordinate (yp) at (10.3,1.1);
\coordinate (b) at (11.5,2.8);
\coordinate (a) at (5.8,4.15);
\coordinate (ag) at (8.4,4.35);
\coordinate (ap) at (2,1.25);
\coordinate (apg) at (3.8,1.25);

\draw[black!13,line width=12pt]
  (e) .. controls (.7,2.12) and (1.35,2.1) .. (p)
  -- (q) -- (z);

\draw[geodesic]
  (e) .. controls (1,2.98) and (3.8,3.05) .. (u)
  .. controls (7,3.05) and (8,3.1) .. (v)
  .. controls (9.45,3.9) and (9.6,4.3) .. (y);
\draw[geodesic]
  (e) .. controls (.7,2.12) and (1.35,2.1) .. (p)
  -- (q) -- (z)
  .. controls (8.6,2.1) and (9.25,1.8) .. (yp);
\draw[geodesic,black!60]
  (y) .. controls (9.4,3.8) and (9.4,1.8) .. (yp);

\draw[projection] (a) -- node[right] {$\leq C$} (u);
\draw[projection] (ag) -- node[left] {$\leq C$} (v);
\draw[projection] (ap) -- node[left] {$\leq C$} (p);
\draw[projection] (apg) -- node[right] {$\leq C$} (q);
\draw[projection] (u) -- node[right] {$\leq\delta$} (z);

\draw[endpointlink] (y) -- node[right] {$\leq M$} (b);
\draw[endpointlink] (yp) -- node[right] {$\leq M$} (b);

\node[point,label=left:{$e$}] at (e) {};
\node[point,label=above:{$a$}] at (a) {};
\node[point,label=above:{$a\gamma$}] at (ag) {};
\node[point,label=below:{$a'$}] at (ap) {};
\node[point,label=below:{$a'\gamma'$}] at (apg) {};
\node[point,label=above left:{$u$}] at (u) {};
\node[point,label=below right:{$v$}] at (v) {};
\node[point,label=below:{$z$}] at (z) {};
\node[point,label=above:{$y$}] at (y) {};
\node[point,label=below:{$y'$}] at (yp) {};
\node[point,label=right:{$b$}] at (b) {};
\fill (p) circle (1pt);
\fill (q) circle (1pt);

\draw[<->, >=Stealth, dashed, black!65, line width=.6pt,
      shorten <=2pt, shorten >=2pt]
  (a) to[bend right=55]
  node[pos=.25,left, ]
  {$\leq M$} (z);

\node[fill=white,text=black!65] at (3.25,2.1) {$[e,z]$};

\end{tikzpicture}
\stepcounter{theorem}
\caption{Two well-separated coarse prefixes of $b$.}\label{fg: prefix_relation}
\end{figure}

\medskip
\noindent\textbf{Conclusion:} The last step applies to both $a'$ and $a'\gamma'$ in view of the assumption (3) and the triangle inequality.
It follows that $(a', \gamma')$ is a $(C,M)$-prefix of $a$.
See Figure \ref{fg: prefix_relation}.
Finally, $\len{a} < \len{b}$ follows from Step 1 and triangle inequalities.
\end{proof}

In a free semigroup, if a subset $A$ satisfies the prefix-free condition and the following non-cancellation condition: $\forall a, b \in A$, $\pair{a^{-1},b}_e = 0$, then it freely generates a free subsemigroup.
The following lemma is an analogue relaxing the non-cancellation condition.

\begin{lemma}
\label{lem:agamma-free}
Let $\Gamma$ be a $\delta$-hyperbolic group.
Given a parameter $\tau>0$, there is $C_0 = C_0(\delta, \tau)$ such that for all $C \geq C_0$, there is  $L_0 = L_0(\delta,\tau, C)$ such that the following holds for all $L \geq L_0$.
Let $A, F \subset \Gamma$ be subsets satisfying:
\begin{enumerate}[\indent(i)]
\item \label{it:agf1} $A$ is $5L$-separated and $F$ is $6C$-separated.
\item \label{it:agf2} For every $a \in A$, $\abs{a} \geq L$ and for every $\gamma \in F$, $\abs{\gamma} \in [L,L + C]$.
\item \label{it:agf3} For every $a\in A$, $\gamma \in F$ we have $\pair{a^{-1}, \gamma }_e \leqslant \tau$ and $\pair{a, \gamma^{-1}}_e\leqslant \tau.$
\item \label{it:agf4} For any $a\ne  a'\in A$ and any $\gamma \in F$, the pair $(a, \gamma)$ is not a $(C, C + \delta)$-prefix of $a'$.
\end{enumerate}
Then the map $(a,\gamma)\mapsto a\gamma$ from $A\times F$ to $\Gamma$ is injective, and its image $AF$ freely generates a free sub-semigroup of $\Gamma$.
\end{lemma}

\begin{proof}
Let $D=D(\delta,\tau)$ and $C_0=C_0(\delta, \tau)$ be the constants given by Lemma \ref{lem: tau,D-chain}.
Let $C \geq C_0$ and $L \geq L_0 \defeq \max\{D, 4C + 5\delta + 2\tau + 1\}$.

Assume that 
\[
a_1 \gamma_1 \dotsm a_m \gamma_m=a'_1 \gamma'_1 \dotsm a'_n \gamma'_n
\]
for some integers $m,n \geq 0$ and some $a_1, \dotsc, a_m, a'_1, \dotsc, a'_n \in A$ and $\gamma_1, \dotsc, \gamma_m, \gamma'_1, \dotsc, \gamma'_n \in F$.
By induction, it suffices to prove the following in order to show that $AF$ freely generates a free semigroup:
\begin{enumerate}[\indent(a)]
\item If $m=0$ then $n=0.$
\item If $m,n>0$ then $ a_1=a'_1$ and $\gamma_1=\gamma'_1$.
\end{enumerate}
By conditions \eqref{it:agf2} and \eqref{it:agf3}, the sequence $\lb{x_i}_{0\leqslant i\leqslant 2m}$ defined inductively by
\[
x_0 = e\quad \text{and} \quad \forall k \in \{0,\dotsc,m-1\},\; x_{2k+1} = x_{2k} a_{k+1} \text{ and } x_{2k + 2}= x_{2k + 1} \gamma_{k+1}
\]
is a $(\tau, D)$-chain.
The same is true for a sequence $\lb{x'_i}_{0 \leqslant i \leqslant 2n}$ defined similarly.

By assumption, $x_{2m} = x'_{2n}$. Using Lemma~\ref{lem: tau,D-chain} twice, we have
\begin{equation}
\label{eq:xiinNC}
\lb{x_i}_{0\leqslant i\leqslant 2m} \cup
\lb{x'_i}_{0\leqslant i\leqslant 2n} \sbs \cN_C\bigl([e, x_{2m}]\bigr).
\end{equation}

Then (a) follows directly.
If $m=0$ and $n \geq 1$, then we would have $a'_1 = x'_1 \in \cN_C(e)$, contradicting the assumption \eqref{it:agf2}.

We now prove claim (b).
First, we show that $a_1 = a'_1$.
Without loss of generality, we can assume $\len{a_1'} \leqslant \len{a_1}$.
By \eqref{eq:xiinNC}, both $(a_1, \gamma_1)$ and $(a'_1, \gamma'_1)$ are $(C,0)$-prefixes hence $(C, C + \delta)$-prefixes of $x_{2m} = x'_{2n}$.
By Lemma~\ref{lem:prefix-linear}, either $d(a_1,a_1') < 5L$ or $(a'_1, \gamma'_1)$ is a $(C,C+\delta)$-prefix of $a_1$.
This leads to $a_1 = a'_1$ on account of the conditions \eqref{it:agf1} and \eqref{it:agf4}.
See Figure \ref{fg: first_a_factor} below.

\begin{figure}[!ht]
	\centering
	
	\begin{tikzpicture}[
  x=1cm,y=1cm,font=\small,line cap=round,line join=round,
  vertex/.style={circle,fill=black,inner sep=1.25pt},
  geodesic/.style={line width=.65pt},
  firstchain/.style={draw=black!65,line width=.55pt},
  secondchain/.style={draw=black!65,line width=.55pt,dashed},
  proximity/.style={draw=black!65,densely dotted,line width=.65pt}
]

  \coordinate (e) at (0,0);
  \coordinate (w) at (11.4,0);
  \coordinate (a) at (3.1,.60);
  \coordinate (b) at (5.5,.70);
  \coordinate (ap) at (2.5,-.55);
  \coordinate (bp) at (4.8,-.70);
  \draw[geodesic] (e) -- (w);
  \draw[firstchain] (e) to[out=35,in=185] (a)
    to[out=5,in=180] (b) to[out=0,in=175] (7.5,.58);
  \draw[firstchain] (8.9,.50) to[out=-5,in=155] (w);
  \draw[secondchain] (e) to[out=-35,in=185] (ap)
    to[out=-5,in=175] (bp) to[out=-5,in=185] (7.5,-.62);
  \draw[secondchain] (8.9,-.50) to[out=5,in=205] (w);
  \node at (8.2,.57) {$\cdots$};
  \node at (8.2,-.60) {$\cdots$};

  \draw[proximity] (a) -- (3.1,0);
  \draw[proximity] (b) -- (5.5,0)
    node[midway,right] {$\leq C$};
  \draw[proximity] (ap) -- (2.5,0);
  \draw[proximity] (bp) -- (4.8,0)
    node[midway,right] {$\leq C$};
  \node[vertex,label=left:{$e$}] at (e) {};
  \node[vertex] at (w) {};
  \node[anchor=north east] at (11.4,-.80)
    {$w=x_{2m}=x'_{2n}$};
  \node[vertex,label=above:{$a_1=x_1$}] at (a) {};
  \node[vertex,label=above:{$x_2=a_1\gamma_1$}] at (b) {};
  \node[vertex,label=below:{$a'_1=x'_1$}] at (ap) {};
  \node[vertex,label=below:{$x'_2=a'_1\gamma'_1$}] at (bp) {};
\end{tikzpicture}
\stepcounter{theorem}
\caption{Recovering the first $a$-factor.}\label{fg: first_a_factor}
\end{figure}

It remains to show that $\gamma_1 = \gamma'_1.$
See Figure \ref{fg: first_gamma_factor}.
By the above, we have $x_1=x'_1$ and $x_{2m} = x'_{2n}$. 
Considering the $(\tau, D)$-chains $\lb{x_i}_{1\leqslant i\leqslant 2m}$ and $\lb{x'_i}_{1\leqslant i\leqslant 2n}$, we deduce as above that
\begin{equation*}
\lb{x_i}_{1\leqslant i\leqslant 2m} \cup
\lb{x'_i}_{1\leqslant i\leqslant 2n} \subset \cN_C\bigl([x_1, x_{2m}]\bigr).
\end{equation*}

\begin{figure}[!ht]
\centering
\begin{tikzpicture}[
  x=1cm,y=1cm,font=\small,line cap=round,line join=round,
  vertex/.style={circle,fill=black,inner sep=1.25pt},
  geodesic/.style={line width=.65pt},
  firstchain/.style={draw=black!65,line width=.55pt},
  secondchain/.style={draw=black!65,line width=.55pt,dashed},
  proximity/.style={draw=black!65,densely dotted,line width=.65pt}
]
  \coordinate (o) at (0,-4.0);
  \coordinate (v) at (4.6,-3.0);
  \coordinate (p) at (4.6,-4.0);
  \coordinate (pp) at (6.8,-4.0);
  \coordinate (vp) at (6.8,-5.0);
  \coordinate (z) at (11.4,-4.0);
  \draw[geodesic] (o) -- (z);
  \draw[firstchain] (o) .. controls (1.1,-3.05) and (2.6,-2.65) .. (v);
  \draw[secondchain] (o) .. controls (1.7,-5.15) and (4.7,-5.60) .. (vp);
  \node at (2.05,-2.70) {$|\gamma_1|$};
  \node at (2.25,-5.50) {$|\gamma'_1|$};
  \draw[proximity] (v) -- (p)
    node[midway,left] {$\leq C$};
  \draw[proximity] (pp) -- (vp)
    node[midway,right] {$\leq C$};
  \draw[line width=1.1pt] (p) -- (pp)
    node[midway,above=3pt] {$\leq 3C$};
  \node[vertex,label=left:{$a_1=a'_1$}] at (o) {};
  \node[vertex,label=above:{$x_2=a_1\gamma_1$}] at (v) {};
  \node[vertex,label=below left:{$y$}] at (p) {};
  \node[vertex,label=above right:{$y'$}] at (pp) {};
  \node[vertex,label=below:{$x'_2=a_1\gamma'_1$}] at (vp) {};
  \node[vertex,label=right:{$w$}] at (z) {};

\end{tikzpicture}
\stepcounter{theorem}
\caption{Recovering the first $\gamma$-factor.}\label{fg: first_gamma_factor}

\end{figure}
In particular, there are points $y,y' \in [x_1, x_{2m}]$ such that $d(x_2,y)\leq C$ and $d(x'_2,y') \leq C$.
By the triangle inequality and then the condition~\eqref{it:agf2}, 
\[
d(x_1,y) \in [d(x_1,x_2) - C, d(x_1,x_2) + C] \subset [L - C, L + 2C]
\]
and the same holds for $d(x_1,y')$.
It follows that $d(y,y') = \bigl\lvert d(x_1,y) - d(x_1,y') \bigr\rvert \leq 3C$ and then
\[
d(\gamma_1, \gamma'_1) = d(x_1\gamma_1, x_1\gamma'_1) = d(x_2,x'_2) \leq d(x_2,y) + d(y, y') + d(y', x'_2) \leq 5C.
\]
By the assumption \eqref{it:agf1}, we must have $\gamma_1 = \gamma'_1$.
\end{proof}

In a free semigroup, if at least $\tau$ letters cancel when multiplying the same word on the left of both $\gamma_1$ and $\gamma_2$, then the first $\tau$ letters of $\gamma_1$ and of $\gamma_2$ must coincide.
The following lemma is an analogue.
\begin{lemma}
\label{lem: common pf}
Let $\tau\geqslant \delta +1.$
Let $a, \gamma_1, \gamma_2$ be elements of a $\delta$-hyperbolic group.
If $\pair{a, \gamma_i}_e > \tau$ for both $i \in \{1,2\}$,
then $\cN_\delta([e, \gamma_1]) \cap \cN_\delta([e, \gamma_2]) \not\subset \cN_{\tau - \delta-1}(e)$.
\end{lemma}
\begin{proof}
Let $x \in [e, a]$ be such that $\abs{x} = \lfloor\tau - \delta\rfloor$.
Such a point exists because $\abs{a} \geq \pair{a, \gamma_1}_e > \tau$.
Moreover, $|x|=\lfloor\tau-\delta\rfloor>\tau-\delta-1$.
It suffices to show that $x \in \cN_\delta\bigl( [e, \gamma_i]\bigr)$ for both $i \in \{1,2\}$.
Indeed, otherwise, by the hyperbolicity, we would have $x \in \cN_\delta\bigl([a,\gamma_i]\bigr)$ for some $i$.
Let $y \in [a, \gamma_i]$ be such that $d(x,y) \leq \delta$.
Then, by the triangle inequality and the fact that $[a, \gamma_i]$ is a geodesic,
\begin{equation*}
\abs{\gamma_i} \leq \abs{x} + d(x,y) + d(y, \gamma_i) \leq \tau + d(a, \gamma_i) - d(a, y).
\end{equation*}
By the fact that $[e,a]$ is a geodesic and the triangle inequality,
\begin{equation*}
\abs{a} = \abs{x} + d(a,x) \leq \abs{x} + d(a, y) + d(x, y) \leq \tau + d(a,y).
\end{equation*}
Summing up these two inequalities, we obtain $\abs{\gamma_i} + \abs{a} \leq 2 \tau + d(a, \gamma_i)$, contradicting the assumption $\pair{a,\gamma_i}_e > \tau$.
\end{proof}

\begin{proof}[Proof of Proposition \ref{prop: free semigroup-loxodromic}]
Let $\Gamma$ and $F \subset \Gamma$ be as in the statement with $\Gamma$ being $\delta$-hyperbolic for some $\delta \geq 1$.
Without loss of generality we may assume $\# F = 5$.

We claim that there is a radius $R \geq 1$ depending on $F$ such that
\begin{equation}
\label{eq:bddcap}
\forall \gamma_1 \neq \gamma_2 \in F,\, \forall n, m \in \ZZ,\quad
\cN_\delta([e, \gamma_1^n]) \cap \cN_\delta([e, \gamma_2^m]) \subset \cN_{R}(e).
\end{equation}
Indeed, otherwise, there would be $\gamma_1 \neq \gamma_2 \in F$ such that the intersection
\(
\cN_\delta\left(\bigcup_{n \in \ZZ} [e,\gamma_1^n]\right) \cap
\cN_\delta\left(\bigcup_{m \in \ZZ} [e,\gamma_2^m]\right)
\)
is unbounded.
In view of Lemma~\ref{lem: cyclic qc}, there is some $r = r(F) \geq 1$ such that this would imply that
\(
\cN_{r + \delta}\left(\pair{\gamma_1}\right) \cap
\cN_{r + \delta}\left(\pair{\gamma_2}\right)
\)
is unbounded.
This yields an unbounded set of pairs $(n,m) \in \ZZ^2$ satisfying $d(\gamma_1^n, \gamma_2^m) \leq 2(r + \delta)$, contradicting the properness assumption on $F$.

Now, set $\tau = R + \delta+1$.
Let $C_0(\delta, \tau)$ and $L_0(\delta,\tau,C)$ be as in Lemma~\ref{lem:agamma-free}.
Pick $C \geq C_0(\delta,\tau)$ large enough so that $C \geq \abs{\gamma}$ for every $\gamma \in F$.
This is to make sure that for every $L \geq C$, there is a function $\ell : F \to \ZZ_{>0}$ such that for each $\gamma \in F$, we have $\bigl\lvert\gamma^{\ell(\gamma)}\bigr\rvert \in [L, L + C]$.
We can pick $L$ large enough so that not only $L \geq L_0(\delta,\tau,C)$, but also
\begin{equation}
\label{eq:LgtCMtau}
L > 2(C + (C + \delta) + \tau) + \delta.
\end{equation}
Making $L$ even larger and using the properness assumption, we can ensure
\begin{equation}
\label{eq:Tseperated}
	\forall \gamma_1 \neq \gamma_2 \in F,\quad d\left(\gamma_1^{\ell(\gamma_1)}, \gamma_2^{\ell(\gamma_2)}\right) > T\defeq  9C+6\delta+2\tau,
\end{equation}
Set $F_1 \defeq \{\, \gamma^{\ell(\gamma)} : \gamma \in F \,\}$.

Given a finite subset $A \subset \Gamma$, we will successively trim the sets $A$ and $F_1$ in order to meet all the requirements~\eqref{it:agf1}-\eqref{it:agf4} of Lemma~\ref{lem:agamma-free}.
Start by taking $A_1$ to be a maximal $5L$-separated subset of $A \setminus \cN_L(e)$ so that \eqref{it:agf1} and \eqref{it:agf2} are satisfied for $A_1$ and $F_1$.
Note that $\# A_1 \geq c_1  \# A - 1$ for some $c_1 = c_1(\Gamma, L) > 0$.

Next, by Lemma~\ref{lem: common pf} and \eqref{eq:bddcap}, for each $a \in A_1$, there are at least $\# F_1 - 2 = 3$ elements $\gamma\in F_1$ such that both $\pair{a^{-1}, \gamma}_e \leq \tau$ and $\pair{a,\gamma^{-1}}_e \leq \tau$ hold.
By the pigeonhole principle, there exist subsets $F_2 \sbs F_1$ and $A_2\sbs A_1$  of cardinalities $\#F_2 = 3$ and $\# A_2 \geqslant {5 \choose 3}^{-1} \# A_1$ for which the condition~\eqref{it:agf3} holds.

Finally, we need to eliminate prefix relations in $A_2 \times F_2$.
To this end, set $M = C + \delta$ and consider the set $P \sbs A_2 \times F_2$ defined by
\[
P \defeq\lb{\,(a,\gamma) \in A_2 \times F_2 : \exists b \in A_2 \text{ such that } (a,\gamma) \text{ is a } (C,M)\text{-prefix of } b\,}.
\]
For $(a,\gamma) \in P$, choose an element $\psi(a,\gamma) \in A_2$ of minimal length such that $(a,\gamma)$ is a $(C,M)$-prefix of $\psi(a,\gamma)$.
\begin{lemma}
	$\psi: P\to  A_2$ is injective.
\end{lemma}
\begin{proof}
For $(a,\gamma),(a',\gamma')\in P$ with $a\ne a',$ the fact $\psi(a,\gamma)\ne\psi(a',\gamma')$ follows from the $5L$-separation of $A_2$, Lemma~\ref{lem:prefix-linear}, $L\leqslant |\gamma|,|\gamma'|\leqslant L+C\leqslant 2L$ and the minimality in the definition of $\psi.$

We now show that $\psi(a,\gamma)\ne\psi(a,\gamma')$ for $\gamma\ne\gamma'.$
Otherwise, there exist $y,y'$ with $d(y,y')\leqslant 2M$ such that $a,a\gamma\in \cN_C([e,y])$ and $a,a\gamma'\in \cN_C([e,y']).$
Let $u\in [e,y]$ such that $d(u,a\gamma)\leqslant C.$
By the hyperbolicity, $u\in \cN_{\delta}([e,y']\cup[y,y'])$ and hence $d(u,[e,y'])\leqslant \delta+d(y,y')\leqslant\delta+2M.$
Then $d(a\gamma,[e,y'])\leqslant C+d(u,[e,y'])\leqslant \delta+C+2M.$

By our choice, $\pair{a^{-1},\gamma}_e,\pair{a^{-1},\gamma'}_e\leqslant \tau.$
Therefore, $|a\gamma|,|a\gamma'|\in [|a|+L-2\tau,|a|+L+C].$
Set $v,v'\in [e,y']$ with $d(a\gamma,v)\leqslant \delta+C+2M$ and $d(a\gamma',v')\leqslant C$.
By the triangle inequality,
\[d( v,v')\leqslant \abs{\abs{a\gamma}-\abs{a\gamma'} }+C+(\delta+C+2M)\leqslant \delta+ 3C+2M+2\tau=5C+3\delta+2\tau.\]
Again by the triangle inequality,
\[d(\gamma,\gamma')= d(a\gamma,a\gamma' )\leqslant d(v,v')+C+(\delta+C+2M)\leqslant 9C+6\delta+2\tau=T,\]
which violates our $T$-separated assumption \eqref{eq:Tseperated}.
\end{proof}

It follows that $\#P\leqslant\# A_2.$
Let $A_3 \sbs A_2$ be the subset consisting of elements $a\in A_2$ for which $\#\{\, \gamma \in F_2 : (a,\gamma) \in P \,\} \leq 1.$
Then $2 \# (A_2 \sm A_3) \leq \#P \leq \# A_2$ and hence $\# A_3 \geq \# A_2 / 2.$
By the pigeonhole principle, there exists a $\gamma_0 \in F_2$ and $A_4 \subset A_3$ such that  $(a,\gamma) \notin P$ for every $(a, \gamma)\in A_4\times F_3,$ where $F_3 \defeq F_2 \sm \{\gamma_0\},$ and  $\# A_4 \geqslant \# A_3 / 3$.
In particular, the condition~\eqref{it:agf4} holds for $A_4$ and $F_3$.

To summarize, we have constructed $A_4 \sbs A$ and $F_3 \sbs F_1$ of cardinalities $\# A_4 \geq c \# A - 1$ with $c = c(F) > 0$ and $\# F_3 = 2$ and such that all the conditions \eqref{it:agf1}-\eqref{it:agf4} hold for $A_4$ and $F_3$ in place of $A$ and $F$.
The conclusion now follows from Lemma \ref{lem:agamma-free}.
\end{proof}

\begin{small}

\addcontentsline{toc}{section}{References}
\bibliographystyle{abbrv}
\bibliography{ref.bib}

\vspace{0.3cm}

\noindent Weikun He\\
State Key Laboratory of Mathematical Sciences, Academy of Mathematics and Systems Science, Chinese Academy of Sciences\\
Beijing 100190, China\\
heweikun@amss.ac.cn

\vspace{0.3cm}

\noindent Yuxiang Jiao\\
Beijing International Center for Mathematical Research, Peking University\\
No.5 Yiheyuan Road, Haidian District, Beijing, China, 100871\\
ajorda@pku.edu.cn

\vspace{0.3cm}

\noindent Disheng Xu\\
Department of Science, Great Bay University\\
Building A5, Songshanhu international community, Dongguan, Guangdong, China, 523000\\
xudisheng@gbu.edu.cn

\end{small}

\end{document}